\documentclass[reqno,english]{amsart}
\usepackage{amsfonts,amsmath,latexsym,verbatim,amscd,mathrsfs,color,array}
\usepackage{amsmath,amssymb,amsthm,amsfonts}
\usepackage{graphicx,color}

\def\texorpdfstring#1#2{#1}
\usepackage[colorlinks=true,citecolor=blue]{hyperref}

\newcommand{\ass}{\quad\mbox{as}\quad}
\usepackage{graphicx}

\newcommand{\EE}{{\mathcal E}  }
\newcommand{\inn}{{\quad\hbox{in } }}
\newcommand{\onn}{{\quad\hbox{on } }}
\newcommand{\ttt}{\tilde }
\newcommand{\TT}{{\mathcal T}  }

\newcommand{\nn}{ {\nabla}  }

\newcommand{\pp}{ {\partial} }

\newcommand{\vp}{\varphi}

\newcommand{\RR}{{{\mathcal R}}}

\newcommand{\N}{\mathbb{N}}
\newcommand{\C}{{\mathbb C}}
\newcommand{\R} {\mathbb R}
\newcommand{\Z} {\mathbb Z}
\newcommand{\cuad}{{\sqcap\kern-.68em\sqcup}}

\newcommand{\DD}{{\mathcal D}}
\newcommand{\Rem}{{\mathcal R}_0}

\newcommand{\KK}{{\mathcal K}}

\newcommand{\foral}{\quad\mbox{for all}\quad}
\newcommand{\ve}{\varepsilon}

\newcommand{\be}{\begin{equation}}
\newcommand{\ee}{\end{equation}}

\newcommand{\la}{\lambda}

\newcommand{\equ}[1]{(\ref{#1})}

\newcommand{\p}{{p}}
\renewcommand{\Re}{\mathop{\rm Re}}
\renewcommand{\Im}{\mathop{\rm Im}}

\renewcommand{\div}{\mathop{\rm div}}
\newcommand{\curl}{\mathop{\rm curl}}

\newtheorem{lemma}{Lemma}[section]
\newtheorem{prop}{Proposition}[section]
\newtheorem{theorem}{Theorem}
\newtheorem{corollary}{Corollary}[section]
\newtheorem{remark}{Remark}[section]
\newtheorem{ch}{Check}
\newcommand{\bremark}{\begin{remark} \em}
\newcommand{\eremark}{\end{remark} }

\long\def\hide#1{}

\long\def\noanot#1{}

\definecolor{redd}{RGB}{200,0,0}
\definecolor{mo1}{rgb}{0.8,0,0.8}
\definecolor{green1}{rgb}{0.1,0.7,0.1}
\definecolor{g2}{rgb}{0,0.5,0}

\def\cb{\color{black}}
\long\def\elim#1{{\color{red} ELIMINAR\\ #1}}

\long\def\elim#1{}

\numberwithin{equation}{section}

\title[Singularity formation in the two-dimensional harmonic map flow]{Singularity formation for the two-dimensional  harmonic map flow into $S^2$}

\author[J. Davila]{Juan Davila}
\address{\noindent
Instituto de Matem\'aticas, Universidad de Antioquia, Calle 67, No. 53--108, Medell\'\i n, Colombia,
and  Departamento de
Ingenier\'{\i}a  Matem\'atica-CMM   Universidad de Chile,
Santiago 837-0456, Chile}
\email{jdavila@dim.uchile.cl}

\author[M. del Pino]{Manuel del Pino}
\address{\noindent   Department of Mathematical Sciences University of Bath,
Bath BA2 7AY, United Kingdom \\
and  Departamento de
Ingenier\'{\i}a  Matem\'atica-CMM   Universidad de Chile,
Santiago 837-0456, Chile}
\email{m.delpino@bath.ac.uk}

\author[J. Wei]{Juncheng Wei}
\address{\noindent
Department of Mathematics,
University of British Columbia, Vancouver, B.C., Canada, V6T 1Z2}
\email{jcwei@math.ubc.ca}

\begin{document}

\begin{abstract}
We construct finite time blow-up solutions to the 2-dimensional harmonic map flow into the sphere $S^2$,
\begin{align*} u_t & = \Delta u + |\nabla u|^2 u \quad \text{in } \Omega\times(0,T)
\\
u &= \varphi \quad \text{on } \partial \Omega\times(0,T)
\\
u(\cdot,0) &= u_0 \quad  \text{in } \Omega  ,
\end{align*}
where $\Omega$ is a bounded, smooth domain in $\mathbb{R}^2$, $u: \Omega\times(0,T)\to S^2$, $u_0:\bar\Omega \to S^2$ is smooth, and $\varphi = u_0\big|_{\partial\Omega}$. Given any $k$ points $q_1,\ldots, q_k$ in the domain, we find initial and boundary data so that the solution blows-up precisely at those points. The profile around each point is close to an asymptotically singular scaling of a 1-corrotational harmonic map.
We build a continuation after blow-up as a $H^1$-weak solution with a finite number of discontinuities in space-time by ``reverse bubbling'', which preserves the homotopy class of the solution after blow-up.
Furthermore, we prove the codimension one stability of the one point blow-up phenomenon.

\end{abstract}

\maketitle


\section{Introduction and main result}

Let $\Omega$ be a bounded domain  in $\R^2$ with smooth boundary $\pp\Omega$.  We denote by $S^2$ the standard 2-sphere.  We consider the {\em harmonic map flow}
for maps from $\Omega$ into $S^2$, given by the semilinear parabolic equation
\begin{align}
\label{har map flow0}
u_t = \Delta u + |\nabla u|^2 u \quad &\text{in } \Omega\times(0,T)\\
\label{har map flow01}
u = \vp  \quad &\text{on } \pp\Omega\times(0,T)\\
\label{har map flow02}
u(\cdot,0) = u_0  \quad & \text{in } \Omega
\end{align}
for a function $u:\Omega \times [0,T)\to S^2$. Here $u_0:\bar \Omega \to S^2$ is a given smooth map and  $\vp= u_0\big|_{\pp\Omega}$.  Local existence and uniqueness of a classical solution follows from the works \cite{chang,Ells,Struwe}. Equation \equ{har map flow0} formally corresponds to the negative $L^2$-gradient flow for the Dirichlet energy $\int_\Omega |\nabla u|^2 dx$. This energy is decreasing
along smooth solutions $u(x,t)$:
$$
\frac {\pp}{\pp t} \int_\Omega |\nabla u(\cdot, t)|^2  = -\int_\Omega |u_t(\cdot, t)|^2.
$$
Struwe \cite{Struwe} established the existence of an $H^1$-weak solution, where just for a finite number of points in space-time loss of regularity occurs.  This solution is unique within the class of weak solutions with decreasing energy, see Freire \cite{Freire} {\cb and also Lin-Wang \cite{Lin-Wang3.5} for another proof.}

\medskip
If $T>0$ designates the first instant at which smoothness is lost, we must have
$$
 \|\nabla u(\cdot, t)\|_\infty \, \to \,  +\infty \ass t\uparrow T.
$$
Several works have clarified the possible blow-up profiles as $t\uparrow T$.
The following fact follows from results by  Ding-Tian \cite{Ding-Tian}, Lin-Wang \cite{Lin-Wang1}, Qing \cite{Qing1},  Qing-Tian \cite{Qing-Tian}, Struwe \cite{Struwe}, Topping \cite{Topping2} and Wang \cite{Wang}:

\medskip
Along a sequence $t_n\to T$ and points $q_1,\ldots, q_k\in \Omega$, not necessarily distinct, $u(x,t_n)$ blow-up occurs at exactly those $k$ points in the form of {\em bubbling}.
More precisely, under some technical assumptions we have
\be\label{aa}
u(x,t_n)\, -\,   u_*(x)  -  \sum_{i=1}^k  [\, U_i \left ( \frac{ x-q_{i}^n} {\la_{i}^n} \right )   - U_i(\infty) \, ]  \to 0 \inn H^1(\Omega)
\ee
where $u_*\in H^1(\Omega)$,  $q_{i}^n\to q_i$,  $0<\la_{i}^n\to 0$,  satisfy for $i\ne j$,
$$
\frac{\la_i^n}{\la_j^n} + \frac{\la_j^n}{\la_i^n}  + \frac{|q_i^n - q_j^n|^2}{ \la_i^n\la_j^n}\to\ +\infty  .
$$
The $U_i$'s are entire, finite energy harmonic maps,
namely solutions $U :\R^2 \to S^2$ of the equation
\begin{align*}
\Delta U + |\nabla U|^2 U = 0 \inn \R^2, \quad \int_{\R^2} |\nabla U|^2 <+\infty .
\end{align*}
After stereographic projection, $U$ lifts to a smooth map in $S^2$, so that its value $U (\infty)$ is well-defined. It is known that $U$ is in correspondence with a complex rational function or its conjugate. Its energy corresponds to the absolute value of the degree of that map times the area of the unit sphere, and hence
\be\label{aa1}
\int_{\R^2} |\nabla U|^2 = 4\pi m, \quad m\in \N,
\ee
  see Topping \cite{Topping2}.

\medskip
In particular, $u(\cdot, t_n) \rightharpoonup u_*$ in $H^1(\Omega)$ and  for some positive integers $m_i$, we have
\be \label{aa2}
|\nabla u(\cdot, t_n)|^2
\ \rightharpoonup  \ |\nabla u_*|^2 +   \sum_{i=1}^k 4\pi m_i\,\delta_{q_i}
\ee
in the measures sense, were $\delta_q$ denotes the unit Dirac mass at $q$.

\medskip
Topping \cite{Topping3} estimated the blow-up rates as $\la_i^n = o ( (T-t_n)^{\frac 12})$ (also valid for more general targets), a fact that tells that the blow-up is of ``type II'', namely it does not occur at a self-similar rate.

\medskip
A decomposition similar to \equ{aa} holds if blow-up occurs in infinite time, $T=+\infty$.
In such a case one has the additional information that $u_*$ is a harmonic map, and the convergence in
 \equ{aa} also holds uniformly in $\Omega$ (the latter is called the ``no-neck property''), see Qing and Tian  \cite{Qing-Tian}.  Finer properties of the bubble-decomposition have been found by Topping \cite{Topping2}.

\medskip
A  {\em least energy} entire, non-trivial harmonic map is given by
\begin{equation}
\label{U00}
W(x)  =    \frac 1{1+ |x|^2} \left (\begin{matrix}    2x   \\  |x|^2 -1 \end{matrix}    \right ) , \quad   x\in \R^2 ,
\end{equation}
which satisfies
$$
\int_{\R^2} |\nabla W |^2 = 4\pi, \quad  W (\infty) = \left (\begin{matrix}   0 \\ 0\\   1 \end{matrix}    \right ).
$$
Very few examples are known of solutions, which exhibit the singularity formation phenomenon \equ{aa2}, and all of them concern single-point blow-up in radially symmetric {\em corrotational} classes. When $\Omega$ is a disk or the entire space, a $1$-corrotational solution of \equ{har map flow0} is one of the form
\begin{align}
\label{1corrot}
u(x,t)  =  \left (\begin{matrix}    e^{ i\theta} \sin v(r,t)   \\  \cos v(r,t) \end{matrix}    \right ) , \quad x= re^{i\theta}  .
\end{align}
Within this class, \equ{har map flow0} reduces to the scalar, radially symmetric problem
\be\label{11}
v_t = v_{rr} + \frac{v_r}r  -  \frac {\sin v\cos v} {r^2}.
\ee
We observe that the function
$$
w(r) = \pi - 2\arctan (r)
$$
is a steady state of \equ{11} which  corresponds precisely
to the harmonic map $W $ in \equ{U00}. Indeed,
$$
W (x) =  \left (\begin{matrix}    e^{ i\theta} \sin {w(r)}   \\  \cos {w(r)} \end{matrix}    \right ).
$$
Chang, Ding and Ye \cite{Chang-Ding-Ye} found the first example of a blow-up solution of problem  \equ{har map flow0}-\equ{har map flow02}  (which was previously conjectured not to exist).
{\cb They obtained the result in the 1-corrotational class in a disk $D$ by finding appropriate sub-super solutions to \eqref{11}.
Assuming that the initial energy satisfies $\int_D |\nabla u_0|^2<8\pi$, the decomposition \eqref{aa} implies that
\be\label{qq}
u(x,t) =   W \Bigl ( \frac x{\la(t)}\Bigr ) + u_* + o(1) ,
\ee
with $u_* \in H^1$, $o(1)\to 0$ in $H^1$-norm, and $0<\la(t)\to 0$ as $t\to T$. }
No information on the precise blow-up rate $\la(t)$ is obtained. Angenent, Hulshof and Matano \cite{ahm} estimated the blow-up rate of 1-corrotational maps as
$\la(t) = o(T-t)$. Using matched asymptotics formal analysis for problem \equ{11},
van den Berg, Hulshof and  King  \cite{bhk} demonstrated that this rate for 1-corrotational maps should generically be given by
\be
\la(t) \approx  \kappa \frac{ T-t}{|\log (T-t)|^2} ,
\label{rate0}\ee
for some $\kappa >0$. Raphael and Schweyer \cite{rs1} succeeded to rigorously construct an entire 1-corrotational solution with this blow-up rate.

\medskip

In this paper we deal with the general, nonsymmetric case in \equ{har map flow0}-\equ{har map flow02}. Our first result asserts that for any  given finite set of points of $\Omega$ and suitable initial and boundary values,
a solution with a simultaneous blow-up at those points exists, with a  profile resembling a translation and rotation of that in \equ{qq} around each bubbling point.

\medskip
To state our result, we observe that the functions
\[
U_{\la, q, Q} (x)\,  :=\,  Q W  \left (\frac {x-q} {\la} \right )
\]
with $\la>0$, $q\in \R^2$ and $Q$ an orthogonal matrix in $\R^3$ do solve problem \equ{aa1}, and all share the least energy property:
$$
\int_{\R^2} |\nn U_{\la, q, Q} |^2 = 4\pi .
$$
Let us consider the $\alpha$-rotation matrix around the third axis  given by
$$e^{J \alpha  }  =   \left [ \begin{matrix}  \cos\alpha  &  -  \sin \alpha  & 0  \\  \sin\alpha  &  \cos\alpha & 0  \\  0 & 0 & 1   \end{matrix} \right ]   ,
\quad J = \left [ \begin{matrix}  0 &  -  1  & 0  \\  1  & 0  & 0  \\  0 & 0 & 0   \end{matrix} \right ]  .
$$

\medskip
In all what follows, we consider problem \equ{har map flow0}-\equ{har map flow02} with the  boundary condition \equ{har map flow01} given by the constant
\be\label{bcc0}
\vp(x) = \mathbf{e_3}   .\ee
Here and in what follows we denote
\begin{align}
\label{e123}
\mathbf{e_1}  = \left [\begin{matrix} 1\\ 0\\ 0 \end{matrix}  \right ], \quad \mathbf{e_2}  = \left [ \begin{matrix} 0\\ 1\\ 0 \end{matrix}  \right ] , \quad \mathbf{e_3}  = \left [ \begin{matrix} 0\\ 0\\ 1 \end{matrix}  \right ].
\end{align}
The constant boundary value $\mathbf{e_3}$ precisely corresponds to $ W  (\infty)$ where  $ W $ is the standard 1-corrotational  harmonic map \eqref{U00}.
This choice is made for convenience, in fact any sufficiently small perturbation of it is also admissible.
 In the radial $1$-corrotational equation \equ{11}, this boundary condition
in the disk $\Omega = D(0,R)$ simply corresponds to $v(R,t)=0$. All results below do apply to a boundary condition which slightly perturbs \equ{bcc0}, or in the case of entire space $\R^2$ where this value is set as a condition at infinity.

\begin{theorem}\label{teo1}
Given  points $q= (q_1,\ldots, q_k)\in \Omega^k$ and any sufficiently small $T>0$, there exist $u_0$ such the solution $u_q(x,t)$ of problem \equ{har map flow0}-\equ{har map flow02}, for $\vp$ given by \equ{bcc0},
blows-up at exactly those $k$ points as $t\uparrow  T$. More precisely, there exist numbers $\kappa_i^*>0$, $\alpha_i^*$ and a function $u_*\in H^1(\Omega)\cap C (\bar \Omega)$ such that
\be \label{blow-up}
u_q(x,t) - u_*(x) -  \sum_{j=1}^k  e^{ J\alpha_i^* }\big [\, W  \left (\frac {x- q_i} {\la_i} \right ) - W  (\infty) \,\big ]\ \to\  0 \ass t\uparrow  T, \ee
in the $H^1$ and uniform senses in $\Omega$ where
\be\label{rate}
\la_i(t) = \kappa_i^* \frac {T-t}{|\log(T-t)|^2} \big (1+ o(1)   \big )  \ass t\uparrow  T.
\ee
In particular, we have
$$
 |\nn u(\cdot, t)|^2 \rightharpoonup  |\nn u_*|^2 + 4\pi \sum_{j=1}^k \delta_{q_j} \ass t\uparrow  T.
$$

\end{theorem}

{ \cb

The blow-up solution we constructed in Theorem~\ref{teo1} has {\em no necks}. By the results of Qing-Tian \cite{Qing-Tian}, (see also Lin-Wang \cite{Lin-Wang1,Lin-Wang3}), this follows from the directly checked fact that the $L^2$ norm of the tension field $ \tau:= u_t $ is bounded as $t\uparrow T$. Our construction suggests that no necks should be present in planar solutions with isolated least energy blow-up points.


}

\medskip

In the next result we analyze the stability of the solutions constructed in Theorem~\ref{teo1}.
We recall that in the 1-corrotational class  in a disc, Chang-Ding-Ye \cite{Chang-Ding-Ye} provided robust conditions on  initial and boundary data that  guarantee finite time blow-up. Raphael-Schweyer \cite{rs1}  established stability {\em within} the  1-corrotational
class in entire space for a solution blowing-up with the rate \equ{rate0}.
Merle-Raphael-Rodnianski \cite{MRR} and  Raphael-Schweyer \cite{rs1}  conjectured instability
outside the 1-corrotational class.
Van der Berg and Williams \cite{BW} provided formal and numerical evidence  that blow-up  may indeed be destroyed by small non-radial perturbations of a 1-corrotational singularity.


\medskip
Our proof of Theorem \ref{teo1}  yields  {\em codimension-one stability}
of the predicted  blow-up phenomenon in the case of a single blow-up point when no symmetries are assumed.
The meaning of this form of stability is as follows:

\medskip
\begin{theorem} \label{stability}
Let $u(x,t)$ be the solution predicted in Theorem \ref{teo1} of the problem  \equ{har map flow0}-\equ{har map flow02}  that blows-up at a point $q\in \Omega$ and a time $T>0$.
Then there exists a $C^1$ manifold $\mathcal M$  in $C^1(\bar \Omega, S^2)$ with codimension one that contains $u_0 $   such that for any $\ttt u_0\in \mathcal M$ close to $ u_0$, the solution $\ttt u(x,t)$ of problem \equ{har map flow0}-\equ{har map flow02} with initial datum $\ttt u_0$ blows-up at a point
$\ttt q\in \Omega$ and a time $\ttt T$ which are close respectively to $ q$ and $T$.
\end{theorem}

{\cb
We discuss the general reason for the codimension-1 stability in Remark~\ref{remarkStability} in \S 2.
The generalization of the previous theorem to the solution with $k$ blow-up points of Theorem~\ref{teo1} is that
there is a manifold in $C^1$ of codimension $2k-1$ of initial data that leads to  $k$ simultaneous  blow-up points at a time $T$.
}

\medskip

The solutions in Theorems \ref{teo1} are classical in $[0,T)$. Our next result concerns the continuation of the solution after blow-up.
As we have mentioned Struwe \cite{Struwe} defined a global $H^1$-weak solution of  $\equ{har map flow0}$-$\equ{har map flow02}$. Struwe's solution is obtained
by just dropping the bubbles appearing at the blow-up time and then restarting the flow. The energy has jumps at each blow-up time generated by this procedure
 and it is decreasing. Decreasing energy suffices for uniqueness of the weak solution, as proven in \cite{Freire, Lin-Wang3.5}. On the other hand the bubble-dropping procedure modifies in time the topology of the image of the solution map. Topping \cite{Topping3} showed a different way to construct a continuation after blow up in the symmetric 1-corrotational class.  The solution in \cite{Chang-Ding-Ye} is continued after blow-up by attaching a bubble with opposite orientation,  which unfolds continuously the energy. The solution referred to is a  {\em reverse bubbling solution}. As emphasized in \cite{Topping3}, this continuation has the advantage that, unlike  Struwe's solution, it preserves the homotopy class of the map after blow-up. Formal asymptotic rates for 1-corrotational reverse bubbling were found in \cite{bhk}. In \cite{Bertsch} other forms of continuation of radial solutions were found.

 \medskip
 We establish that Topping's continuation can be made without symmetry assumptions, with exact asymptotics, for the solution in Theorem \ref{teo1}.
 We define  the bubble $\bar w$  with reverse orientation to that of $W $ as
 \be\label{barW}
\bar W (x) \ =\   e^{ J\pi } W (x) =  \frac 1{1+ |x|^2} \left (\begin{matrix}    -2x   \\  |x|^2 -1 \end{matrix}    \right )\  =\ \left (\begin{matrix}    -e^{ i\theta} \sin { w(r)}   \\  \cos { w(r)} \end{matrix}    \right ). \quad
\ee

\begin{theorem}\label{teo2}
Let $u_q(x,t) $ be the solution in Theorem \ref{teo1}. Then $u_q$ can be continued as an $H^1$-weak solution  in  $\Omega \times (0, T+\delta )$, which is continuous except at the points
$(q_i,T)$, with the property
that, besides expansion \equ{blow-up}, we have $u_q(x,T) = u_*(x)$
\[
u_q(x,t) - u_*(x) -  \sum_{j=1}^k  e^{ J\alpha_i^{*} } \big [\, \bar W  \left (\frac {x- q_i} {\la_i(t)} \right ) - \bar W  (\infty) \,\big ]\ \to\  0 \ass t\downarrow  T,
\]
in the $H^1$ and uniform senses in $\Omega$, where
\be
\lambda_i(t) =\kappa_i^* \frac{t-T}{|\log(t-T)|^2}.
\label{blow-up6}\ee

\end{theorem}

We observe that the energy in this continuation fails to be decreasing: it has a jump exactly at time $T$ and it goes back to its previous level immediately after.

\bigskip
We consider a question related to Theorem \ref{teo2}  treated in the 1-corrotational symmetric class in \cite{Topping3} and in \cite{Bertsch}:
the occurrence of perfectly smooth solutions which spontaneously develop a singularity in finite time by the addition of an infinitely concentrated bubble which instantaneously raises the energy in a multiple of $4\pi$. We find that the typical rate for this backward bubbling is $\dot \la(t)$ of order $\frac {t-T}{|\log(t-T)|}$ rather than \equ{blow-up6}. This was formally derived in \cite{bhk}.

\begin{theorem}
\label{teo3}
Given points $q_1,\ldots, q_k$ in $\Omega$ and any sufficiently small $T>0$      there exists an $H^1$-weak solution $u(x,t)$ of  problem $\equ{har map flow0}$-$\equ{har map flow02}$ in  $\Omega \times (0, T+\delta )$ which is continuous except at the points
$(q_i,T)$,  it is
 smooth in $\Omega\times (0,T]$ and has spontaneous reverse bubbling at the points $q_i$ in the form
\[
u(x,t) - u(x,T) -  \sum_{j=1}^k  \big [\, W  \left (\frac {x- q_i} {\la_i(t)} \right ) -  W  (\infty) \,\big ]\ \to\  0 \ass t\downarrow  T,
\]
in the $H^1$ and uniform senses in $\Omega$, where for some positive numbers $\kappa_i$
\be
\lambda_i(t) =\kappa_i \frac{t-T}{|\log(t-T)|}.
\label{lambdai}\ee

\end{theorem}

\bigskip
Before proceeding into the proof we make some further comments. It is plausible that the solutions of the form described in Theorem \ref{teo1} represent a form of ``generic''
bubbling phenomena for the two-dimensional harmonic map flow.
For instance, it is reasonable to think that the limits
along any sequence should have the same elements in the bubble decomposition.  On the other hand, evidence in the literature suggests that
typically only simple blow-up is present, having as a profile scalings of the 1-corrotational maps $W$ and $\bar W$. Higher degree  maps are represented by the $d$-corrotational symmetry class, $d\ge 1$,
$$
u(x,t)  =  \left (\begin{matrix}    e^{ di\theta} \sin v(r,t)   \\  \cos v(r,t) \end{matrix}    \right ) , \quad x= re^{i\theta}.
$$
 Steady states in this class correspond to scalings of  $v= w_d(r) = \pi -  2\arctan(r^d)$. It turns out that blow-up is not present in this class
 for $d\ge 4$. See Guan-Gustafson-Tsai \cite{Guan-Gustafson-Tsai}. It is conjectured that no blow-up exists also for $d=2,3$. This essentially discards higher degree blow-up.
On the other hand,  no multiple blow-up  (bubble trees) in the 1-corrotational class exists. See Van der Hout
\cite{vdh}. Infinite time multiple bubbling was found by Topping  \cite{Topping2} in a target different from $S^2$.
{\cb Bubbling rates faster than \equ{rate}}
do exist in the 1-corrotational case, but they are not stable, see Rapha\"el and Schweyer \cite{rs2}.
Many other results on bubbling phenomena, and regularity for harmonic maps and the harmonic map flow are available in the literature, we refer the reader to the book by Lin and Wang \cite{Lin-Wang4}.

\medskip
In bubbling phenomena in this and related problems very little is known in nonradial situations.
The method in \cite{rs1,rs2}, was successfully applied to very related blow-up phenomena in dispersive equations in symmetric classes. See for instance Rodnianski-Sterbenz \cite{Rod1}
 Merle-Rapha\"el-Rodnianski \cite{MRR},  Rapha\"el\cite{Rap1},  Rapha\"el-Rod\-nians\-ki \cite{RR1}.
Our results share a flavor with finite time multiple blow-up in the subcritical semilinear heat equation,
as in the results by Merle and Zaag \cite{MZ}. Bubbling associated to the critical exponent has been recently studied in \cite{DDS,CDM}.
Our approach is parabolic in nature. It is based on the construction of a good approximation and then linearizing inner and outer problems. An appropriate inverse for the inner equation
is then found (which works well if the parameters of the problems are suitably adjusted) which makes it possible the application of fixed point arguments.
The general approach, which we call inner-outer gluing, has already been applied to various singular perturbation elliptic problems, see for instance \cite{dkw,dkw1}.  A major difficulty we have to overcome is the coupled nonlocal ODE satisfied by the scaling and rotation parameter. We now explain in more details below.

\section{The 1-corrotational harmonic maps and the ansatz for a blowing-up solution}

The harmonic map equation for functions $U:\R^2\to S^2$ is the elliptic problem
\begin{align}
\label{hm0}
\Delta U + |\nn U|^2 U =0 \inn \R^2, \quad |U|=1 .
\end{align}
For $\xi\in \R^2$,  $\omega\in \R$, $\la>0$, we consider the family of solutions of \equ{hm0} given by the following 1-corrotational harmonic maps
\[
 U_{\la, \xi , \omega}  (x)  :=
 Q_\omega\, W (  y ) ,\quad y= \frac {x- \xi }{\la}
\]
 where $ W(y)  $ is the canonical 1-corrotational harmonic map
\[
 W  (y) \ =\    \frac 1{1+ |y|^2} \left (\begin{matrix}    2y   \\  |y|^2 -1 \end{matrix}    \right ) ,
 \quad  y\in \R^2 ,
\]
 and $Q_\omega$ is the $\omega$-rotation matrix
\[
Q_\omega:=   \left [ \begin{matrix}  \cos\omega  &  -  \sin \omega  & 0  \\  \sin\omega  &  \cos\omega & 0  \\  0 & 0 & 1   \end{matrix} \right ].
\]

Our purpose is to build a smooth blowing-up solution  $u:\bar \Omega \times [0,T)\to S^2$ of the problem
\be\label{har map flow} \left \{ \begin{aligned}
u_t  &= \Delta u + |\nabla u|^2 u \quad  \text{in } \Omega\times(0,T)\\
u  &= \mathbf{e_3}  \quad   \text{on } \pp\Omega\times(0,T)
\\
u(\cdot,0)  &= u_0 \quad   \text{in } \Omega .
\end{aligned}\right.\ee

\medskip
In order to keep the notation to a minimum, we shall do this in the case $k=1$ of a single given bubbling point $q\in \Omega$. The changes needed for the general case of Theorem \ref{teo1} are minor. More precisely
for any sufficiently small number $T>0$
we look for an initial datum $u_0$ such that the solution $u(x,t)$ of problem \equ{har map flow}
looks at main order like
\begin{align}
\label{formAnsatz}
U (x,t) :=    U_{\la(t), \xi(t) , \omega(t)}  (x)  =
Q_{\omega(t)}\,  W(  y), \quad y= \frac {x- \xi(t) }{\la(t)} ,
\end{align}
for certain functions  $\xi(t)$, $\la(t)$ and $\omega (t)$ of class $C^1([0,T])$ such that
\[
\xi(T) =q, \quad \la(T)=0 .
\]
 We shall find values for these functions
so that for a small remainder $v(x,t)$ we have that
$ u  = U + v $ solves \equ{har map flow}.  The condition $|U+ v|=1$
tells us that $u$ can be written as
\be\label{v}
u(x,t)= U+  \Pi_{U^\perp}\varphi +  a(\Pi_{U^\perp}\varphi) U,
\ee
where $\vp$ is a small function with values into $\R^3$ and we denote
$$
\Pi_{U^\perp} \vp :=   \vp - (\vp\cdot U) U, \quad  a(\zeta)  :=   \sqrt{1 - |\zeta|^2}-1  .
$$
The term $a(\Pi_{U^\perp}\varphi)$ has  a quadratic size in $\vp$ so it is of smaller order. We choose to decompose the remainder $\vp(x,t)$ in \equ{v} as the addition
of an ``outer'' part, better expressed in the global variable $x$, and an ``inner'' part which supported near the singularity and it is naturally  expressed as function of the slow variable $y$. More precisely, we let
\be\label{aris}
\vp(x,t) \  = \  \vp^{out} (x,t)  +   \vp^{in} (y,t),  \quad y=\frac{x-\xi(t)}{\la(t)}
\ee
where
$$\begin{aligned}
 \vp^{in} (y,t) \ = &  \  \eta_{R(t)}\left (y  \right) Q_{\omega(t)} \phi(y,t), \quad \phi(y,t)\cdot W(y) \equiv 0 \end{aligned} $$ and  $ \eta_R(y) :=  \eta\left (\frac {|y|} R  \right)$
 with $\eta(s)$ a smooth cut-off function so that
$$\eta(s)= \begin{cases} 1 & \hbox{  for $s<1$,}\\ 0 &\hbox{  for $s>2$}. \end{cases}$$
 The function $\phi(y,t)$ is defined for $|y|< 3R(t)$ where $R(t)\to +\infty$ and $\la(t) R(t)\to 0$ as $t\to T$.
With these definitions we see that $\Pi_{U^\perp} \vp^{in}= \vp^{in}$.

\medskip
We choose to the  decompose the outer part $\vp^{out}(x,t)$  in \equ{aris} as
\begin{align}
\label{outerSol}
\vp^{out}(x,t) \  =\   \Phi^0[\omega,\la, \xi] +  Z^*(x,t)  \, + \,  \psi(x,t),
\end{align}
where $ \Phi^0$ and  +  $Z^*(x,t)$ are explicit functions chosen as follows:
$\Phi^0[\omega,\la, \xi]$ is a function (which will be precisely described in the next section) that at main order eliminates the largest slow-decaying  part of the error of approximation $U_t$ in \equ{har map flow}.
Writing  $p(t) := \la(t) e^{i\omega(t)}$ and using polar coordinates $x= \xi(t)+ re^{i\theta}$,  we require
$$
 \pp_t \Phi^0  - \Delta_x \Phi^0  \approx   \frac 2r \left [\begin{matrix} \dot p(t) e^{i\theta} \\ 0   \end{matrix}\right ] \approx U_t  .
 $$
 On the other hand, we let $Z^*:\Omega\times (0,\infty) \to \R^3$ satisfy
\begin{align}
\label{heatZ*}
\left\{
\begin{aligned}
Z_{t}^* &= \Delta Z^*  \inn \Omega\times(0,\infty), \\
Z^*(\cdot ,t)  &=  0\inn  \pp\Omega \times (0,\infty),\\
Z^*(\cdot ,0)  &=  Z^*_0   \inn  \Omega ,
\end{aligned}
\right.
\end{align}
where
\begin{align}
\label{notationZ0star}
Z_0^*(x) =  \left [ \begin{matrix}z_0^*(x) \\  z_{03}^* (x) \end{matrix}   \right ] , \quad z_0^*(x) = z^*_{01}(x)  + i z^*_{02}(x)
\end{align}
is a small, sufficiently regular function essentially satisfying
 $$Z_0^*(q)= 0, \quad  \div  z^*_0(q)  + i \curl  z^*_0(q) \ne 0 . $$
In summary, we make the ansatz
\be
u\  = \  U \,  +\, v, \quad v =  \Pi_{U^\perp} \big(    \Phi^0[\omega, \la , \xi] + Z^* + \psi\big ) \, + \, \eta_R Q_\omega \phi  + a U    \label{canave}\ee
for a blowing-up solution $u(x,t)$ of \equ{har map flow},
where $\Phi$ and $\psi$ are lower order corrections.
Our task is to find  functions $\omega(t), \la(t) , \xi(t)$, $\psi(x,t)$ and $\phi(y,t)$ as described above, such that the remainder $v$ remains uniformly small.

\medskip
We will define a system of equations that we call the
{\em inner-outer gluing system}, essentially  of the form
\be\label{gluing0}
\left\{
\begin{aligned}
\la^2 \phi_t & =  L_ W  [\phi ]\ +\ H[p,\xi, \psi,\phi],  \quad
\phi\cdot  W   =  0 \inn     \R^2\times (0,T)
\\
\psi_t  &=  \Delta_x \psi \ +\  G[p,\xi, \psi,\phi]\qquad \qquad \qquad\ \,   \inn \Omega \times (0,T)
\end{aligned}
\right.
\ee
where
\[
L_W[ \phi] = \Delta_y \phi + |\nn_y W|^2 \phi + 2(\nn_y\phi \cdot \nn_y W)W , \quad \phi\cdot W = 0
\]
is the linearized operator for equation \equ{hm0} around $U= W$,
so that if the pair of functions $(\phi(y,t),\psi(x,t))$ solves it then $u$ given by \equ{canave}
is a solution of \equ{har map flow}.
 The point is to adjust the parameter functions $\omega, \la,\xi$ such that the inner problem can be solved for $\phi(y,t)$ which decays as $|y|\to \infty$. To fix the idea, let us consider the approximate elliptic equation, where time is regarded just as a parameter,
 $$
  L_ W  [\phi ]\ +\ H[p,\xi, 0,0] = 0 \inn \R^2
  $$
As we will discuss, a space-decaying solution $\phi(y,t)$ to this problem exists if a set of orthogonality conditions of the form
\begin{align}
\label{orthCond}
\int_{\R^2} H[p,\xi, 0,0](y,t)\, Z(y)\, dy = 0 \foral Z\in \mathcal Z
\end{align}
where $\mathcal Z $  is a  4-dimensional space constituted by decaying functions $Z(y)$ with $L_W[Z]=0$. These solvability conditions lead to an essentially explicit system
of equations for the parameter functions which will tell us in particular that for some small $\sigma>0$
\be\label{jofre}
\begin{aligned}
p(t) &= \cb - (\div z^*_0(q)  + {i\curl z_0^* (q)} ) \frac {|\log T| (T-t)}{\log^2 (T-t)} (1+ O( |\log T|^{-1+\sigma} )) ,
\\
\xi(t) & =   q + O((T-t)^{1+\sigma}),\end{aligned}  \ee
and we recall that we are consistently asking $ \div z^*_0(q)  + i{\curl z_0^* (q)} \ne 0$.

{\cb

\begin{remark}
\label{remarkStability}{\em
In the case of  blow-up at a single point, our codimension-1 stability result is directly connected to the solvability conditions \eqref{orthCond}. Indeed, the solution we construct depends at main order on four parameters functions: a scaling $\lambda(t)>0 $, a rotation angle $\omega(t)\in \R$, and the concentration point $\xi(t) \in \Omega$, see formula \eqref{formAnsatz}.
The presence of decaying functions in the kernel of the operator $L_W$ limits the decay of solutions to the inner linearized evolution. Too slow decay could make the contribution to the error in the remote regime too large. Sufficient decay in the linearized evolution can only be achieved if the right hand side satisfies four solvability conditions at all times $t\in [0,T)$.
These are conditions \eqref{orthCond}, which translate into a system of integro-differential equations for
$\lambda(t)$, $\omega(t)$, $\xi(t)$.
For $\xi(t)$ the equation is almost a first order ODE, which imposes a constraint between $\xi(0)$ and $\xi(T)$.
The equations for $\lambda(t)$ and $\omega(t)$ are better expressed for the combined quantity $p(t)=\lambda(t)e^{i\omega(t)}$. It is an integro-differential equation, whose solution has the expansion \eqref{jofre}.
This relation evaluated at time $t=0$ says that $\lambda(0)e^{i\omega (0)} = - (\div z^*_0(q)  + {i\curl z_0^* (q)} ) \frac { T}{|\log T|} (1+ O( |\log T|^{-1+\sigma} )) $.
Considering $z_0^*$ as fixed, this equation links $\lambda(0)$ with $T$ and determines $\omega(0)$ uniquely in $[0,2\pi)$. In other words in the initial condition we lose the freedom to choose $\omega(0)$.
We also loose the  freedom of choosing $\lambda (0)$ if $T$ was fixed, but this is recovered by letting  $T$ vary.
In the 1-corrotational case, the symmetries imply that $\curl z_0^*(0)=0$ and $\omega\equiv 0$, and therefore there is no loss of stability in this situation.
The argument above considers $z_0^*$ as fixed, but the analysis with all variables taken into consideration  is detailed in \S~\ref{Stability}.
}
\end{remark}
}

{\cb

\begin{remark}{\em
 Let us explain why the numbers $ \div z^*_0(q)$ and $\curl z_0^* (q)$ appear in expression \eqref{jofre}. Let us restrict the analysis to the 1-corrotational ansatz \eqref{1corrot} so that the harmonic map flow reduces to \eqref{11}.
 We look for a solution that approximately looks like the superposition of a bubble \eqref{formAnsatz} with $\xi(t)\equiv 0$, $\omega(t)\equiv 0$ perturbed by \eqref{outerSol} consisting only of a term $Z^*$ of the form
\[
Z^*(r,t) = \left[
\begin{matrix}
e^{i\theta}f(r,t)\\0
\end{matrix}
\right]
\]
with $f$ satisfying $\partial_t f= \partial_{rr}f+\frac{1}{r} \partial_r f - \frac{1}{r^2}f$ and $f(0,t)=0$,
namely we propose an approximate solution $v(r,t) = w(\frac{r}{\lambda}) + f(r,t)$  of \eqref{11}.
With the notation \eqref{notationZ0star}, we have that $\div z_0^*(0)= 2 \partial_r f(0,0)$, $\curl z_0^*(0)=0$.
Expanding $f(r,t) \approx \partial_r f(0,0) r$ we get that
\begin{align*}
-\partial_t v + \Delta v - \frac{\sin(2v)}{2r^2}
&\approx
\rho  w_\rho \frac{\dot\lambda}{\lambda}
-\frac{1}{\lambda} \frac{w_\rho}{\rho} \partial_r f(0,0)
,\quad \rho = \frac{r}{\lambda}.
\end{align*}
Imposing that the right hand side above is $L^2$-orthogonal to the kernel of the linearized equation  in a ball of radius $ \sqrt{T-t}$ suggests that
\[
|\log (T-t)| \dot \lambda(t)  \approx  c f_r(0,0),
\]
for a positive universal constant $c$.
This derivation is not correct because significant boundary terms appear in the integration. This issue is solved by the addition of the
nonlocal term $\Phi^0$. On the other hand, this suggests the role played by $\div z_0^*(0)$ in the expression for $\lambda$. The term  $\curl z_0^*(0)$ appears when we introduce the rotation angle $\omega$, which is needed outside the $1$-corrotational regime. }
\end{remark}
}

\medskip
In the next sections we will carry out in detail the program for the construction sketched above. In \S \ref{lin}
we will set up several facts about the elliptic linearized operator that will be needed in all subsequent computations.
In \S \ref{phi00} we will compute in precise way the error of approximation and define the function $\Phi^0$ mentioned. We also introduce the precise terms appearing in the inner-outer gluing system \equ{gluing0}. In \S \ref{informal} we will perform the computations of the orthogonality conditions which lead to expressions \equ{jofre}. In \S \ref{sectInnerOuter} we will carry out the full construction setting up the system as a fixed point problem. We make precise statements of
the necessary (major) steps needed, in particular a subtle linear theory for the parabolic inner problem that mimics the Fredholm alternative for the elliptic equation mentioned above, which is developed in \S \ref{sectLinearTheory}.
Related Lipschitz estimates and linear bounds for the outer problem are performed in \S \ref{lips} and \S \ref{lips1}. The adjustment of the parameters to solve the full system is the purpose of \S \ref{secLambda}. The stability statement is proved in \S \ref{Stability}. Finally, we discuss the continuation and reverse bubbling results in \S \ref{reverse}.

\section{The linearized operator around the bubble}\label{lin}

The linearized operator for \equ{hm0} around  $U=U_{\la, \xi , \omega}$ is the elliptic operator
\[
L_U[ \vp] = \Delta \vp + |\nn U|^2 \vp + 2(\nn\vp \cdot \nn U)U .
\]
Differentiating $U$ with respect to each of its parameters we obtain functions that annihilate this operator, namely solutions of $L_U[\vp]=0$.
Setting  $ y= \frac {x-\xi } {\la}$, these functions are
\begin{align*}
 \pp_\la U_{\la, \xi , \omega}(x)   = & \frac {1}{\la} Q_\omega \nn  W  (y)\cdot y ,    \\
     \pp_\omega U_{\la, \xi , \omega} (x)   =  & (\pp_\omega  Q_\omega)  W  (y), \\
             \pp_{\xi_j } U_{\la, \xi , \omega}(x)  =&  \frac 1\la Q_\omega \pp_{y_j}  W   (y).
\end{align*}
We observe that
\[
(\pp_\omega  Q_\omega)  =    Q_\omega J_0, \quad \mbox{where} \ \  J_0 = \left [ \begin{matrix}  0  &  -  1  & 0  \\  1  &  0 & 0  \\  0 & 0 & 0   \end{matrix} \right ] .
\]

We can represent
 $ W  (y)$ in polar coordinates,
$$
 W  (y) =  \left (\begin{matrix}    e^{ i\theta} \sin {w(\rho )}   \\  \cos {w(\rho )} \end{matrix}    \right ), \quad w(\rho ) = \pi - 2\arctan (\rho ), \quad y= \rho e^{i\theta}.
$$
We notice that
$$
w_\rho = - \frac 2{1+\rho^2} , \quad \sin w = -\rho w_\rho = \frac {2\rho} {1+\rho^2}, \quad \cos w = \frac {\rho^2 -1}{1+\rho^2} ,
$$
and derive the alternative expressions
\begin{align}
\nonumber
\pp_\la U_{\la, \xi , \omega}(x)   & = \frac {1}{\la} Q_\omega Z_{01}(y) ,
& Z_{01}(y) & = \rho w_\rho(\rho)\, E_1(y),
\\
\nonumber
\pp_\omega U_{\la, \xi , \omega} (x)  &  =   \quad  Q_\omega Z_{02}(y) ,
&   Z_{02}(y) &= \rho w_\rho(\rho)\, E_2(y),
\\
\nonumber
\pp_{\xi_j } U_{\la, \xi , \omega}(x)  &= \frac 1\la Q_\omega Z_{11}(y),
&  Z_{11}(y) & = w_\rho(\rho)\,  [ \cos\theta \, E_1(y) + \sin\theta\,  E_2(y)  ],
\\
\label{ZZ}
\pp_{\xi_j } U_{\la, \xi , \omega}(x)  & = \frac 1\la Q_\omega Z_{12}(y),
& Z_{12}(y)  & =w_\rho(\rho)\,  [ \sin\theta  \, E_1(y)  - \cos \theta\,  E_2(y)  ] ,
\end{align}
where
\[
E_1(y) =    \left (\begin{matrix}    e^{ i\theta} \cos w(\rho )   \\ - \sin {w(\rho )} \end{matrix}    \right ), \quad E_2(y) =
\left (\begin{matrix}    ie^{ i\theta}    \\  0 \end{matrix}    \right ) .
\]
The relation $|U_{\la, \xi , \omega}|=1$ implies that all the functions $Z_{ij}$ are pointwise orthogonal to $U_{\la, \xi , \omega}$. In fact the vectors $E_1(y)$, $E_2(y)$
constitute an orthonormal basis of the tangent space to $S^2$ at the point $ W  (y)$.

We have
$L_ W  [ Z_{lj}] =0 $ where for a function $\phi(y)$ we define
\[
L_ W  [\phi] =   \Delta_y \phi  + |\nn  W (y)|^2  \phi  + 2(\nn  W (y)\cdot \nn\phi) {  W (y) }.
\]
In addition to the elements \eqref{ZZ} in the kernel of $L_W$ there are also two other relevant functions in the kernel, namely
\begin{equation}
\label{Zmode-1}
Z_{-1,1} = \rho^2 w_\rho(\rho)(\cos \theta E_1 -\sin \theta E_2),  \ Z_{-1,2} =  \rho^2 w_\rho (\rho)(\sin \theta E_1 + \cos \theta E_2) .
\end{equation}

It is worth noticing the connection between this operator and $L_U$ which is given by
\[
L_U[ \vp]  = \frac 1{\la^2} Q_\omega L_ W  [\phi], \ \quad \vp(x) =\phi(y), \quad y=\frac{x-\xi}{\la}.
\]

\medskip

\subsection*{The linearized operator at functions orthogonal to \texorpdfstring{$U$}{U}}
It will be especially significant to compute  the action of $L_U$ on functions with values pointwise orthogonal to $U$.
In what remains of this section we will derive various formulas that will be very useful later on.

\medskip
For an arbitrary function
$\Phi(x)$ with values in $\R^3$ we denote the projection
$$
\Pi_{U^\perp} \Phi :=   \Phi - (\Phi\cdot U) U.
$$
A direct computation shows the validity of the following:
\begin{equation}
\nonumber
L_U[\Pi_{U^\perp}\Phi] = \Pi_{U^\perp} \Delta \Phi  + \ttt L_U[\Phi ]
\end{equation}
where
\[
\ttt L_U[ \Phi ] := |\nn U|^2 \Pi_{U^\perp} \Phi - 2\nn (\Phi \cdot U ) \nn U,
\]
and
$$
\nn (\Phi \cdot U ) \nn U =  \pp_{x_j} (\Phi \cdot U )\, \pp_{x_j} U .
$$
A very convenient expression for $\ttt L_U[ \Phi ]$ is obtained if we use polar coordinates. Writing in complex notation
$$
\Phi(x) = \Phi(r,\theta), \quad x= \xi + r e^{i\theta},
$$
we find
\begin{align}
\label{Ltilde}
 \ttt L_U[\Phi] =  - \frac 2{\la} w_\rho(\rho)\,  [ (\Phi_r \cdot U)Q_\omega E_1  - \frac 1{r} (\Phi_\theta \cdot U) Q_\omega E_2 ], \quad \rho = \frac r\lambda.
\end{align}

\bigskip
We single out two consequences of formula \equ{Ltilde} which will be crucial for later purposes.
Let us assume that $\Phi(x)$ is a  $C^1$ function $\Phi : \Omega \to \C \times \R$, which we express in the form
\begin{align}
\label{notation-Phi}
\Phi(x)\ =\ \left ( \begin{matrix} \vp_1 (x) + i \vp_2(x)  \\ \vp_3 (x)  \end{matrix} \right ).
\end{align}
We also denote $$\vp = \vp_1 + i \vp_2 , \quad \bar \vp = \vp_1  - i \vp_2 $$ and define the operators
$$
\div \vp   = \pp_{x_1}\vp_1 + \pp_{x_2}\vp_2, \quad \curl \vp = \pp_{x_1}\vp_2 - \pp_{x_2}\vp_1 . $$
We have the validity of the following formula
\be
 \ttt L_U [\Phi ] =  \ttt L_U [\Phi ]_0  +  \ttt L_U [\Phi ]_1 +  \ttt L_U [\Phi ]_2\ ,
 \label{Ltilde2} \ee
 where
\begin{align}
\label{Ltilde-j}
\left\{
\begin{aligned}
  \ttt L_U [\Phi ]_0 & =    \la^{-1} \rho w_\rho^2\, \big [\, \div ( e^{-i\omega} \vp)\, Q_\omega  E_1    + \curl ( e^{-i\omega} \vp)\, Q_\omega E_2
  \, \big ]\,
 \\
\ttt L_U [\Phi ]_1 & =
   -
  2\la^{-1} w_\rho  \cos w \, \big [\,(\pp_{x_1} \vp_3) \cos \theta +    (\pp_{x_2} \vp_3) \sin  \theta \, \big ]\,Q_{\omega} E_1
\\
& \quad - 2\la^{-1}  w_\rho  \cos w  \, \big [\, (\pp_{x_1} \vp_3) \sin \theta -    (\pp_{x_2} \vp_3) \cos  \theta \, \big ]\, Q_{\omega}E_2\ ,
\\
\ttt L_U [\Phi ]_2 &=
\la^{-1} \rho w_\rho^2 \, \big [\, \div (e^{i\omega}\bar \vp)\, \cos 2\theta  -   \curl ( e^{i\omega}\bar \vp)\, \sin 2\theta  \, \big ]\, Q_{\omega} E_1
\\
&  \quad
+  \la^{-1} \rho w_\rho^2 \, \big [\,  \div ( e^{i\omega}\bar \vp)\, \sin 2\theta +   \curl  ( e^{i\omega}\bar \vp)\,  \cos 2\theta   \, \big ]\,Q_{\omega}E_2 .
\end{aligned}
\right.
\end{align}


Another corollary of formula \equ{Ltilde} that we single out is the following:
assume that
$$
\Phi (x) =  \left ( \begin{matrix} \phi(r) e^{i\theta} \\ 0  \end{matrix}  \right) , \quad x = \xi + re^{i\theta} , \quad \rho =\frac r\la
$$
where $\phi(r)$ is complex valued.
Then
\be\label{uii}
\ttt L_U [\Phi] =    \frac 2\la  w_\rho(\rho)^2   \left [  {\rm Re } \,( e^{-i\omega} {\partial_r \phi(r) } )  Q_\omega E_1   +
 \frac 1r {\rm Im}  \,( e^{-i\omega} \phi(r))  Q_\omega E_2  \right ].
\ee

A final result in this section is a computation (in polar coordinates) of the operator $L_U$ acting on a function of the form
\[
\Phi (x) =  \vp_1(\rho, \theta) Q_\omega E_1 + \vp_2(\rho, \theta)Q_\omega E_2 ,  \quad x= \xi + \la \rho e^{i\theta}.
\]
We have:
\begin{align}
L_U[\Phi ] & =   \la^{-2}  \left ( \pp^2_\rho \vp_{1} +  \frac  { \pp_\rho\vp_{1} } {\rho }  +  \frac { \pp^2_\theta \vp_{1} } {\rho^2  }+    (2w_\rho ^2 - \frac 1{\rho^2} )\vp_1 - \frac{2}{\rho^2}  \pp_\theta \vp_{2}\cos w   \right ) \,  Q_\omega E_1
\nonumber\\ & \quad
+   \la^{-2}  \left ( \pp^2_\rho \vp_{2} +  \frac  { \pp_\rho\vp_{2} } {\rho }  +  \frac { \pp^2_\theta \vp_{2} } {\rho^2  }+    (2w_\rho  ^2 - \frac 1{\rho^2} )\vp_2  + \frac{2}{\rho^2}  \pp_\theta \vp_{1}\cos w \right ) \, Q_\omega E_2 .
\nonumber
\end{align}

\section{The error and the inner-outer gluing system}
\label{phi00}

The linearized operator for \equ{hm0} around  $U=U_{\la, \xi , \omega}$ is the elliptic operator
\[
L_U[ \vp] = \Delta_x \vp + |\nn_x U|^2 \vp + 2(\nn_x\vp \cdot \nn_x U)U ,
\]
where $\varphi=\varphi(x,t)$. Consistently we denote for a function $\phi = \phi(y,t)$.
Let us denote
$$
S(u) :=  -u_t + \Delta u + |\nn u|^2 u 
$$
A useful observation that we make is that as long as the constraint $|u|=1$ is kept at all times and $u= U+ v$ with $|v|\le \frac 12 $ uniformly, then  for $u$ to solve  equation \equ{har map flow} it suffices that
\be \label{bU} S(U+v) = b(x,t) U \ee  for some scalar function $b$. Indeed, we observe that since $|u|\equiv 1 $ we have
\[
b\, (U\cdot u) = S(u) \cdot u = -  \frac 12 \frac d{dt} {|u|^2} +  \frac 12 \Delta {|u|^2} = 0 ,
\]
and since  $U\cdot u \ge  \frac 12 $, we find that  $b\equiv 0$.

\medskip
Using that
 $$
 \Delta U + |\nn U|^2 U =0
 $$
 we find the following expansion for $S(U+v)$ with $v$ given by \equ{v}:
\[
S( U + \Pi_{U^\perp} \varphi + aU )
=
-U_t  - \pp_t \Pi_{U^\perp} \varphi+  L_U(\Pi_{U^\perp} \varphi) +  N_U( \Pi_{U^\perp} \vp ) + c(\Pi_{U^\perp} \vp) U   \nonumber
\]
where for $\zeta =\Pi_{U^\perp} \vp $, $a = a(\zeta )$,
\begin{align}
\nonumber
L_U(\zeta )
&= \Delta \zeta + |\nabla U|^2 \zeta + 2(\nabla U\cdot  \zeta ) U
\\
\nonumber
N_U( \zeta )
&=
\big[
2 \nn (aU)\cdot \nn (U+ \zeta  )  + 2 \nabla U \cdot \nabla \zeta   + |\nabla \zeta  |^2
+ |\nabla (a U ) |^2 \,
\big] \zeta
- aU_t
\\
& \quad
\nonumber
+ 2\nn a \nn U  ,
\\
\nonumber
c(\zeta )
& =  \Delta a - a_t  + ( |\nn (U + \zeta  + aU)|^2
- |\nn U|^2 )(1 + a)  -   2\nn U\cdot \nn \zeta
\end{align}
Since we just need to have an equation of the form \equ{bU} satisfied,  we find that
$$ u = U + \Pi_{U^\perp} \varphi + a(\Pi_{U^\perp} \varphi)U $$
solves \equ{har map flow}
if and only if $\vp$ satisfies
\be\label{ecuacion}
0=
-U_t  - \pp_t \Pi_{U^\perp} \varphi+  L_U(\Pi_{U^\perp} \varphi) +  N_U(\Pi_{U^\perp}\vp ) + b(x,t) U ,
\ee
for some scalar function $b$. The logic of the construction goes like this:
We decompose $\varphi$ into the sum of two functions $\vp = \vp^i + \vp^o$, the ``inner'' and ``outer'' solutions and reduce equation \equ{ecuacion}   to solving a  system of two equations in $(\vp^i, \vp^o)$ that we call the inner and outer problems.

\medskip
The inner function $\vp^i(x,t)$ will be assumed supported only near $x=\xi(t)$ and better read as a function of the scaled space variable $y= \frac{x-\xi(t)}{\la(t)} $ with zero initial condition and such that $\vp^i\cdot U=0$, so that $\Pi_{U^\perp} \vp^i = \vp^i$.
The outer function   $\vp^o(x,t)$ will be made out of several pieces and its role is essentially to satisfy \equ{ecuacion} far away from the concentration point $x=\xi(t)$.

We write equation \equ{ecuacion} in the following way:
\begin{align}
\label{eqsys0}
0 & = -\pp_t \varphi^i+  L_U[\vp^i] + \ttt L_U [\vp^o]
-\Pi_{U^\perp} [\pp_t \varphi^o -\Delta \vp^o  + U_t ]
\\
\nonumber
& \quad +  N_U( \vp^i + \Pi_{U^\perp}\vp^o ) + (\vp^o\cdot U)U_t + b U .
\end{align}
For the outer problem, we consider a function $\Phi^0$ that depends explicitly on the parameter functions chosen in such a way that $\Pi_{U^\perp} [\pp_t \Phi^0 -\Delta \Phi^0  + U_t ]$ gets concentrated near $x=\xi(t)$ by elimination of the
terms in the first error $U_t$ associated to dilation and rotation.
Then we  write
\be\label{outer01}
\vp^o (x,t)=  \Phi^0(x,t) + \Psi^*(x,t) .
\ee
For the inner solution,  we consider a smooth smooth cut-off function $\eta_0(s)$ with $\eta_0(s)=1$ for $s<1$ and $=0$ for $s>\frac 32$.
We also consider a positive, large smooth function
$
R(t) \to +\infty
$ as $t\to T$
that we will later specify.
We define
$$
\eta(x,t):=  \eta_0 \left ( R(t)^{-1} |y|    \right ) , \quad y= \frac {x- \xi(t) }{\la(t) } $$
and let
$$
\vp^i(x,t)  = \eta(x,t) Q_\omega \phi (y, t)   , \quad y= \frac {x- \xi(t) }{\la(t) }
$$
for a function  $\phi(y,t)$ with initial condition
$ \phi(\cdot , 0 ) = 0$ that satisfies $\phi(\cdot ,t) \cdot  W  \equiv  0 $, defined for
$|y|\le 2R(t)$ and that vanishes as $t\to T$.
Then we have
\begin{align*}
Q_{-\omega} L_U[\vp^i]
&=   \la^{-2} \eta   L_ W  [\phi]   +  (\Delta_x \eta) \phi + 2 \la^{-1} \nn_x \eta \nn_y \phi
\\
Q_{-\omega} \vp^i_t
& =  \eta \bigl( \phi_t  - \la^{-1}\dot\la y\cdot \nn_y \phi
- \la^{-1}\dot\xi \cdot\nn_y \phi
+ \dot\omega   Q_{-\omega} \pp_\omega Q_\omega \phi  \bigr) + \eta_t \phi .
\end{align*}
Equation \equ{eqsys0} then becomes
\begin{align}
 \label{eqsys1}
0 & =  \la^{-2} \eta Q_\omega  [- \la^2 \phi_t +  \ L_ W  [\phi ] + \la^2 Q_{-\omega} \ttt L_U [\Psi^*] ]
\\
\nonumber
& \quad
+ \eta Q_\omega( \la^{-1}\dot\la  y\cdot \nn_y \phi
+ \la^{-1} \dot\xi \cdot\nn_y \phi  - \dot\omega J \phi
)
\\
\nonumber
& \quad
+      \ttt L_U  [ \Phi^0  ]  +
 \Pi_{U ^\perp} [\pp_t \Phi^0 -\Delta_x \Phi^0  + U_t ]
\\
\nonumber
& \quad - \pp_t \Psi^* +\Delta \Psi^*  +  (1-\eta) \ttt L_U [\Psi^*] +  Q_\omega[(\Delta_x \eta) \phi + 2  \nn_x \eta \nn_x \phi - \eta_t  \phi]
\\
\nonumber
& \quad +
N_U( \eta Q_\omega \phi  + \Pi_{U^\perp}( \Phi^0 +\Psi^*) ) + ((\Psi^*+ \Phi^0)\cdot U)U_t + b U .
\end{align}
Next we will define precisely the operator $\Phi^0$ and estimate the quantity
 \be\label{eqs} \ttt L_U  [ \Phi^0  ]  +\Pi_{U ^\perp} [\pp_t \Phi^0 -\Delta_x \Phi^0  + U_t ] . \ee
The idea is to choose $\Phi^0$ such that  $\pp_t \Phi^0 -\Delta_x \Phi^0  + U_t \approx 0 $  whenever $|x-\xi|\gg \la$,
so that in particular the last error term in the outer equation \equ{outer01} is of smaller order.

\medskip
Invoking formulas \eqref{ZZ}  to compute $U_t$ we get
\begin{align*}
 U_t =   \dot\la \pp_\la U_{\la, \xi , \omega}  + \dot\omega  \pp_\omega U_{\la, \xi , \omega} +   \pp_{\xi } U_{\la, \xi , \omega}\cdot \dot \xi
 =   \EE_{0}   + \EE_{1} ,
 \end{align*}
where, setting $  y =  \frac{x-\xi }{\la}= \rho e^{i\theta}  $, we have
\begin{align*}
\EE_{0} (x,t)
& =  - Q_\omega [  \frac{\dot \la}{\la} \rho w_\rho(\rho)\,  E_1(y) \, + \,   {\dot \omega } \rho w_\rho(\rho)\,   E_2(y)\, ]
\\
\EE_{1} (x,t)
& = -\frac{\dot \xi_{1}}{\la} \, w_\rho(\rho)\, Q_\omega [\ \cos\theta \, E_1(y) + \sin\theta\,  E_2(y)  ]\,
\\
& \quad
- \frac{\dot \xi_{2}}{\la}\,w_\rho(\rho) \,  Q_\omega[ \sin\theta  \, E_1(y)  - \cos \theta\,  E_2(y) \, ].
\end{align*}
Since $\EE_1$ has faster space decay in $\rho $ than $\EE_0$ we will choose $\Phi^0$ to be an approximate solution of
\begin{align}
\label{xx}
\Phi^0_t -\Delta_x \Phi^0  + \EE_0 =0 .
\end{align}
For $ x = \xi + re^{i\theta}$ and $r\gg \la$ we have
\begin{align*}
\EE_0 (x,t) & = - \frac {2 r } {r^2+\la^2}\left [  \dot \la Q_\omega E_1 + \la\dot \omega   Q_\omega E_2 \right]
\approx  - \frac {2 r } {r^2+\la^2}\left [\begin{matrix} (\dot\lambda+  i \lambda \dot\omega) e^{i(\theta+\omega)} \\ 0   \end{matrix}\right ] .
\end{align*}
Here and in what follows we let
\begin{align*}
p(t) = \lambda (t) e^{i\omega(t)} .
\end{align*}
Then
\begin{align*}
 - \frac {2 r } {r^2+\la^2}\left [\begin{matrix} (\dot\la+  i \la \dot\omega) e^{i(\theta+\omega)} \\ 0   \end{matrix}\right ]  =  - \frac {2 r } {r^2+\la^2}\left [\begin{matrix} \dot p(t)e^{i\theta} \\ 0   \end{matrix}\right ] =: \ttt \EE_0 (x,t) .
\end{align*}
With the aid of Duhamel's formula for the standard heat equation, we find that the following function is a good approximate solution of $ \Phi^0_t -\Delta_x \Phi^0  + \ttt \EE_0 =0 $ and hence of \equ{xx}.
We define
\begin{align}
\nonumber
\Phi^0[\omega,\la,\xi]  & :=   \left [ \begin{matrix}  \vp^0(r,t) e^{i\theta }  \\ 0 \end{matrix}   \right ]
\\
\nonumber
\vp^0(r,t)
&
= -\int_{-T}^t  \dot p(s) r k(z(r),t-s) \, ds
\\
\nonumber
z(r) & = \sqrt{ r^2+ \la^2} ,\quad k(z,t) = 2\frac{1-e^{-\frac{z^2}{4 t}}}{z^2} ,
\end{align}
where for technical reasons that will be made clear later on, $p(t)$ is also assumed to be defined for negative values of $t$.

A direct computation yields
$$
\Phi^0_t + \Delta_x \Phi^0  + \ttt \EE_0    =
\ttt \RR_0 +\ttt \RR_1 ,  \quad \ttt \RR_0 =  \left ( \begin{matrix} \RR_0   \\ 0 \end{matrix}   \right ),\quad
\ttt\RR_1 =  \left ( \begin{matrix} \RR_1   \\ 0 \end{matrix}   \right )
$$
where
\[
\RR_0:=  - re^{i\theta}   \frac {\la^2}{z^4} \int_{-T}^t  \dot p(s)  ( z{k_z} - z^2 k_{zz}) (z(r),t-s) \, ds
\]
and
\begin{align*}
\RR_1 & :=
- e^{ i\theta}  {\rm Re}\,( e^{-i\theta} \dot \xi(t))
 \int_{-T}^t  \dot p(s) \, k(z(r),t-s) \, ds
\\
&\qquad
+  \frac r{z^2} e^{i\theta} \, (\la\dot\la(t)  -  {\rm Re}\,( re^{i\theta} \dot\xi(t)) )
\int_{-T}^t  \dot p(s) \ {zk_z}(z(r),t-s)\,  ds.
\end{align*}
We observe that $\RR_1$ is actually a term of smaller order.
Using formulas \equ{Ltilde2},  \equ{uii} and the facts
 $$ \frac {\la^2r }{z^4} = \frac 1{4\la} \rho w_\rho^2, \quad  \frac r {z^2} (1-\cos w) = \frac 1{2\la}  \rho w_\rho^2 ,  $$
we derive an expression for the quantity \equ{eqs}:
\begin{align*}
&
\ttt L_U[\Phi^0]  + \Pi_{U^\perp} [ -U_t + \Delta \Phi^0 -\Phi^0_t ]
\\
 & =   \ttt L_U[\Phi^0]  -\EE_1 + \Pi_{U^\perp} [\ttt \EE_0] - \EE_0 +
\Pi_{U^\perp} [\ttt \RR_0] +  \Pi_{U^\perp} [\ttt \RR_1]
\\
&=  \KK_{0}[p,\xi]  + \KK_1[p,\xi] +\Pi_{U^\perp} [\ttt \RR_1]
\end{align*}
where
\begin{align*}
\KK_{0}[p,\xi] =  \KK_{01}[p,\xi] + \KK_{02}[p,\xi]
\end{align*}
with
\begin{align}
\nonumber
\KK_{01}[p,\xi]
:= - \frac {2}{\la} \rho w_\rho^2
\int_{-T} ^t  \left [ {\rm Re  } \,( \dot p(s) e^{-i\omega(t)} )   Q_\omega E_1+
 {\rm Im  } \,( \dot p(s) e^{-i\omega(t)} ) Q_\omega E_2   \right ]
\\
\label{K01}
\cdot  k(z,t-s)  \, ds
\end{align}
\begin{align}
\nonumber
\KK_{02}[p,\xi]
&
:=  \frac 1{\la} \rho w_\rho^2  \left [  {\dot\la}
-
\int_{-T} ^t  {\rm Re  } \,( \dot p(s) e^{-i\omega(t)} ) r k_z(z,t-s) z_r \, ds\, \right]  Q_\omega E_1
\\
\nonumber
&\quad
-   \frac{1}{4\lambda} \rho w_\rho^2 \cos w  \left [  \int_{-T}^t   {\rm Re}\, ( \dot p(s)e^{-i\omega(t) }  )
\, ( z{k_z} - z^2 k_{zz}) (z,t-s)\, ds\, \right ]  Q_\omega E_1
\\
\label{K02}
&\quad
-    \frac{1}{4\lambda} \rho w_\rho^2  \left [  \int_{-T}^t   {\rm Im }\, ( \dot p(s)e^{-i\omega(t) }  )
\, ( z{k_z} - z^2 k_{zz}) (z,t-s)\, ds\,  \right ]  Q_\omega E_2  ,
\end{align}
\begin{align}
\label{K1}
\KK_{1}[p,\xi]
& :=
\frac 1\la  w_\rho \, \big [
\Re \big (  (\dot  \xi_1 - i \dot \xi_2)  e^{i\theta } \big ) Q_\omega E_1
+ \Im \big(  (\dot  \xi_1 - i \dot \xi_2)  e^{i\theta } \big ) Q_\omega E_2       \big ].
\end{align}
We insert this decomposition in equation \equ{eqsys1}
and see that we  will have a solution to the equation if the pair $(\phi,\Psi^*)$ solves
the {\em inner-outer gluing system}
\begin{align}
\label{inner1}
\left\{
\begin{aligned}
\la^2 \phi_t  & =  L_ W  [\phi ]
+ \la^2  Q_{-\omega} \left[ \ttt L_U  [\Psi^* ]
+ \KK_{0}[p,\xi]+ \KK_{1}[p,\xi]
\right]
  \inn \DD_{2R}
\\
\phi\cdot  W  & =  0   \inn \DD_{2R}
\\
\phi(\cdot, 0)  & = 0=\phi(\cdot, T),
\end{aligned}
\right.
\end{align}
\smallskip
\be
\label{outer1}
\Psi^*_t  =  \Delta_x \Psi^*  +  g[p,\xi, \Psi^*,\phi]\inn \Omega \times (0,T)
\ee
where
\begin{align}
\label{GG}
 g[p,\xi, \Psi^*,\phi] & := (1-\eta) \ttt L_U [\Psi^*] + (\Psi^*\cdot U ) U_t
\\
\nonumber
& \quad  +
Q_\omega
\bigl( (\Delta_x \eta) \phi + 2  \nn_x \eta \nn_x \phi - \eta_t  \phi
\bigr)
\\
\nonumber
& \quad +  \eta Q_\omega\bigl( - \dot\omega J \phi  +  \la^{-1}\dot\la  y\cdot \nn_y \phi + \la^{-1} \dot\xi \cdot\nn_y \phi \bigr)
\\
\nonumber
& \quad + (1-\eta)[ \KK_{0}[p,\xi]+ \KK_{1}[p,\xi]] + \Pi_{U^\perp}[ \ttt \RR_1] + ( \Phi^0\cdot U)U_t
\\
\nonumber
& \quad   +  N_U( \eta Q_\omega \phi  + \Pi_{U^\perp}( \Phi^0 +\Psi^*) ) ,
\end{align}
and we denote
$$
\DD_{\gamma R} = \{(y,t)\in \R^2\times (0,T) \ /\ |y|< \gamma R(t) \}.
$$
Indeed if $(\phi,\Psi^*)$ solves this system, then we have that
\be\label{upa}
u(x,t) =  U +  \Pi_{U^\perp} [ \Phi^0 + \Psi^* + \eta Q_\omega \phi ] + a(\Pi_{U^\perp} [ \Phi^0 + \Psi^* + \eta Q_\omega \phi ] ) U
\ee
solves equation \equ{har map flow}.
The boundary condition
$u=  {\bf e}_3$  amounts to
$$
\Pi_{U^\perp} [ \Phi^0+  \Psi^*  ] +  a(\Pi_{U^\perp} [U+ \Phi^0+  \Psi^*  ]) U  =  ({\bf e}_3 - U)
$$
and then it suffices that we take the boundary condition for \equ{outer1}
\be
  \Psi^*\big|_{\pp\Omega}  =  {\bf e}_3 - U -\Phi^0 .
\label{bcpsi}\ee
Since we want  $u(x,t)$ to be a small perturbation of $U(x,t)$ when we stand close to $(q,T)$,  it is natural to require that $\Psi^*$ satisfies the final condition
\[
\Psi^*\big(q,T)  = 0.
\]
 This constraint amounts to three Lagrange multipliers when we solve the problem, which we choose to put in the initial condition. Then we assume
\[
\Psi^*\big(x,0)  = Z_0^*(x)  +  c_1{\mathbf e_1} +  c_2{\mathbf e_2} + c_3{\mathbf e_3} ,
\]
 where $c_1,c_2,c_3$ are undetermined constants and $Z_0^*(x)$ is a small function for which specific assumptions will later be made.

\section{The reduced equations}
\label{informal}
In this section we will informally discuss the procedure to achieve our purpose in particular deriving the order of vanishing of the scaling parameter $\la(t)$ as  $t\to T$.

The main term that couples equations \equ{inner1} and \equ{outer1} inside the second equation is
the linear expression
\[
Q_\omega[(\Delta_x \eta) \phi + 2  \nn_x \eta \nn_x \phi+ \eta_t  \phi],
\]
which is supported in $|y|= O(R)$.
This motivates the fact that we want $\phi$ to exhibit some type of space decay in $|y|$ since in that way $\Psi^*$ will eventually be smaller and in turn that would make the two equations at main order {\em uncoupled}.
Equation \equ{inner1} has the form
\begin{align*}
\la^2 \phi_t  & =  L_ W  [\phi ] +   h[p,\xi, \Psi^*] (y,t)  \inn \DD_{2R}
\\
\phi\cdot  W  & =  0   \inn \DD_{2R}
\\
\phi(\cdot, 0) & = 0 \inn B_{2R(0)}  ,
\end{align*}
where, for convenience we assume that $h(y,t)$ is defined for all $y\in \R^2$ extending outside $\DD_{2R}$ as
\begin{equation}
\label{HH2}
h[p,\xi, \Psi^*] = \la^2  Q_{-\omega}\ttt L_U  [\Psi^* ] \chi_{\DD_{2R} }
+ \lambda^2  Q_{-\omega} \KK_{0}[p,\xi]
+ \lambda^2  Q_{-\omega}  \KK_{1}[p,\xi]  \chi_{\DD_{2R} }    ,
\end{equation}
where $\chi_A$ designates characteristic function of a set $A$,
$\KK_0$ is defined in \eqref{K01}, \eqref{K02} and $\KK_1$ in \eqref{K1}.
If $\lambda	(t)$ has a relatively smooth vanishing as $t\to T$ it seems natural that the term $\la^2 \phi_t $ be of smaller order and then the equation is
approximately represented by the elliptic problem
\begin{align}
\label{linearized-elliptic}
L_ W  [\phi ] +   h[p,\xi, \Psi^*]=0, \quad \phi\cdot W  =0  \inn \R^2 .
\end{align}

Let us consider the decaying functions $Z_{lj}(y)$ defined in formula \eqref{ZZ}, which satisfy $L_ W [Z_{lj}]=0$.
If $\phi(y,t)$ is a solution of \equ{linearized-elliptic} with sufficient decay, then necessarily
\be\label{ww1}
\int_{\R^2 }   h[p,\xi, \Psi^*](y,t)\cdot Z_{lj} (y)\, dy = 0  \quad \foral t\in (0,T) ,
\ee
for $l=0,1$, $j=1,2$.
These relations amount to an integro-differential system of equations for $p(t)$, $\xi(t)$, which, as a matter of fact, {\em detemine} the correct values of the parameters so that the solution $(\phi,\Psi^*)$ with appropriate asymptotics exists.

\medskip
We derive next useful expressions for relations \equ{ww1}.   Let us first compute the quantities
\begin{align}
\label{defB0j}
\mathcal B_{0j} [p] (t) :=  \frac{\la}{2\pi} \int_{\R^2}   Q_{-\omega} [ \KK_{0}[p,\xi]+ \KK_{1}[p,\xi]] \cdot Z_{0j} (y)\, dy.
\end{align}
Using \eqref{K01}, \eqref{K02} the following expressions for $\mathcal B_{01}$, $\mathcal B_{02}$ are readily obtained:
\begin{align*}
\mathcal B_{01} [p](t)
&=
  \int_{-T} ^t  {\rm Re  } \,(\dot p(s) e^{-i\omega(t)} )\,
\Gamma_1 \left ( \frac {\la(t)^2}{t-s}   \right )  \,\frac{ ds}{t-s}\,  -2 \dot\la (t)
\\
 \mathcal B_{02}[p](t)
& =
 \int_{-T} ^t  {\rm Im  } \,(\dot p(s) e^{-i\omega(t)} )\,
\Gamma_2 \left ( \frac {\la(t)^2}{t-s}   \right )  \,\frac{ ds}{t-s}\,
\end{align*}
where  $\Gamma_j(\tau)$, $j=1,2$  are the smooth functions
defined as follows:
\begin{align*}
\Gamma_1 (\tau)
&
=  - \int_0^{\infty} \rho^3 w^3_\rho \left [  K ( \zeta )
+ 2 \zeta K_\zeta (\zeta ) \frac {\rho^2} { 1+ \rho^2}
-4\cos(w) \zeta^2 K_{\zeta\zeta} (\zeta)
\right ]_{\zeta = \tau(1+\rho^2)}   \, d\rho
\\
\Gamma_2 (\tau) & =
- \int_0^{\infty} \rho^3 w^3_\rho  \left [K(\zeta)   - \zeta^2 K_{\zeta\zeta}(\zeta) \right ]_{\zeta = \tau(1+\rho^2)}
\, d\rho\,
\end{align*}
where
\[
 K(\zeta)  =  2\frac {1- e^{-\frac{\zeta}4}} {\zeta} ,
\]
and we  have used that $\int_0^{\infty} \rho^3 w_\rho^3 d\rho=-2$. Using these expressions we find that
\begin{align}
\label{GammaNear0}
| \Gamma_l (\tau)- 1|
& \le   C \tau(1+  |\log\tau|) \quad  \hbox{ for }\tau<1 ,
\\
\nonumber
|\Gamma_l (\tau)|
& \le  \frac C\tau\qquad \qquad\hbox{ for }\tau> 1, l=1,2.
\end{align}
Let us define
\begin{align}
\label{defB0-new}
\mathcal B_0[p ] :=
\frac{1}{2}e^{i \omega(t) }
\left(
\mathcal B_{01}[p ]
+  i\mathcal B_{02}[p ] \right)
\end{align}
and
\begin{align}
\nonumber
a_{0j}[p,\xi, \Psi^*]
&:=
- \frac{ \la}{2\pi} \int_{B_{2R}}   Q_{-\omega} \ttt L_U  [\Psi^* ]  \cdot Z_{0j} (y)\, dy, j=1,2,
\\
\label{defA0}
a_{0}[p,\xi, \Psi^*]
& :=
\frac{1}{2}
e^{i \omega(t)} \left( a_{01}[p,\xi, \Psi^*] + i a_{02}[p,\xi, \Psi^*] \right).
\end{align}

\noanot{ 
\begin{ch}
\cb
Maybe change to
\begin{align}
\label{defA0-new}
a_{0}[p,\xi, \Psi^*]
=-\frac{\lambda}{4\pi} e^{i \omega(t)}
\int_{B_{2R} }
\left(
Q_{-\omega(t)} \tilde L_U[\Psi^*] \cdot Z_{01}
+ i Q_{-\omega(t)} \tilde L_U[\Psi^*] \cdot Z_{02}
\right)\,dy .
\end{align}
\end{ch}
} 

Similarly, we let
\begin{align*}
\mathcal B_{1j} [\xi ] (t)  & :=  \frac{\la}{2\pi} \int_{\R^2}   Q_{-\omega} [ \KK_{0}[p,\xi]+ \KK_{1}[p,\xi]] \cdot Z_{1j} (y)\, dy, j=1,2, \\
\mathcal B_{1} [\xi ] (t) & :=   \mathcal B_{11}[\xi](t) + i \mathcal B_{12}[\xi](t) .
\end{align*}
Using \equ{K1}, \eqref{ZZ} and the fact that  $\int_0^{\infty} \rho w_\rho^2 d\rho\, =2$ we get
\[
\mathcal B_{1} [\xi ](t)\,  =  \, 2[\, \dot \xi_1(t) + i\dot \xi_2(t)\,]  .
\]
At last, we set
\begin{align*}
a_{1j} [p,\xi, \Psi^* ] &:= \frac{\lambda}{2\pi}
\int_{B_{2R}} Q_{-\omega} \tilde L_U[\Psi^*]\cdot Z_{1j}(y)  \,dy, j=1,2,
\\
a_1[p,\xi, \Psi^* ] & := - e^{i \omega(t) } ( a_{11}[p,\xi, \Psi^* ] + i a_{12} [p,\xi, \Psi^* ] ) .
\end{align*}

We get that
the four conditions \equ{ww1}  reduce to the system of two complex equations
\begin{align}
\mathcal B_0[p ] & = a_0[p,\xi,\Psi^* ],\label{eqB0}\\
\mathcal B_1[\xi ] & =    a_1[p,\xi,\Psi^* ].\label{eqB1}
\end{align}

At this point we will make some preliminary considerations on this system that will allow us to find a first guess of the parameters $p(t)$ and $\xi(t)$.
First, we observe that
\begin{align*}
\mathcal B_0[p ]
=   \int_{-T} ^{t-\la^2}    \frac{\dot p(s)}{t-s}ds\, + O\big( \|\dot p\|_\infty \big).
\end{align*}

To get an approximation for $a_0$, we analyze the operator $\tilde L_U$ in $a_0$. For this let us write
$$
\Psi^* =  \left [ \begin{matrix}\psi^* \\  \psi^*_3  \end{matrix}   \right ] , \quad \psi^* = \psi^*_1 + i \psi^*_2 .
$$
From formula \equ{Ltilde2} we find that
\[
\ttt L_U [\Psi^* ](y)  =  [\ttt L_U]_0 [\Psi^* ]   +  [\ttt L_U]_1 [\Psi^* ]+ [\ttt L_U]_2 [\Psi^* ] ,
\]
 where
\begin{align*}
\la Q_{-\omega} [\ttt L_U]_0 [\Psi^* ]& =   \ \   \rho w_\rho^2\, \big [\, \div ( e^{-i\omega} \psi^*)\,  E_1   + \curl ( e^{-i\omega} \psi^* )\, E_2
  \, \big ]\,
 \\
\la Q_{-\omega}[\ttt L_U]_1 [\Psi^* ] & =
   -\,
  2 w_\rho  \cos w \,  \big [\,(\pp_{x_1} \psi^*_3) \cos \theta +    (\pp_{x_2} \psi^*_3) \sin  \theta \, \big ]\, E_1
\\
& \quad - 2 w_\rho  \cos w  \, \big [\, (\pp_{x_1} \psi^*_3) \sin \theta -    (\pp_{x_2} \psi^*_3) \cos  \theta \, \big ]\,  E_2\ ,
\\
 \la Q_{-\omega}[\ttt L_U]_2 [\Psi^* ] &=   \quad
 \rho w_\rho^2 \,  \big [\, \div (e^{i\omega}\bar \psi^*)\, \cos 2\theta  -   \curl ( e^{i\omega}\bar \psi^*)\, \sin 2\theta  \, \big ]\,  E_1
\\
&  \quad
+   \rho w_\rho^2 \, \big [\,  \div ( e^{i\omega}\bar \psi^*)\, \sin 2\theta +   \curl  ( e^{i\omega}\bar \psi^*)\,  \cos 2\theta   \, \big ]\,E_2 ,
\end{align*}
and the differential operators in $\Psi^*$  on the right hand sides
are evaluated at $(x,t)$ with   $x= \xi(t)+ \la(t) y$,  $y = \rho e^{i\theta}$ while $E_l= E_l(y)$, $l=1,2$.

\medskip
From the above decomposition, assuming that $\Psi^*$ is of class $C^1$ in space variable, we find that
\[
a_{0}[p,\xi, \Psi^*] =   [   \div \psi^*+  i\curl \psi^*](\xi,t )  + o(1) ,
\]
where $o(1)\to 0$ as $t\to T$.

Similarly, we have that
\begin{align*}
a_1(p,\xi) & =   2  ( \pp_{x_1} \psi^*_3 + i\pp_{x_2} \psi^*_3) (\xi, t) \int_{0}^\infty \cos w \,w_\rho^2 \rho \, d\rho  + o(1)
\\
&  =  o(1) \ass t\to T,
\end{align*}
since  $\int_0^\infty w_\rho^2 \cos w \rho \, d\rho = 0 $.

 \medskip

Let us discuss informally how to handle  \equ{eqB0}-\equ{eqB1}.
For this we simplify this system in the form
\begin{align}
\nonumber
\int_{-T} ^{t-\la^2}     \frac{\dot p(s)}{t-s}ds
& =
[ \div \psi^*+  i\curl \psi^*](\xi(t),t )  + o(1) + O(\|\dot p\|_\infty)  \\
\dot \xi(t) &  =  o(1)\ass t\to T. \label{equB1}
 \end{align}

\medskip
We assume for the moment that the function $\Psi^*(x,t)$ is fixed, sufficiently regular, and we regard $T$ as a parameter that will always be taken smaller if necessary. We recall that we want $\xi(T)=q$ where $q\in \Omega$ is given, and $\la(T)=0$.  Equation  \equ{equB1} immediately suggests us
to take  $\xi(t) \equiv q$ as a first approximation.
Neglecting lower order terms, we arrive at the ``clean'' equation for $p(t)= \lambda (t) e^{i\omega(t)}$,
\begin{align}
\label{kuj}
\int_{-T} ^{t-\la(t)^2} \frac{ \dot p(s)}{t-s}ds   =
\div \psi^*(q,0 ) + i\curl \psi^*(q,0 ) =: a_0^*
\end{align}
At this point we make the following assumption:
\begin{align}
\label{negativeDiv}
\div \psi^*(q,0 ) <0.
\end{align}
This implies that
$a_0^*  =  -|a_0^*| e^{i\omega_0}$ for a unique $\omega_0\in (-\frac \pi2 , \frac \pi 2)$. Let us take $\omega(t)\equiv \omega_0$. Then equation \equ{kuj} becomes
\be
 \int_{-T} ^{t-\la^2} \frac{ \dot \la(s)}{t-s}ds   =
 - |a_0^*| .
\label{cccc4}\ee
 We claim that a good approximate solution of \equ{cccc4} as $t\to T$  is given by
\[
\dot \la(t) =  -\frac {\kappa} {\log^2(T-t)}
\]
for a suitable $\kappa>0$. In fact, substituting, we have
\begin{align}
\int_{-T}^{t-\la^2} \frac {\dot\la(s)}{t-s}\, ds \,  =& \
\int_{-T}^{t-  (T-t)  } \frac{ \dot\la  (s)}{t-s} \, ds +     \, \dot \la (t)\left [ \log (T-t)   - 2\log (\la(t)) \right ]\nonumber \\  + &  \int_{t-(T-t)   } ^{ t- \la(t)^2}\frac{\dot \la(s)-\dot \la(t)}{t-s} ds    \nonumber \\
  \approx & \
\int_{-T}^{t } \frac{ \dot\la  (s)}{T-s}\, ds\,  - \, \dot \la (t) \log (T-t) \, =: \beta(t)
\label{formal}
\end{align}
as $t\to T$. We see that
$$
\log(T-t) \frac {d\beta} {dt}(t)  =
   \frac d{dt} (\log^2(T-t) \, \dot\la(t))= 0
$$
from the explicit form of  $\dot\la(t)$. Hence $\beta(t)$ is constant. As a conclusion, equation \equ{cccc4}
is approximately satisfied if $\kappa$ is such that
$$
\kappa \int_{-T}^{T} \frac{ \dot\la  (s)}{T-s}\ =\ -|a_0^*| .
$$
And this finally gives us the approximate expression
$$
\dot\la (t)= -  |\div \psi^*(q,0) + i \curl\psi^* (q,0) |\, \dot \la_* (t) ,
$$
where
\[
 \dot \la_* (t) = -\frac { |\log T|}{\log^2(T-t)}.
\]
Naturally imposing $\la_*(T) =0$ we then have
\[
 \la_* (t) =  \frac { |\log T|}{\log^2(T-t)}(T-t)\, (1+ o(1)) \ass t\to T.
\]


%
%

\section{Solving the inner-outer gluing system}

\label{sectInnerOuter}

Our purpose is to determine, for a given $q\in \Omega$  and a sufficiently small $T>0$, a solution $(\phi,\Psi^*)$ of system \equ{inner1}-\equ{outer1} with a boundary condition of the form
\equ{bcpsi} such that  $u(x,t)$ given by \equ{upa} blows up with $U(x,t)$ as its main order profile.
This will only be possible for adequate choices
of the  parameter functions $\xi(t)$ and $p(t)= \la(t) e^{i\omega(t)}$.
These functions
will eventually be found by fixed point arguments, but a priori we need to make some assumptions regarding their behavior.
For some  positive numbers $a_1,a_2,\sigma$  independent of $T$ we will assume that
\begin{align}
a_1 |\dot \la_* (t)| \le    |\dot p (t)| &  \le    a_2 |\dot \la_* (t)| \foral t\in (0,T),
\label{cotin}
\\
 |\dot \xi(t) |  & \le   \la_*(t)^{\sigma}\quad\ \foral t\in (0,T).
\label{cotin1}
\end{align}
We also take
\begin{align}
\label{RRR}
R(t)  =\la_*(t)^{-\beta},
\end{align}
where  $\beta \in ( 0,\frac 12)$.

\medskip
To solve the outer equation \eqref{outer1} we will decompose $\Psi^*$ in the form
\[
\Psi^* =  Z^* + \psi
\]
where we let $Z^*:\Omega\times (0,\infty) \to \R^3$ satisfy
\eqref{heatZ*}
with $Z_0^*(x)$  a function satisfying certain conditions to be described below.
Since we would like that  $u(x,t)$ given by \equ{upa} has a blow-up behavior given at main order by that of $U(x,t)$, we will require
\[
\Psi^*(q,T)=0 .
\]
This constraint has three parameters. Therefore we need three ``Lagrange multipliers'' which we include in the initial datum.

\subsection{Assumptions on  \texorpdfstring{$Z_0^*$}{Z0star}}\label{asss}

To describe the assumptions on $Z_0^*$, let us write
\begin{align}
\nonumber
Z_0^*(x) =  \left [ \begin{matrix}z_0^*(x) \\  z_{03}^* (x) \end{matrix}   \right ] , \quad z_0^*(x) = z^*_{01}(x)  + i z^*_{02}(x)  .
\end{align}
A first condition that we require, consistent with \eqref{negativeDiv},  is
$ \div  z^*_0(q) <0 $.
In addition we require that $Z_0^*(q)\approx 0 $ in a non-degenerate way.

We want also $Z^*$ to be sufficiently small, but independently of $T$, so that the heat equation \eqref{heatZ*} is a good approximation of the linearized harmonic map flow far from the singularity.
In order to achieve later the desired stability property, it is convenient to split $Z_0^*$ into two parts
\begin{align*}
Z_0^* = Z_0^{*0} + Z_0^{*1} ,
\end{align*}
where $Z_0^{*0}$ is sufficiently smooth and $ Z_0^{*1}$ allows more irregular perturbations.
More precisely, for $Z_0^{*0}$ we assume that
for some $\alpha_0>0$ small and some $\alpha_1, \alpha_2>0$, all independent of $T$, we have
\begin{align}
\label{conditionsZ0a}
\left\{
\begin{aligned}
\|Z_0^{*0}\|_{C^3(\overline \Omega)}  & \le  \alpha_0, \\
|Z_0^{*0}(q)|  & \le  5T,  \\
|(Dz_0^{*0}(q))^{-1}|  & \le \alpha_1 ,  \\
-\alpha_1 & \leq \div z_0^{*0}(q)   \le   - \alpha_2 .
\end{aligned}
\right.
\end{align}
(The notation here is analogous to \eqref{notationZ0star}.)

To describe $Z_0^{*1}$ we introduce the following norm
\begin{align}
\label{normZ0}
\| Z_0^{*1} \|_* & =
\sup_\Omega |Z_0^{*1}(x) |
+  \frac{1}{|\log \varepsilon_*|} \sup_\Omega | \nabla_x Z_0^{*1} (x)|
\\
\nonumber
& \quad
+  \frac{1}{|\log \varepsilon_*|^{1/2}} \sup_\Omega \left(  |x-q_0| +\varepsilon_* \right)  |D^2_x Z_0^{*1}(x)| ,
\end{align}
where
\begin{align}
\label{defEpsilon}
\varepsilon_*  = \lambda_*(0) .
\end{align}
Then we assume that for some $\sigma>0$ fixed we have
\begin{align}
\label{conditionsZ0b}
\|Z_0^{*1} \|_* \leq T^\sigma.
\end{align}

In summary, the conditions on $Z_0^*$ are the following:
\begin{align}
\label{condZ0}
\text{$Z_0^*=Z_0^{*0} + Z_0^{*1}$ with $Z_0^{*0}$, $Z_0^{*1}$  satisfying \eqref{conditionsZ0a} and \eqref{conditionsZ0b}.}
\end{align}

\medskip
\noanot{ 
\crr

OLD ASSUMPTIONS

More precisely, we consider positive numbers $\alpha_0$, $\alpha_1$, $\alpha_2$, all of them  independent of $T$, with $\alpha_0$ sufficiently small
\begin{align}
\label{conditionsZ0}
\left\{
\begin{aligned}
\|Z_0^*\|_{C^1(\Omega)}  +  T|\log T|^{-1} \|D^2 Z_0^*\|_{L^\infty (\Omega)}
& \le  \alpha_0,
\\
|Z_0^*(q)|
& \le  5T,
\\
|(Dz_0(q))^{-1}|
& \le -\alpha_1,
\\
\div z_0^*(q)  & \le   - \alpha_2.
\end{aligned}
\right.
\end{align}
}  

\medskip

\subsection{Linear theory for the inner problem}
The inner problem \equ{inner1} is written as
\begin{align*}
\left\{
\begin{aligned}
\la^2 \pp_t \phi  & =  L_ W  [\phi ] +   h[p,\xi, \Psi^*]   \inn \DD_{2R}
\\
\phi  \cdot  W  & =  0   \inn \DD_{2R}
\\
\phi (\cdot, 0) & = 0 \inn B_{2R(0)}
\end{aligned}
\right.
\end{align*}
where  $h[p,\xi, \Psi^*] $ is given by \equ{HH2}.
To find a good solution to this problem we would like that $ h[p,\xi, \Psi^*] $ satisfies the orthogonality conditions \eqref{ww1}.

We split  the right hand side $h[p,\xi, \Psi^*] $  and the inner solution  into components with different roles regarding these orthogonality conditions.

Recall that
\begin{equation*}
h[p,\xi, \Psi^*] = \lambda^2  Q_{-\omega} \tilde L_U  [\Psi^* ] \chi_{\DD_{2R} }
+ \lambda^2  Q_{-\omega} \KK_{0}[p,\xi]
+ \lambda^2  Q_{-\omega}  \KK_{1}[p,\xi]  \chi_{\DD_{2R} }    ,
\end{equation*}
the decomposition of $\tilde L_U$ given in \eqref{Ltilde2}:
\begin{align}
\nonumber
\ttt L_U [\Psi^* ]
=  \tilde L_U [\Psi^* ]_0  +  \tilde L_U [ \Psi^* ]_1 +  \tilde L_U [ \Psi^* ]_2\ ,
\end{align}
with $\tilde L_U[\Phi]_j$ defined in \eqref{Ltilde-j}.
Using the notation \eqref{notation-Phi}, we then define
\begin{align}
\nonumber
\tilde L_U [\Phi ]_1^{(0)} & =
   -\,
  2\la^{-1} w_\rho  \cos w \, \big [\,(\partial_{x_1} \varphi_3(\xi(t),t)) \cos \theta +    (\partial_{x_2} \varphi_3(\xi(t),t))) \sin  \theta \, \big ]\,Q_{\omega} E_1
\\
\nonumber
& \quad - 2\la^{-1}  w_\rho  \cos w  \, \big [\, (\partial_{x_1} \varphi_3(\xi(t),t))) \sin \theta -    (\partial_{x_2} \varphi_3(\xi(t),t))) \cos  \theta \, \big ]\, Q_{\omega}E_2\ .
\end{align}
We then decompose
\[
h = h_1+ h_2 + h_3
\]
where
\begin{align}
\nonumber
h_1[p,\xi, \Psi^*]
&=
\lambda^2  Q_{-\omega} (
\tilde L_U  [\Psi^* ]_0
+\tilde L_U  [\Psi^* ]_2 ) \chi_{\DD_{2R} }
+ \lambda^2  Q_{-\omega} \KK_{0}[p,\xi] ,
\\
\nonumber
h_2[p,\xi, \Psi^*] &=
\lambda^2  Q_{-\omega} \tilde L_U  [\Psi^* ]_1^{(0)} \chi_{\DD_{2R} }
+ \lambda^2  Q_{-\omega}  \KK_{1}[p,\xi]  \chi_{\DD_{2R} },
\\
\nonumber
h_3[p,\xi, \Psi^*] &=   \lambda^2  Q_{-\omega}
( \tilde L_U [\Psi^* ]_1 - \tilde L_U  [\Psi^* ]_1 ^{(0)})  \chi_{\DD_{2R} }     .
\end{align}

Next we decompose  $\phi = \phi_1+ \phi_2 + \phi_3+\phi_4$. The function $\phi_1$ will solve the inner problem with right hand side $  h_1[p,\xi, \Psi^*]   $ projected so that it satisfies essentially \eqref{ww1}.
The advantage of doing this is that $h_1$ has faster spatial decay, which gives better bounds for the solution.
For this we let, for any function $h(y,t)$ defined in $\R^2 \times (0,T)$ with sufficient decay,
\begin{align}
\label{defCij}
c_{lj}[h](t)  :=     \frac 1 { \int_{\R^2} w_\rho^2 |Z_{lj}|^2  } \int_{\R^2} h (y  ,t)\cdot Z_{lj}(y)\, dy .
\end{align}
Note that  $h[p,\xi, \Psi^*] $  is defined in $\R^2\times (0,T)$, and for simplicity we will assume that the right hand sides appearing in the different linear equations are always defined  in $\R^2\times (0,T)$.

We would like that  $\phi_1$ solves
\begin{align*}
\lambda^2 \pp_t \phi_1  &=  L_ W  [\phi_1 ] +   h_1[p,\xi, \Psi^*]
- \sum_{l=-1}^1
\sum_{j=1}^2 c_{lj}[h_1(p,\xi, \Psi^*)] w_\rho^2 Z_{lj}  \inn \DD_{2R} ,
\end{align*}
but the estimates for $\phi_1$ are better if the projections $c_{0j}[h(p,\xi, \Psi^*)] $ are modified slightly.

Here is the precise result that we will use later.
We define the norms
\begin{align}
\label{norm-h}
\|h\|_{\nu,a}  =
\sup_{\R^2 \times (0,T)} \  \frac{ |h(y,t)| }{ \lambda_*^\nu (1+|y|)^{-a}} ,
\end{align}
and
\begin{align}
\label{norm-phi1}
\| \phi \|_{*,\nu,a,\delta}
=
\sup_{\DD_{2R}}
\frac{| \phi(y,t) | + (1+|y|) |\nabla_y \phi(y,t)|}{ \lambda_*^\nu \max(\frac{R^{\delta(5-a)}}{(1+|y|)^3} , \frac{1}{(1+|y|)^{a-2} })} .
\end{align}

\begin{prop}
\label{prop1.0}
Let $a \in (2,3)$, $\delta \in (0,1)$, $\nu>0$.
Assume $\| h \|_{\nu,a}<\infty$. Then there is a solution $\phi =  \TT_{\lambda,1} [h]$,   $\tilde c_{0j}[h]$  of
\begin{align}
\nonumber
\left\{
\begin{aligned}
\lambda^2 \partial_t \phi
& =  L_ W  [\phi ] +   h
-  \sum_{  j=1,2} \tilde c_{0j}[h] Z_{0j} \chi_{B_1}
-  \sum_{ \substack{l=-1,1\\ j=1,2}} c_{lj}[h] Z_{lj} \chi_{B_1}
\quad \text{in } \DD_{2R}
\\
\phi\cdot  W  & =  0   \quad \text{in } \DD_{2R}
\\
\phi(\cdot, 0) & = 0 \quad \text{in } B_{2R(0)}
\end{aligned}
\right.
\end{align}
where $c_{lj}$ is defined in \eqref{defCij}, which is  linear in $h$, such that
\begin{align*}
\| \phi \|_{*,\nu,a,\delta}
\leq C \|h\|_{\nu,a}
\end{align*}
and such that
\begin{align}
\nonumber
|c_{0j}[h]  - \tilde c_{0j}[h]| \leq
C  \lambda_*^{\nu}  R^{-\frac{1}{2}\delta(a-2)} \| h \|_{\nu,a}.
\end{align}
\end{prop}

\medskip

The function $\phi_2$ solves the equation with right hand side  $h_2[p,\xi,\Psi^*]$, which is in {\em mode 1}, a notion that we define next (this is basically motivated by the analysis of section~\ref{sectLinearTheory}, where we consider the linearized parabolic equation and use a Fourier decomposition of the right hand side and the solution).

Let $h(y,t)\in \R^3$, be defined in $\R^2 \times (0,T)$ or $\DD_{2R}$ with $h\cdot W = 0$. We say that $h$ is a mode $k\in \Z$ if $h$ has the form
\[
h(y,t)= \Re ( \tilde h_k(|y|,t) e^{ik\theta}) E_1 + \Re ( \tilde h_k(|y|,t) e^{ik\theta}) E_2 ,
\]
for some complex valued function $\tilde h_k(\rho,t)$.

Consider then
\begin{align}
\label{1.11-mode1}
\left\{
\begin{aligned}
\lambda^2 \partial_t \phi
& =  L_ W  [\phi ] +   h
-  \sum_{  j=1,2} c_{1j}[h] w_\rho^2 Z_{1j}
\quad \text{in } \DD_{2R}
\\
\phi\cdot  W  & =  0   \quad \text{in } \DD_{2R}
\\
\phi(\cdot, 0) & = 0 \quad \text{in } B_{2R(0)}
\end{aligned}
\right.
\end{align}

\begin{prop}
Let $a \in (2,3)$, $\delta \in (0,1)$, $\nu>0$.
Assume that $h$ is in mode 1 and $\| h \|_{\nu,a}<\infty$. Then there is a solution $\phi  =\TT_{\lambda,2} [h]$  of \eqref{1.11-mode1}, which is  linear in $h$, such that
\begin{align*}
\| \phi \|_{\nu,a-2}
\leq C \|h\|_{\nu,a} .
\end{align*}
\end{prop}
In the above statemen the norm $\|\phi\|_{\nu,a-2}$ analogous to the one in \eqref{norm-h}, but the supremum is taken in $\DD_{2R}$.

Another piece of the inner solution, $\phi_3$, will handle  $h_3[p,\xi,\Psi^*]$, which does not satisfy  orthogonality conditions in mode 0.
We will still project it to satisfy the orthogonality condition in mode 1.
Let us consider then \eqref{1.11-mode1} without any orthogonality conditions on $h$ in mode 0.
We define
\begin{align}
\label{norm-starstar}
\|\phi\|_{**,\nu}
= \sup_{\DD_{2R}} \
\frac{ |\phi(y,t)| + (1+|y|)\left |\nn_y \phi(y,t)\right | }
{ \la_*(t)^{\nu}  R(t)^{2} ( 1+|y| )^{-1}  } .
\end{align}

\begin{prop}
\label{prop02}
Let  $1<a<3$ and $\nu>0$. There exists a $C>0$ such that if  $\|h\|_{a,\nu} <+\infty$ there is a solution $ \phi = \TT_{\lambda,3} [h]$ of \eqref{1.11-mode1}, which is linear in $h$ and satisfies the estimate
$$
\|\phi\|_{**,\nu} \ \le\ C \|h\|_{a,\nu} .
$$
\end{prop}

Note that we allow $a$ to be less than 2 in the previous proposition.

Next we have a variant of Proposition~\ref{prop02} when  $h$ is in mode -1.

\begin{prop}
Let  $2<a<3$ and $\nu>0$. There exists a $C>0$ such that for any $h$ in mode -1   with  $\|h\|_{a,\nu} <+\infty$, there is a solution $\phi = \TT_{\lambda,4} [h]$ of problem \eqref{1.11-mode1}, which is linear in $h$ and satisfies the estimate
$$
\|\phi\|_{***,\nu}  \leq C \|h\|_{a,\nu}  ,
$$
where
\begin{align}
\nonumber
\|\phi\|_{***,\nu}
= \sup_{\DD_{2R}} \
\frac{ |\phi(y,t)| + (1+|y|)\left |\nn_y \phi(y,t)\right | }
{ \la_*(t)^{\nu}  \log(R(t))  } .
\end{align}
\end{prop}

All propositions stated here are corollaries of Proposition~\ref{prop2} and proved in section~\ref{sectLinearTheory}.

\subsection{The equations for \texorpdfstring{$p = \lambda e^{i\omega}$}{}}
We need to choose the free parameters $p$, $\xi$ so that  $c_{lj}[h(p,\xi, \Psi^*)]=0$ for $l=-1,0,1$, $j=1,2$. This will be easy to do for $l=1$ (mode 1), but mode $l=0$ is more complicated.


To handle $c_{0j}$ we note that by definitions \eqref{HH2}, \eqref{defB0j}, \eqref{defA0}
\begin{align*}
c_{0,j}[h(p,\xi,\Psi^*)]  = \frac{2\pi \lambda }{\int_{\R^2} w_\rho^2 |Z_{0j}|^2}\left(
\mathcal B_{0j}[p] - a_{0j}[p,\xi,\Psi^*]
\right)
\end{align*}
where $B_0$, $a_0$ are defined in \eqref{defB0-new}, \eqref{defA0} and we recall that $p = \lambda e^{i \omega}$.
\noanot{ 
\cb
By the definition of $h[ p,\xi,\Psi^*]$ \eqref{HH2}
\[
h[p,\xi, \Psi^*] = \la^2  Q_{-\omega}\ttt L_U  [\Psi^* ] \chi_{\DD_{2R} }   + \la^2  Q_{-\omega} [ \KK_{0}[p,\xi]+ \KK_{1}[p,\xi]] ,
\]
and then
\begin{align*}
c_{0j} [  h[p,\xi, \Psi^*]   ]
&=  \frac 1 { \int_{\R^2} w_\rho^2 |Z_{lj}|^2  } \int_{\R^2} h[p,\xi, \Psi^*]   \cdot Z_{0j}(y)\, dy
\\
&=
\frac{\lambda^2}{ \int_{\R^2} w_\rho^2 |Z_{lj}|^2  }
\int_{B_{2R}}    Q_{-\omega}\ttt L_U  [\Psi^* ] \cdot Z_{0j}(y)\, dy
\\
& \quad +
\frac{\lambda^2}{ \int_{\R^2} w_\rho^2 |Z_{lj}|^2  }
\int_{\R^2}Q_{-\omega} [ \KK_{0}[p,\xi]+ \KK_{1}[p,\xi]] \cdot Z_{0j}(y)\, dy
\\
&=\frac{2\pi \lambda }{\int_{\R^2} w_\rho^2 |Z_{0j}|^2}\left(
\mathcal B_{0j}[p] - a_{0j}[p,\xi,\Psi^*] \right).
\end{align*}
} 

So  to achieve $c_{0j}[h(p,\xi, \Psi^*)]=0$ we should solve
\begin{align}
\label{eqAbc}
\mathcal B_0[p ](t)  = a_0[p,\xi,\Psi^* ](t), \quad t\in [0,T],
\end{align}
adjusting   the parameters $\lambda(t)$ and $\omega(t)$.
This equation is delicate and we will instead impose a modified version of this condition.
The modification of \eqref{eqAbc} consists in introducing another term in the equation, essentially modifying the operator $\mathcal B_0$.

To make this precise we define the following norms.
Let $I$ denote either the interval $[0,T]$ or $[-T,T]$.
For  $\Theta\in (0,1)$, $l\in \R$ and a continuous function $g:I\to \C$ we let
\begin{align}
\label{normG0}
\|g\|_{\Theta,l} = \sup_{t\in I} \, (T-t)^{-\Theta} |\log(T-t)|^{l} |g(t)| ,
\end{align}
and for  $\gamma \in (0,1)$, $m \in (0,\infty) $, and $l \in \R$ we let
\begin{align}
\label{normG1}
[ g]_{\gamma,m,l} = \sup \, (T-t)^{-m}  |\log(T-t)|^{l} \frac{|g(t)-g(s)|}{(t-s)^\gamma} ,
\end{align}
where the supremum is taken over $s \leq t$ in $ I$  such that $t-s \leq \frac{1}{10}(T-t)$.

We have then the following result, whose proof is in section~\ref{secLambda}.

\begin{prop}
\label{propIntegralOp}
Let $\alpha ,  \gamma \in (0,\frac{1}{2})$, $l\in \R$, $C_1>1$.
There is $\alpha_0>0$ such that if $\Theta \in (0,\alpha_0)$ and
$m \leq \Theta - \gamma$,
then for  $a:[0,T]\to \C$ is such that
\begin{align}
\label{hypA00}
\left\{
\begin{aligned}
& \frac{1}{C_1} \leq | a(T) | \leq C_1 ,
\\
& T^\Theta |\log T|^{1+\sigma-l} \| a(\cdot) - a(T) \|_{\Theta,l-1}
+ [a]_{\gamma,m,l-1}
\leq C_1 ,
\end{aligned}
\right.
\end{align}
for some $\sigma>0$,
then, for $T>0$ small enough there are two operators $\mathcal P $ and $\Rem$ so that $p = \mathcal P[a]: [-T,T]\to \C$ satisfies
\begin{align}
\label{eq-modified0}
\mathcal B_0[p](t)
= a(t) + \Rem[a](t) , \quad t \in [0,T],
\end{align}
with
\begin{align}
\nonumber
& |\Rem[a](t) |
\\
\label{ineqL1b-1}
& \leq  C
\Bigl( T^{\sigma}
+ T^\Theta  \frac{\log |\log T|}{|\log T|}  \| a(\cdot) - a(T) \|_{\Theta,l-1}
+ [a]_{\gamma,m,l-1} \Bigr)
\frac{(T-t)^{m+(1+\alpha ) \gamma}}{  |\log(T-t)|^{l}} ,
\end{align}
for some $\sigma>0$.
\end{prop}

We have additional properties of the solution to this problem.

\begin{prop}
Let us make the same assumptions as in Proposition~\ref{propIntegralOp}.
Then $\mathcal P[a]$ can be written as
\[
\mathcal P[a] = p_{0,\kappa[a]} + \mathcal P_1[a] + \mathcal P_2[a]
\]
where $p_{0,\kappa}$ is defined in \eqref{p0kappa} and each
term
\[
\kappa = \kappa[a] ,
\quad
p_1 = \mathcal P_1[a] ,
\quad
p_2 = \mathcal P_2[a] ,
\]
has the following bounds:
\begin{align*}
\kappa & = |a(T)|( 1 + O(\frac{1}{|\log T|}),
\\
|\dot p_1(t) - \dot \p_{0,\kappa}(t)|  & \leq  C \frac{  |\log T |^{1-\sigma} \log(|\log T |)^2 }{|\log(T-t)|^{3-\sigma}},
\\
|\ddot\p_1(t)| &\leq  C \frac{  |\log T| }{|\log(T-t)|^{3} (T-t)} , \\
\\
\| \dot\p_2 \|_{\Theta,l}  & \leq C \big(T^{\frac{1}{2}+\sigma-\Theta} +  \| a(\cdot) - a(T) \|_{\Theta,l-1} ),
\\
[ \dot \p_2 ]_{\gamma,m	,l}
& \leq C(
|\log T|^{l-3} T^{\alpha_0-m-\gamma}
+ T^\Theta \frac{\log |\log T|}{|\log T|}\| a(\cdot) - a(T) \|_{\Theta,l-1}
+ [ a]_{\gamma,m,l-1}) ,
\end{align*}
where $\alpha_0>0$ is some fixed some constant and $\sigma>0$ is arbitrary (with $C$ depending on $\sigma$).
%
%
\end{prop}

%
%
%

Roughly speaking, to obtain the modified equation \eqref{eq-modified0} we notice that the main term in $p$ in $ \mathcal B_0[p ] $ is the integral operator
\[
\int_{-T} ^{t-\la_*(t)^2}     \frac{\dot p(s)}{t-s}ds.
\]
Thus we define
\[
\tilde{\mathcal B}_0[p ]  =  \mathcal B_0[p ] - \int_{-T} ^{t-\la_*(t)^2}     \frac{\dot p(s)}{t-s}ds.
\]
It will be sufficient  to solve approximately equations \eqref{ww1} replacing in part this integral operator
 by a ``regularized'' version of it following the logic of the formal derivation of the rate \equ{formal}.
For  $\alpha >0 $ let us write
\begin{align*}
\int_{-T} ^{t-\la_*(t)^2}     \frac{\dot p(s)}{t-s}ds
=  S_\alpha  [ \dot p]  + R_\alpha  [\dot p]
\end{align*}
where
\begin{align}
\label{defSalpha}
S_\alpha  [ g]
& :=
g(t) [ - 2\log \lambda_*(t) + (1+\alpha  ) \log (T-t)]
+
\int_{-T} ^{t-(T-t)^{1+\alpha }}  \frac{g(s)}{t-s}ds  ,
\\
\label{defRem}
R_{\alpha }[g]
& :=
-\int_{t-(T-t)^{1+\alpha } } ^{t-\la_*^2}  \frac  {g(t) -g(s)}{t-s} ds.
\end{align}
Thus equation \eqref{eqAbc} can be written in the form
\[
S_{\alpha}  [\dot p]  +
R_{\alpha}[\dot p]  +
 \tilde{\mathcal B}_0[p ]  = a(t), \quad \text{in } [0,T],
\]
for some function $a(t)$.
The modified equation is
\[
S_{\alpha}  [ \dot p]   +
 \tilde{\mathcal B}_0[p ]  = a(t) \quad \text{in } [0,T],
\]
and the remainder $\Rem$ is essentially $R_{\alpha}[\dot p]$.
This is a sketch of how we obtain the modified equation and remainder.  For more details see section~\ref{secLambda}.

Another modification to equations \eqref{eqAbc} that we introduce is to replace $a_0[p,\xi,\Psi^*]$ by its main term.
To do this we write
\begin{align*}
a_0[p,\xi,\Psi]
=a_0^{(0)}[p,\xi,\Psi] + a_0^{(1)}[p,\xi,\Psi] + a_0^{(2)}[p,\xi,\Psi]
\end{align*}
where
\begin{align}
\nonumber
a_0^{(l)}[p,\xi,\Psi]
= -\frac{\lambda}{4\pi}
e^{i\omega}
\int_{B_{2R}}
\left(
Q_{-\omega}\tilde L_U[\Psi]_l \cdot Z_{01} +
i Q_{-\omega}\tilde L_U[\Psi]_l \cdot Z_{02}
\right)\,dy
\end{align}
for $l=0,1,2$.

We define
\begin{align}
\nonumber
c_0^*[p,\xi,\Psi^*](t)
& :=
\frac{ 4 \pi \lambda}{\int_{\R^2} w_\rho^2 |Z_{01}|^2  }
e^{- i\omega}
\Bigl(
\Rem\left[  a_0^{(0)}[p,\xi,\Psi^*]  \right](t)
+a_0^{(1)}[p,\xi,\Psi^*](t)
\\
\nonumber
& \qquad
+ a_0^{(2)}[p,\xi,\Psi^*](t)
\Bigr)  - (c_0[  h[ p,\xi,\Psi^* ]]-\tilde c_0[  h_1[ p,\xi,\Psi^* ] ] ) ,
\end{align}
and
\begin{align*}
c_{01}^* :=  \Re(c_{0}^*) , \quad
c_{02}^* :=
\Im(c_{0}^*) ,
\end{align*}
where  $\Rem$ is the operator given Proposition~\ref{propIntegralOp} and $\tilde c_0 = \tilde c_{01} + i \tilde c_{02}$ are the operators defined in Proposition~\ref{prop1.0}.

\subsection{The system of equations}

We transform the system \eqref{inner1}-\eqref{outer1} in the problem of finding functions $\psi(x,t)$,  $\phi_1,\ldots,\phi_4$,  parameters $p(t) = \la(t)e^{i\omega(t)} $, $\xi(t)$ and constants $c_1,c_2,c_3$
such that the following system is satisfied:
\begin{align}
\label{eq-psi}
\left\{
\begin{aligned}
\psi_t &=  \Delta_x \psi  +  g(p,\xi, Z^*+ \psi,\phi_1+ \phi_2+\phi_3+\phi_4)
\inn \Omega \times (0,T)
\\
\psi &= ({\bf e}_3 - U) -\Phi^0 \qquad\qquad\ \  \onn \pp\Omega\times (0,T)
\\
\psi(\cdot ,0)
&=  (c_1 \,  \mathbf{e_1}  + c_2 \, \mathbf{e_2}  + c_3\,  \mathbf{e_3})\chi +
\ (1-\chi) ({\bf e}_3 - U -\Phi^0)       \inn  \Omega
\\
\psi(q,T) &  =  - Z^* (q,T)
\end{aligned}
\right.
\end{align}
\begin{align}
\left\{
\begin{aligned}
\lambda^2  \partial_t \phi_1
&= L_W [\phi_1] + h_1[p,\xi, \Psi^*]
-  \sum_{  j=1,2} \tilde c_{0j}[ h_1[p,\xi, \Psi^*]  ] w_\rho^2 Z_{0j}
\\
\label{eqphi1}
& \qquad
-  \sum_{ \substack{l=-1,1\\ j=1,2}} c_{lj}[ h_1[p,\xi, \Psi^*]  ] w_\rho^2  Z_{lj}
\inn \DD_{2R}
\\
\phi_1\cdot  W  &=  0   \inn \DD_{2R}
\\
\phi_1(\cdot, 0) &=0 \inn B_{2R(0)}
\end{aligned}
\right.
\end{align}

\begin{align}
\left\{
\begin{aligned}
\label{eqphi2}
\lambda^2 \partial_t \phi_2
&= L_W [\phi _2]  + h_2[p,\xi, \Psi^*]
 -  \sum_{  j=1,2} c_{1j}[  h_2[p,\xi, \Psi^*]    ] w_\rho^2  Z_{1j}    \inn \DD_{2R}
\\
\phi_2\cdot  W  & = 0   \inn \DD_{2R}
\\
\phi_2(\cdot, 0) & =0 \inn B_{2R(0)}
\end{aligned}
\right.
\end{align}

\begin{align}
\label{eqphi3}
\left\{
\begin{aligned}
\lambda^2  \partial_t  \phi_3 &= L_W [\phi_3] +
h_3  -  \sum_{  j=1,2} c_{1j}[  h_3[p,\xi, \Psi^*]    ] w_\rho^2 Z_{1j}
\\
& \quad
+ \sum_{j=1,2} c_{0j}^*[p,\xi,\Psi^*] w_\rho^2 Z_{0j} \inn \DD_{2R}
\\
\phi_3\cdot  W  & = 0   \inn \DD_{2R}
\\
\phi_3(\cdot, 0) & =0 \inn B_{2R(0)}
\end{aligned}
\right.
\end{align}

\begin{align}
\label{eqphi4}
\left\{
\begin{aligned}
\lambda^2  \partial_t  \phi_4 &= L_W [\phi_4 ]
+  \sum_{j=1,2}  c_{-1,j}[   h_1[p,\xi, \Psi^*]   ] w_\rho^2 Z_{-1j}
\\
\phi_4\cdot  W  & = 0   \inn \DD_{2R}
\\
\phi_4(\cdot, t) & = 0 \onn \pp B_{2R(t)}
\\
\phi_4(\cdot, 0) & =0 \inn B_{2R(0)}
\end{aligned}
\right.
\end{align}

\begin{align}
\label{1.3}
c_{0j}[h(p,\xi, \Psi^*)](t) - \tilde c_{0j}[p,\xi,\Psi^*] (t) &= 0 \foral t\in (0,T), \quad j=1,2,
\\
\label{1.4}
c_{1j}[h(p,\xi, \Psi^*)](t)  &=  0 \foral t\in (0,T), \quad j=1,2.
\end{align}
In \eqref{eq-psi}  $\chi$ is a smooth cut-off function with compact support in $\Omega$ which is identically 1 on a fixed neighborhood of $q$ independent of $T$ and the function $ g(p,\xi, \Psi^*,\phi)$ is given by \equ{GG}.

\medskip
We see that if $(\phi_1,\phi_2,\phi_3,\phi_4,\psi,p,\xi)$ satisfies system \equ{eq-psi}--\equ{1.4}  then the functions
$$
\phi= \phi_1 + \phi_2+\phi_3+\phi_4, \quad \Psi^* = Z^*+ \psi
$$
solve the outer-inner gluing system  \equ{inner1}--\equ{outer1}.


\medskip
The way in which we will proceed to solve the full problem \equ{eq-psi}--\equ{1.4}  is the following. For given functions $\phi_1,\ldots,\phi_4$ and parameters $p$, $\xi$  in a suitable class, we solve first the outer problem  \eqref{eq-psi} in the form of an operator $\psi = \Psi[\phi_1+\phi_2+\phi_3+\phi_4, p,\xi] $ and denote $\Psi^*[\phi_1+\phi_2+\phi_3, p,\xi] = Z^* + \Psi[\phi_1+\phi_2+\phi_3+\phi_4, p,\xi]$.
Then we substitute $\Psi^*[\phi_1+\phi_2+\phi_3+\phi_4, p,\xi] $ in \equ{eqphi1}--\equ{eqphi4} and solve for $\phi_1$, $\phi_2$, $\phi_3$, $\phi_4$ as operators of the pair $(p,\xi)$. Finally, we solve for $p$ and $\xi$ the remaining equations. All this will be done by suitable control on the linear parts of the equation
and contraction mapping principle.

\subsection{Choice of constants}

We state here the constraints we impose in the parameters involved in the different norms. The values assumed will be sufficient
for the inner-outer gluing scheme to work.

\medskip
\begin{itemize}
\item
$\beta \in (0,\frac 12)$ is so that  $ R(t) = \la_*(t)^{-\beta}$.

\item
$\alpha \in (0,\frac{1}{2})$ appears in Proposition~\ref{propIntegralOp}.
It is the parameter used to define the remainder $\mathcal R_\alpha$ in \eqref{defRem}.

\item
We use the norm $\| \ \|_{*,\nu_1,a_1,\delta}$ \eqref{norm-phi1} to measure the solution $\phi_1$ in \eqref{eqphi1}. Here we will ask that $
\nu_1 \in (0,1)$, $ a_1 \in (2,3)$,  and $\delta > 0$ small and fixed.

\item
We use the norm $\| \ \|_{\nu_2,a_2-2}$ \eqref{norm-h} to measure the solution $\phi_2$ in \eqref{eqphi2}, with   $\nu_2 \in (0,1) $, $a_2 \in (2,3)$.

\item
We use the norm $\| \ \|_{**,\nu_3}$ \eqref{norm-starstar} for the solution $\phi_3$ of \eqref{eqphi3}, with $\nu_3>0$.

\item We use the norm $\| \ \|_{***,\nu_4}$ for the solution $\phi_4$ of \eqref{eqphi4}, with $\nu_4>0$.

\item
We are going to use the norm $\| \ \|_{\sharp,\Theta,\gamma}$ with a parameters $\Theta$, $\gamma$ satisfying some restrictions given below.

\item
We have parameters $m$, $l$ in Proposition~\ref{propIntegralOp}.
We work with $m$ given by
\begin{align*}
 m= \Theta -2\gamma(1-\beta).
\end{align*}
and $l$ satisfying $l<1+2m$.
\end{itemize}

We will assume that
\[
\alpha-1+2\beta>0
\]
which ensures that $m+(1+\alpha)\gamma > \Theta$.

To get the estimates for the outer problem \eqref{eq-psi}, we  need \eqref{assumpPar1} and
\begin{align*}
\Theta <   \min \Bigl( \beta , \frac{1}{2} - \beta ,
\nu_1-1+\beta(a_1-1) ,  \nu_2-1+\beta(a_2-1) , \nu_3-1,\nu_4-1+\beta  \Bigr)
\end{align*}
\begin{align*}
\Theta < \min
\Bigl(  \nu_1 - \delta \beta (5-a_1) -\beta , \nu_2 - \beta , \nu_3-3\beta , \nu_4 - \beta  \Bigr)
\end{align*}
and
\begin{align*}
\Theta>0.
\end{align*}
Also to control the nonlinear terms in \eqref{eq-psi} we need $\delta>0$ in  $\| \ \|_{*,\nu_1,a_1,\delta}$ to be small.

To find $\Theta$ in the range above we need
\begin{align}
\nonumber
\nu_1 & > \max\Bigl(1-\beta (a_1-1),\delta  \beta (5-a_1) - \beta \Bigr)
\\
\nonumber
\nu_2 & >  \max\Bigl(1-\beta (a_2-1), \beta \Bigr)
\\
\nonumber
\nu_3 & > \max ( 1 , 3\beta)
\\
\nonumber
\nu_4 & > \max(1-\beta,\beta) .
\end{align}

To solve the inner system given by equations
\eqref{eqphi1},
\eqref{eqphi2},
\eqref{eqphi3}, and
\eqref{eqphi4}
we will need
\begin{align}
\nonumber
\nu_1 & < 1 , \\
\nonumber
\nu_2 & < 1-\beta(a_2-2) ,  \\
\nonumber
\nu_3 &<  \min\Bigl(
1+\Theta+\sigma_1  ,
1+\Theta + 2\gamma \beta ,
\nu_1 + \frac{1}{2} \delta \beta ( a_1-2 ) \Bigr) , \\
\nonumber
\nu_4 & < 1,
\end{align}
where $\sigma_1  \in (0,\gamma(\alpha-1+2\beta))$.

\medskip

\subsection{The outer problem}
Our main result for problem \eqref{eq-psi} is the existence of a small solution for all small $T$, with certain precise absolute and Lipschitz estimates satisfied. To obtain this result we need a suitable norm that we define next.

\medskip
Given $\Theta>0$, $\gamma \in (0,\frac{1}{2})$ we define
\begin{align}
\nonumber
\| \psi\|_{\sharp, \Theta,\gamma}
&:=
\lambda_*(0)^{-\Theta}
\frac{1}{|\log T|  \lambda_*(0) R(0) }\|\psi\|_{L^\infty(\Omega\times (0,T))}
+ \lambda_*(0)^{-\Theta} \|\nabla_x \psi\|_{L^\infty(\Omega\times (0,T))}
\\
\nonumber
&\quad
+
\sup_{\Omega\times (0,T)}   \lambda_*(t)^{-\Theta-1} R(t)^{-1}
\frac{1}{|\log(T-t)|} |\psi(x,t)-\psi(x,T)|
\\
\nonumber
&\quad
+ \sup_{\Omega\times (0,T)} \, \lambda_*(t)^{-\Theta}
|\nabla_x \psi(x,t)-\nabla_x \psi(x,T) |
\\
\label{normPsi}
& \quad
+ \sup_{}
\lambda_*(t)^{-\Theta}
(\lambda_*(t) R(t))^{2\gamma}  \frac {|\nn_x \psi(x,t) -\nn_x \psi(x',t') |}{ ( |x-x'|^2 + |t-t'|)^{\gamma   }} ,
\end{align}
where the last supremum in taken in the region
\[
x,x'\in \Omega,\quad  t,t'\in (0,T), \quad |x-x'|\le 2 \la_*R(t), \quad  |t-t'| < \frac 14 (T-t) .
\]

We define the spaces
\begin{align*}
E_1 &= \{ \phi_1 \in L^\infty(\DD_{2R}) :
\nabla_y \phi_1  \in L^\infty(\DD_{2R}), \
\|\phi_1\|_{*,\nu_1,a_1,\delta} <\infty \}
\\
E_2 &= \{ \phi_2 \in L^\infty(\DD_{2R}) :
\nabla_y \phi_2  \in L^\infty(\DD_{2R}), \
\|\phi_2\|_{\nu_2,a_2}<\infty \}
\\
E_3 &= \{ \phi_3 \in L^\infty(\DD_{2R}) :
\nabla_y \phi_3  \in L^\infty(\DD_{2R}), \
\|\phi_3\|_{**,\nu_3} < \infty \}
\\
E_4 &= \{ \phi_4 \in L^\infty(\DD_{2R}) :
\nabla_y \phi_4  \in L^\infty(\DD_{2R}), \
\|\phi_4\|_{***,\nu_4} <\infty \}
\end{align*}
and use the notation
\[
E = E_1\times E_2 \times E_3 \times E_4,
\]
\begin{align*}
\Phi &= ( \phi_1,\phi_2,\phi_3,\phi_4) \in E
\\
\|\Phi\|_E &=
\|\phi_1\|_{*,\nu_1,a_1,\delta}
+\|\phi_2\|_{\nu_2,a_2-2}
+\|\phi_3\|_{**,\nu_3}
+\|\phi_4\|_{***,\nu_4}
\end{align*}
We define the closed ball
\[
\mathcal B = \{  \Phi \in E : \| \Phi\|_E \leq 1 \} .
\]

\begin{prop} \label{propi1}
Assume $Z_0^*$ satisfies \eqref{condZ0}.
Let $p(t)= \la(t) e^{i\omega(t)}$ and $\xi(t)$ satisfy estimates \eqref{cotin}, \eqref{cotin1}, $\Phi \in \mathcal B$.
Then there exists $C>0$  such that if  $T>0$ is  sufficiently small then there exists a solution $\psi = \Psi (p,\xi,\Phi,Z_0^*)$ to equation \eqref{eq-psi} such that
\be\begin{aligned}
\|  \Psi (p,\xi,\Phi,Z_0^*) \|_{\sharp,\Theta,\gamma}
& \leq C T^\sigma
(\| \Phi \|_E
+ \| \dot p \|_{L^\infty(-T,T)}
+  \| \dot \xi \|_{L^\infty(0,T)}
+  \| Z_0^* \|_*
)  .
\end{aligned}\label{A1}\ee
\end{prop}

\begin{proof}
The proof consists in writing problem \equ{eq-psi}  in a fixed point form involving an inverse for the inhomogeneous
 linear heat equation
 \begin{align}
 \label{heat-eq0}
 \left\{
 \begin{aligned}
 \psi_t  & = \Delta_x \psi + f(x,t) \inn \Omega \times (0,T) \\
 \psi  & = 0 \onn \pp \Omega \times (0,T) \\
 \psi(q,T)  & = 0  \\
 \ \psi(x,0)  & =  (c_1 \,  \mathbf{e_1}  + c_2 \, \mathbf{e_2}  + c_3\,  \mathbf{e_3}) \eta_1    \inn  \Omega
 \end{aligned}
 \right.
 \end{align}
 for suitable constants $c_1,c_2,c_3$, where $\mathbf{e_1}$, $\mathbf{e_1}$, $\mathbf{e_1}$ are defined in \eqref{e123},
and   $q\in \Omega$ and $T>0$ is sufficiently  small.
The fixed smooth cut-off  $\eta_1$ has compact support in $\Omega$ and is such that $\eta_1\equiv 1$ in a neighborhood of $q$.
The right hand side is assumed to satisfy $\| f\|_{**}<\infty$ where
 \begin{align}
\nonumber
 \|f\|_{**} : =   \sup_{ \Omega \times (0,T)}  \Big ( 1 + \sum_{i=1}^3 \varrho_i(x,t)\, \Big )^{-1}  {|f(x,t)|} .
 \end{align}
 and the weights are defined by
\begin{align*}
 \left\{
 \begin{aligned}
 \varrho_1 & :=   \lambda_*^{\Theta}  (\lambda_* R)^{-1}  \chi_{ \{ r \leq 3R\la_* \} }
 \\
 \varrho_2 & := T^{-\sigma_0}  \frac{\lambda_*^{1-\sigma_0}}{r^2}  \chi_{ \{ r \geq  R\la_* \} }
 \\
 \varrho_3 & := T^{-\sigma_0} ,
 \end{aligned}
 \right.
 \end{align*}
 where $r= |x-q|$, $\Theta>0$ and  $\sigma_0>0$ is  small.
 (The factor $T^{\sigma_0}$ in front of $\varrho_2$ and $\varrho_3$ is a simple way to have parts of the error small in the outer problem.)
 These weights  naturally adapt to the form of the outer error $g$  in \eqref{GG}.
In Proposition \ref{prop3} a solution of Problem \eqref{heat-eq0} is built as a linear operator of $f$ with the estimate
\begin{align*}
\| \psi\|_{\sharp, \Theta ,\gamma}
+ \frac{\lambda_*(0)^{-\Theta}  ( \lambda_*(0) R(0) )^{-1} }{ |\log T| }( |c_1| +|c_2| +|c_3| )  \leq C \|f\|_{**} ,
\end{align*}
This fact and direct estimates for the outer error
 make the the contraction mapping principle applicable in a suitable region, producing an operator as in \equ{A1}. To illustrate some of these estimates, let us write
$
g = g_1 + g_2 + g_3 + g_4
$
where
\begin{align*}
g_1 & =
Q_\omega
\bigl( (\Delta_x \eta) \phi + 2  \nn_x \eta \nn_x \phi - \eta_t  \phi
\bigr)
\\
& \quad
+ \eta Q_\omega\bigl( - \dot\omega J \phi  +  \la^{-1}\dot\la  y\cdot \nn_y \phi + \la^{-1} \dot\xi \cdot\nn_y \phi \bigr)
\\
g_2
& =  (1-\eta) \ttt L_U [\Psi^*] + (\Psi^*\cdot U ) U_t
\\
g_3 & =
(1-\eta)[ \KK_{0}[p,\xi]+ \KK_{1}[p,\xi]] + \Pi_{U^\perp}[ \ttt \RR_1] + ( \Phi^0\cdot U)U_t ,
\\
g_4 & = N_U( \eta Q_\omega \phi  + \Pi_{U^\perp}( \Phi^0  + \Psi)^* ) .
\end{align*}
We claim that
\begin{align}
\nonumber
\|g_1\|_{**} \leq C  T^\sigma  \|\Phi\|_E ,
\end{align}
for some $\sigma>0$.
Indeed, we have
\begin{align*}
| \Delta_x \eta \phi_1  |
&\leq C \lambda_*^{\nu_1-2} R^{-a_1} \chi_{[ |x-q| \leq 3 \lambda_* R]}
\|\phi_1\|_{*,\nu_1,a_1,\delta}
\\
| \Delta_x \eta \phi_2  |
&\leq C \lambda_*^{\nu_2-2} R^{-a_2} \chi_{[ |x-q| \leq 3 \lambda_* R]}
\|\phi_2\|_{\nu_2,a_2-2}
\\
| \Delta_x \eta \phi_3  |
&\leq C \lambda_*^{\nu_3-2} R^{-1} \chi_{[ |x-q| \leq 3 \lambda_* R]}
\|\phi_3\|_{**,\nu_3}
\\
| \Delta_x \eta \phi_4  |
&\leq C \lambda_*^{\nu_4-2} R^{-2} \log R \chi_{[ |x-q| \leq 3 \lambda_* R]}
\|\phi_4\|_{***,\nu_4}  .
\end{align*}
The norm $\| \ \|_{**}$ is actually motivated by the weights appearing above.
If
\begin{align*}
\Theta < \min( \nu_1-1+\beta(a_1-1) ,  \nu_2-1+\beta(a_2-1) , \nu_3-1,\nu_4-1+\beta ) ,
\end{align*}
we find that  for any $j=1,2,3,4$:
\begin{align*}
| \Delta_x \eta \phi_j |\leq C T^\sigma \lambda_*^{\Theta-1+\beta}
\chi_{[ |x-q| \leq 3 \lambda_* R]}
\| \Phi \|_E,
\end{align*}
for some $\sigma>0$.
Then we have
\[
\| Q_\omega  (\Delta_x \eta) \phi  \|_{**} \leq C
T^\sigma
\|\Phi\|_E
\]
and similarly
\[
\| ( \partial_t \eta)  Q_\omega \phi \|_{**}
+\|   Q_\omega  \la^{-1} \nabla_x\eta\nabla_y \phi   \|_{**}
\leq  C T^\sigma
\|\Phi\|_E .
\]

The other terms $g_2$, $g_3$, $g_4$ can be estimated in the same way.
In the estimate for $g_2$ it is important to have the property that $\Psi^* = Z^* + \psi$ vanishes at $(q,T)$.
Lipschitz properties are proved using similar calculations.
\end{proof}


\medskip
The operator $\Psi (p,\xi,\Phi,Z_0^*)$ satisfies Lipschitz properties with respect to its arguments, which are consequence of its construction.
See Corollaries~\ref{coroLipOuter1} and \ref{coroLipOuter2} in the appendix.

\medskip
What we do next is to take $\Phi \in E$ with $\|\Phi\|_E  \leq 1 $ and substitute  $\Psi^*(p,\xi,\Phi,Z_0^*) = Z^*+ \Psi(p,\xi,\Phi,Z_0^*) $  into
\eqref{eqphi1}--\eqref{eqphi4}.
We can then write equations \eqref{eq-psi}--\eqref{eqphi4}
as the fixed point problem
\begin{align}
\label{ptofijo}
\Phi = \mathcal F (\Phi)
\end{align}
where
\begin{align*}
\mathcal F (\Phi) = ( \mathcal F_1(\Phi) ,  \mathcal F_2(\Phi) ,
\mathcal F_3(\Phi)  ,
\mathcal F_4(\Phi) ) , \quad \mathcal F : \bar{\mathcal B}_1\subset E \to E
\end{align*}
wiith
\begin{align*}
\mathcal F_1(\Phi) &=   \mathcal T_{\lambda,1}  (
h_1[p,\xi,  \Psi^*(p,\xi,\Phi,Z_0^*) ] )
\\
\mathcal F_2(\Phi) &=   \mathcal T_{\lambda,2}  (
h_2[p,\xi,  \Psi^*(p,\xi,\Phi,Z_0^*) ] )
\\
\mathcal F_3(\Phi) &=   \TT_{\lambda,3 }
\Bigl(
h_3[p,\xi,  \Psi^*(p,\xi,\Phi,Z_0^*) ]+
\sum_{j=1}^2 c_{0j}^*[p,\xi, \Psi^*(p,\xi,\Phi,Z_0^*) ] w_\rho^2 Z_{0j}
\Bigr)
\\
\mathcal F_4(\Phi) &=   \TT_{\lambda,4 } \Bigl( \sum_{j=1}^2 c_{-1,j}[h_1[p,\xi, \Psi^*(p,\xi,\Phi,Z_0^*) ]] w_\rho^2 Z_{-1,j}  \Bigr)	  .
\end{align*}
Although $\mathcal F$ also depends on $p$, $\xi$, $Z_0^*$ we will omit this dependence from the notation for the moment.

Our next step is to solve problem \eqref{ptofijo}.

\subsection{The inner problem}
\begin{prop}
\label{innnerSystem}
Assume that $p$ and $\xi$ satisfy estimates \eqref{cotin} and that $Z_0^*$ satisfies \eqref{condZ0}.
Then the system of equations \eqref{ptofijo} for $\Phi= ( \phi_1, \phi_2 , \phi_3,\phi_4)$ has a solution $\Phi(p,\xi,Z_0^*)$ in $\bar{\mathcal B}_1 \subset E$.
\end{prop}
\begin{proof}
We estimate in detail the operator $\mathcal F_1$. The others are handled similarly.
We recall that we have decomposed $Z_0^* = Z_0^{*0} + Z_0^{*1}$ (c.f. \eqref{condZ0}).
We claim that for $\| \Phi\|_E \leq 1$ we have
\begin{align}
\label{estF1-1}
\|  \mathcal F_1(\Phi)  \|_{*,a,_1,\nu_1}
&  \leq C \lambda_*(0)^{\Theta}   T^\sigma
( \| \Phi \|_E
+ \| \dot p \|_{L^\infty(-T,T)}
+  \| \dot \xi \|_{L^\infty(0,T)}  )
+ C T^\sigma \|  Z_0^{*0} \|_{*} ,
\end{align}
and for $\|\Phi_1\|_E, \|\Phi_2\|_E \leq 1$
\begin{align}
\label{estF1-lip}
\|  \mathcal F_1(\Phi_1) - \mathcal F_1(\Phi_2)  \|_{*,a_1,\nu_1}
\leq C T^\sigma
\lambda_*(0)^{\Theta }
\| \Phi_1  - \Phi_2 \|_E  .
\end{align}
To prove \eqref{estF1-1},
we recall that by Proposition~\ref{prop1.0} we have
\begin{align*}
\|  \mathcal F_1(\Phi)  \|_{*,\nu_1,a_1,\delta}
& \leq C \| h_1[p,\xi,  \Psi^*(p,\xi,\Phi,Z_0^*) ] \|_{\nu_1,a_1}.
\end{align*}
From the definition of $h_1$ and recalling that
$
\Psi^*(p,\xi,\Phi,Z_0^*) = Z^* +  \Psi(p,\xi,\Phi,Z_0^*)
$
we get
\begin{align*}
&
\| h_1[p,\xi,  \Psi^*(p,\xi,\Phi,Z_0^*) ] \|_{\nu_1,a_1}
\\
& \quad \leq
\| \lambda^2  Q_{-\omega}   (\tilde L_U  [ \Psi(p,\xi,\Phi,Z_0^*) ]_0
+\tilde L_U  [ \Psi(p,\xi,\Phi,Z_0^*) ]_2  ) \chi_{\DD_{2R} }  \|_{\nu_1,a_1}
\\
& \quad \quad
+\| \lambda^2  Q_{-\omega}
 (\tilde L_U  [ Z^* ]_0
+\tilde L_U  [ Z^* ]_2  )
\chi_{\DD_{2R} }  \|_{\nu_1,a_1}
+  \|  \lambda^2  Q_{-\omega}  \KK_{0}[p,\xi] \|_{\nu_1,a_1} .
\end{align*}
We claim that for $j=0$ and $j=2$:
\begin{align}
\nonumber
\| \lambda^2  Q_{-\omega} \tilde  L_U  [ \Psi(p,\xi,\Phi,Z_0^*)  ]_j \, \chi_{\DD_{2R} }
 \|_{\nu_1,a_1}
&\leq
C T^\sigma
\lambda_*(0)^{ \Theta  }
( \| \Phi \|_E + \| \dot p \|_{L^\infty(-T,T)}
\\
\label{phi1RHS1-0}
& \quad
+  \| \dot \xi \|_{L^\infty(0,T)}
+  \| Z_0^* \|_*  )  .
\end{align}
Indeed,  let $\psi = \Psi(p,\xi,\Phi,Z_0^*)$.
From \eqref{Ltilde-j} we get, for $j=0$ and $j=2$:
\begin{align*}
| \lambda^2  Q_{-\omega} \tilde  L_U  [ \psi ]_j|
& \leq C \frac{\lambda_*}{(1+|y|)^3} \|\nabla_x \psi  \|_{L^\infty} .
\end{align*}
We use  $ \nu_1 <1 $  and $a_1<3$ to estimate for $|y|\leq 2 R$
\begin{align}
\nonumber
\frac{\lambda_*}{(1+|y|)^3}
& \leq
\frac{\lambda_*^{\nu_1}}{(1+|y|)^{a_1}}  \lambda_*(0)^{1-\nu_1}  .
\end{align}
Then for $|y|\leq 2 R$ and $j=0,2$:
\begin{align*}
| \lambda^2  Q_{-\omega} \tilde  L_U  [ \psi ]_j |
& \leq C \frac{\lambda_*^{\nu_1}}{(1+|y|)^{a_1}}
\lambda_*(0)^{1-\nu_1}
\|\nabla_x \psi \|_{L^\infty} .
\end{align*}
By the definition of the norm $\| \ \|_{\sharp,\Theta,\gamma}$ (c.f. \eqref{normPsi}) and Proposition~\ref{propi1} we have
\begin{align*}
\| \nabla_x  \psi \|_{L^\infty}
& \leq C \lambda_*(0)^{\Theta}
\|  \Psi(p,\xi,\Phi,Z_0^*) \|_{\sharp,\Theta,\gamma}
\\
& \leq C \lambda_*(0)^{\Theta}   T^\sigma
( \| \Phi \|_E
+ \| \dot p \|_{L^\infty(-T,T)}
+  \| \dot \xi \|_{L^\infty(0,T)}
+  \| Z_0^* \|_*
)  .
\end{align*}
Hence for $j=0,2$
\begin{align*}
| \lambda^2  Q_{-\omega} \tilde  L_U  [ \psi ]_j |
&\leq C \frac{\lambda_*^{\nu_1}}{(1+|y|)^{a_1}}
T^\sigma
\lambda_*(0)^{ \Theta }
( \| \Phi \|_E + \| \dot p \|_{L^\infty(-T,T)}
+  \| \dot \xi \|_{L^\infty(0,T)}
\\
& \quad
+  \| Z_0^* \|_*  ),
\end{align*}
and therefore we see that \eqref{phi1RHS1-0} is valid.
Next we claim that
\begin{align}
\label{phi1RHS2}
\| \lambda^2  Q_{-\omega} \tilde  L_U  [Z^* ]_j \chi_{\DD_{2R} }  \|_{\nu_1,a_1}
\leq
C T^\sigma \| Z_0 \|_*,
\end{align}
for $j=0,2$ and some $\sigma>0$.
Indeed, we use estimate \eqref{estGrad} of Lemma~\ref{lemmaGradEst} to obtain
for $j=0,2$:
\begin{align*}
|  \la^2  Q_{-\omega}\ttt L_U  [ Z^* ]_j \,  \chi_{\DD_{2R} } |
\leq C \frac{\lambda_*}{(1+\rho)^3} |\log \varepsilon|
\| Z_0 \|_* ,
\end{align*}
where $\varepsilon>0$ is given by  \eqref{defEpsilon}.
Since $\nu_1<1$, we get
\begin{align*}
\|  \la^2  Q_{-\omega}\ttt L_U  [  Z^*    ]_j \, \chi_{\DD_{2R} }  \|_{\nu_1,a_1}
& \leq C \lambda_*(0)^{1-\nu_1}  |\log \lambda_*(0)|  \| Z_0 \|_* .
\end{align*}
This implies \eqref{phi1RHS2}.
Next we estimate $ \lambda^2  Q_{-\omega} \KK_{0}[p,\xi] $.
We claim that
\begin{align}
\label{phi1RHS3}
\|  \lambda^2  Q_{-\omega}  \KK_{0}[p,\xi] \|_{\nu_1,a_1}
& \leq C T^\sigma \| \dot p \|_{L^\infty(-T,T)}.
\end{align}
Indeed, consider $\KK_{01}$ given in \eqref{K01}.
We have
\begin{align*}
|\lambda^2 Q_{-\omega} \KK_{01}[p,\xi] |
&
\leq C \frac{\lambda_*}{(1+\rho)^3} \int_{-T} ^t  |  \dot p(s)    k(z,t-s) | \, ds .
\end{align*}
A direct computation shows that
\begin{align*}
\|  \la^2  Q_{-\omega} \tilde  L_U  [  \KK_{01}[p,\xi]   ] \chi_{\DD_{2R} }  \|_{\nu_1,a_1}
& \leq C \lambda_*(0)^{1-\nu_1} \| \dot p \|_{L^\infty(-T,T)}
\\
& \leq CT^\sigma \| \dot p \|_{L^\infty(-T,T)} ,
\end{align*}
for some $\sigma>0$.
The estimate for $\KK_{02}$ is similar, and we obtain \eqref{phi1RHS3}.
Combining \eqref{phi1RHS1-0},  \eqref{phi1RHS2},  and \eqref{phi1RHS3} we finally obtain
\begin{align}
\nonumber
& \| h_1[p,\xi,  \Psi^*(p,\xi,\Phi,Z_0^*) ] \|_{\nu_1,a_1}
\leq C T^\sigma( \| \Phi \|_E + \| \dot p \|_{L^\infty(-T,T)}
+  \| Z_0^* \|_*  ) .
\end{align}
Then thanks to Proposition~\ref{prop1.0} we get
\eqref{estF1-1}.
The proof of estimate \eqref{estF1-lip} is  similar.
\end{proof}

\medskip

Let $\Phi (p,\xi,Z_0^*)$ be the solution of \eqref{ptofijo} constructed in Proposition~\ref{innnerSystem}.
As a consequence of the construction above and the Lipschitz estimates for the inner problem in \S \ref{lips}  $\Phi$  is Lipschitz in the parameters $p,\xi,Z_0^*$ in the following sense.

\begin{corollary}
Assume that $p_1,p_2$ and $\xi_1,\xi_2$ satisfy estimates \eqref{cotin} and that $Z_{0,1}^*$, $Z_{0,2}^*$ have the form
\[
Z_{0,l}^* = Z_0^{*0} +  Z_{0,l}^{*1} , \quad l=1,2 ,
\]
with $ Z_0^{*0} $ satisfying \eqref{conditionsZ0a} and
$
\| Z_{0,l}^{*1}  \|_* \leq T^\sigma.
$
Let us write $p_j = \lambda_j e^{i\omega_j}$ for $j=1,2$.
for some $\sigma>0$.
Then
\begin{align*}
\| \Phi (p_1,\xi_1, Z_{0,1}^* ) - \Phi (p_2,\xi_2, Z_{0,2}^* ) \|_E
&
\leq \lambda_*(0)^\sigma
\Bigl[
\|\lambda_* (\dot \omega_1 - \dot \omega_2)\|_\infty
+\Bigl\| \frac{\lambda_1-\lambda_2}{\lambda_*} \Bigr\|_{L^\infty}
\\
& \quad
+ \| \dot\lambda_1 - \dot\lambda_2\|_{L^\infty}
+ \Bigl\| \frac{\xi_1-\xi_2}{\lambda_* R} \Bigr\|_{L^\infty}
+ \Bigl\| \frac{ \dot\xi_1 - \dot\xi_2}{R} \Bigr\|_{L^\infty}
\\
& \quad
+ \| Z_{0,1}^{*1}  - Z_{0,2}^{*1} \|_*
\Bigr] ,
\end{align*}
for some possibly smaller $\sigma>0$.
\end{corollary}

With this we can now state the following result.
Let $\Phi(p,\xi,Z_0^*)$ denote the solution of \eqref{ptofijo} constructed in Proposition~\ref{innnerSystem}.

\begin{prop}
\label{propSolPar}
Given $Z_0^*$ of the form \eqref{condZ0} there exists $p= \lambda e^{i\omega}$ and $\xi$ such that \eqref{1.3} and \eqref{1.4} are satisfied.
\end{prop}

The proposition above yields the existence of a blow-up solution.
The proof is given in Section~\ref{sectSolPar}.

\section{Linear theory for the inner problem}
\label{sectLinearTheory}

At the very heart of capturing  the bubbling structure is the construction of an inverse for
the linearized heat operator around the basic harmonic map.
We consider the linear equation
\begin{align} \la^2 \pp_t  \phi   &=     L_{ W  }  [\phi]       + h(y,t)\inn  \DD_{2R} \label{eqL00}\\
 \phi(\cdot,0 ) &=   0 \inn B_{2R( 0)}
\nonumber
\\
\nonumber
\phi \cdot  W  &=  0 \inn   \DD_{2R}
\end{align}
where
\[
\DD_{2R} =  \{ (y,t) \ /\  t \in ({0}, T) ,\  y \in  B_{2R(t)}(0) \} .
\]
We assume that $h(y,t)$ is defined for all $(y,t)\in \R^2 \times (0,T)$ and   satisfies
\[
 h \cdot  W  =   0  , \quad
|h(y,t)| \leq C \frac{\lambda_*^\nu}{(1+|y|)^a} ,
\]
where $\nu>0$ and $a\in (2,3)$ (so that $\|h\|_{a,\nu}<\infty$ with the norm defined in \eqref{norm-h}).

The parameter $R$ is given by \eqref{RRR}, that is $R(t) = \lambda_*(t)^{-\beta}$, $\beta\in (\frac{1}{4},\frac{1}{2})$.
Also, we assume that the parameter function $\la(t)$ satisfies
we have that
$$
a\la_* (t)  \le \la (t)\le b \la_* (t)  
\foral t\in (0,T)
$$
for some positive numbers $a,b,c$  independent of $T$.

We observe that a priori  we are not imposing boundary conditions in problem \equ{eqL00}. Our purpose is to construct a solution $\phi$ that defines a linear operator of $h$ and satisfies uniform bounds in terms of suitable norms. In some sense this is an extension of ''Fredholm" theory for linear parabolic problem \equ{eqL00}.

\medskip
All functions $h(y,t)$ with $h(y, t) \cdot  W (y) \equiv 0$ can be expressed in polar form as
\begin{equation}
\label{hh1}
h(y,t ) \ =\   h^1(\rho,\theta,t )E_1(y) +
h^2(\rho,\theta,t )E_2(y),    \quad y= \rho e^{i\theta}.
\end{equation}
We can also expand in Fourier series
\begin{equation}
\label{fourierH}
\ttt h(\rho,\theta,t)   := h^1 + i h^2  \ =\  \sum_{k=-\infty}^{\infty} \ttt h_k (\rho,t)e^{ik\theta}, \quad  \ttt h_k =  \ttt h_{k1} + i\ttt h_{k2}
\end{equation}
so that
\begin{equation}
\label{hh3}
h(y,t )  = \sum_{k=-\infty}^{\infty} h_k(y,t) =:h_0(y,t)+  h_1(y,t) + h_{-1}(y,t) +h^\perp (y,t) ,
\end{equation}
where
\begin{equation}
\label{hh4}
  h_k(y,t) = {\rm Re}\, (\ttt h_k(\rho,t)e^{ik\theta})\,E_1 +  {\rm Im}\, (\ttt h_k(\rho,t)e^{ik\theta})\,E_2.
\end{equation}

We consider the functions $Z_{kj}(y)$ defined in \eqref{ZZ} and \eqref{Zmode-1} and define for $k=-1,0,1$,
$$
\bar h_{k}(y,t)   :=       \sum_{j=1}^2   \frac {\chi Z_{kj}(y)} { \int_{\R^2} \chi  |Z_{kj} |^2  }  \, \int_{ \R^2 }h(x  , t )  \cdot Z_{kj}(z)\, dz, \quad
$$
where
$$
\chi(y,t) = \begin{cases}
w_\rho^2(|y|)  & \hbox{ if }  |y|< 2R(t),\\
0& \hbox{ if }  |y|\ge 2R(t).
\end{cases}
$$
The main result in this section  is the following, where we use the norm  $\|h\|_{a,\nu}$ defined in \eqref{norm-h}.

\begin{prop}
\label{prop2}
Let $2<a<3$, $\nu>0$ and let $h$ with $\|h\|_{a,\nu} <+\infty$.
Let us write $h = h_0 + h_1 + h_{-1} + h^\perp$ with $h^\perp = \sum_{k\not=0,\pm 1} h_k$.
Then there exists a solution $\phi[h]$ of problem $\equ{eqL00}$, which defines
a linear operator of $h$, and satisfies the following estimate in $\DD_{2R}$:
\begin{align*}
&
(1+|y|)\left |\nn_y \phi(y,t)\right |  +  |\phi(y,t)|
\\
& \lesssim
  \frac { \lambda_*(t)^\nu R(t)^{\frac{5-a}2} } { 1+|y| } \,
 \min\{1,  R^{\frac{5-a}2} |y|^{-2} \}
 \, \| h_0 -\bar h_{0} \|_{a,\nu}
+
\frac{ \lambda_*(t)^\nu  R(t)^{2}} {1+ |y|}  \|\bar h_0\|_{a,\nu}
\\
& \quad
+   \frac{ \lambda_*(t)^\nu }{ 1+ |y|^{a-2} }\, \left \| h_1 - \bar h_1\right  \|_{a,\nu}
+   \frac{ \lambda_*(t)^\nu  R(t)^{4}} {1+ |y|^2} \left \| \bar h_{1} \right  \|_{a,\nu}
\\
& \quad
+
  \frac { \lambda_*(t)^\nu R(t)^{\frac{5-a}2} } { 1+|y| } \,
 \min\{1,  R^{\frac{5-a}2} |y|^{-2} \}
 \, \| h_{-1} -\bar h_{-1} \|_{a,\nu}
 + \lambda_*(t)^\nu \log R(t) \,   \| \bar h_{-1} \|_{a,\nu}
\\
& \quad
+
\frac{ \lambda_*(t)^\nu  }{ 1+ |y|^{a-2} }\,   \|h^\perp \|_{a,\nu} .
\end{align*}
\end{prop}

\medskip
The construction of the operator $\phi[h]$ as stated in the proposition will be carried out mode by mode in the Fourier series expansion.
We shall use the convention that $h(y,t)=0$ for $|y|> 2R(t)$.
Let us write
$$
\phi = \sum_{k=-\infty}^\infty \phi_k , \qquad
 \phi_k(y,t) =  {\rm Re}\, (\vp_k(\rho,t)e^{ik\theta})\,E_1 +  {\rm Im}\, (\vp_k(\rho,t)e^{ik\theta})\,E_2.
$$
We shall build a solution of \equ{eqL00} by solving separately each of the equations
\begin{align}\label{eqL00k}
  \la^2\pp_t \phi_k\  = & \   L_ W  [\phi_k]  + h_k(y,t) = 0  \inn \DD_{4R},   \\
  \phi_k(y ,0)\ = & \  0 \inn B_{4R(0)}(0) , \quad  \nonumber\end{align}
which,  are equivalent to the problems
%
%
%
%
%
%
%
%
\begin{align*}
\la^2 \pp_t \vp_k  \ =& \ \mathcal L_k[\vp_k]  + \ttt h_k (\rho, t)  \inn  \ttt D_{4R} ,
\\
 \vp_k(\rho,0)  \ = &\ 0 \inn (0, 4R (0)) \nonumber
\end{align*}
with
$$ \ttt D_{4R}= \{ (\rho,t) \ /\  t\in (0,T),\ \rho \in (0,4R(t))  \} $$
and we recall
$$
\mathcal L_k[\vp_k]   := \pp^2_\rho \vp_k   + \frac{ \pp_\rho \vp_k}{\rho}  -  (k^2 + 2k\cos w  + \cos(2w) ) \frac{\vp_k} {\rho^2}.  $$

 We have the validity of the following result.
\begin{lemma} \label{step1} Let  $\nu>0$ and  $0<a < 3 $, $a\ne 1,2$. Assume that
$$\| h_k(y,t)\|_{a,\nu} < +\infty.$$
Then problem  \equ{eqL00k} has a unique bounded solution $\phi_k(y,t)$ of the form
$$
\phi_k(y,t)\ =\  {\rm Re}\, (\vp_k(\rho,t)e^{ik\theta})\,E_1 +  {\rm Im}\, (\vp_k(\rho,t)e^{ik\theta})\,E_2
$$
which in addition satisfies the boundary condition
\be\label{bc}  \phi_k(y,t) = 0 \quad \hbox{for all }  t\in (0,T), \quad y\in \pp B_{R(t)}(0). \ee
These solutions satisfy the estimates
\begin{align*}
|\phi_k(y,t) | \ \le  \  C \|h\|_{a,\nu}\, \la_* ^{\nu} k^{-2} \,\left \{ \begin{matrix}
 R^{2-a}  & \hbox{ if }&  \  a<2,  \\
\ \ (1+\rho)^{2-a}   & \hbox{ if }&   \ a>2 , \\
\end{matrix}\right .\quad\hbox{if } k\ge 2.
\end{align*}
\begin{align*}
|\phi_{-1}(y,t) | \ \le  \  C \|h\|_{a,\nu}\, \la_* ^{\nu} \,\left \{ \begin{matrix}
 R^{2-a}  & \hbox{ if }&  \  a<2,  \\
\ \ \log R   & \hbox{ if }&   \ a>2 , \\
\end{matrix}\right .
\end{align*}
\begin{align*}
|\phi_0(y,t) | \ \le  \  C \|h\|_{a,\nu}   \la_*^{\nu } (1+ \rho)^{-1}  \,\left \{ \begin{matrix}
 R^{2}  & \hbox{ if }&  \  a>1,  \\
\ \ R^{3-a}  & \hbox{ if }&   \ a<1 , \\
\end{matrix}\right .
\end{align*}
\begin{align*}
\ \ |\phi_1(y,t) | \ \le  \  C \|h\|_{a,\nu}  \la_*^{\nu } (1+ \rho)^{-2}  R^{4}
\end{align*}
with $C$ independent of $R$ and $k$.
\end{lemma}

\proof
Standard parabolic theory yields existence of a unique solution to equation \equ{eqL00k} that satisfies the boundary condition \equ{bc},
for each $k$. Equivalently, the problem
\begin{align}  \label{Llxx}
\la^2 \pp_t \vp_{k}  & = \mathcal L_{k}[\vp_{k}]  + \ttt h_{k}  (\rho, t)  \inn  \ttt D_{4R} ,
\\  \vp_{k} (t , 4R )  & =  0 \foral t\in (0, T) \nonumber\\
 \vp_k(0,\rho)  & = 0 \inn (0, 4R(0)), \nonumber
\end{align}
$$
\mathcal L_k[\vp_k]   = \pp^2_\rho \vp_k   + \frac{ \pp_\rho \vp_k}{\rho}  -  (k^2 + 2k\cos w  + \cos(2w) ) \frac{\vp_k} {\rho^2}  $$
has a  unique solution $\vp_{k}(\rho,t) $
which is bounded in $\rho$ for each $t$.

\medskip
 We use barriers to derive the
desired estimates.  A first observation we make is that for mode $k=-1$  the elliptic equation
$\mathcal L_{-1}[\varphi]+g(\rho)=0$ in $(0,4R)$ with $\varphi(4R)=0$
has a unique bounded solution given by the variation of parameters formula
\begin{align}
\label{vp-1}
 \vp(\rho) & :=   Z_{-1} (\rho ) \int_\rho^{4R}  \frac {d r}{ \rho Z_{-1} (r)^2} \int_0^r g(s) Z_{-1}(s) s \, ds, \quad
 \\ Z_{-1} (\rho) & =  -\rho^2w_\rho= \frac{2\rho^2}{\rho ^2 + 1}.\nonumber
\end{align}
Here we have used that $\mathcal L_{-1}[Z_{-1}] =0$.
Let us call $\vp_0(\rho)$ the function in \equ{vp-1} with
$ g(\rho) :=   2 (1+\rho)^{-a} $.
We readily estimate
$$
|\vp_0(\rho)|  \ \le \  \begin{cases}   R^{2-a} & \hbox{ if } a<2, \\   (1+\rho)^{2-a} & \hbox{ if } a>2.   \end{cases}     .
$$
Let us call  $\bar \vp (\rho,t) =  \la_*(t)^{\nu}\vp_0(\rho) . $
Then we see
that
\begin{align*}
 - \la^2 \bar \vp_t (\rho,t) +   \mathcal L_{-1} [\bar \vp(\rho,t)  ]   +  \frac{ \la_*^{\nu}}{  (1+\rho)^a}
& \le
c\,\la_*^{\nu+1} |\dot \la_*| \vp_0(\rho) - \frac{ \la_*^{\nu}}{  (1+\rho)^a}
\\
&
\le   -\la_*^{\nu}(1+ \rho)^{-a}
\left [  1   - C \la_*  R^{2-a}(1+\rho)^{a}  \right ]
\\
&
< 0
\end{align*}
in $ \ttt D_{4R}$.
Indeed, since $R(t) \ll  \la_*^{-\frac 12}  $,  the inequality  holds provided that $T$ was chosen sufficiently small.
Thus for $k=-1$ the barrier $\| h\|_{a,\nu} \, \bar \vp(\rho,t)$ dominates both, real and imaginary parts of $\vp_{-1}(\rho, t)$. As a conclusion, we find
$$
|\phi_{-1} (y , t)| \ \le\ C  \| h\|_{a,\nu}  \la_*^{\nu}\,  \begin{cases}   R^{2-a} & \hbox{ if } a<2, \\   (1+\rho)^{2-a} & \hbox{ if } a>2,\end{cases}  \inn \DD_{4R}.
$$
The cases $k=0,1,-2$ can be dealt with in exactly the same manner, by replacing $Z_{-1}$   in Formula \equ{vp-1} respectively by the functions
\begin{align}
\label{Z000}
Z_0(\rho) =  \frac \rho{\rho^2+1} ,\quad   Z_1(\rho) =  \frac 1{\rho^2+1}, \quad Z_{-2} (\rho) =    \frac{\rho ^{3}}{\rho^2+1} .
\end{align}
The estimates for $\phi_k$ predicted in the lemma then readily follow for $k=-2,-1,0,1$.
Finally,
let us now consider $k$ with $|k|\ge 2$ and $k\ne -2$ and the function $\bar \vp(\rho, t)  $ as above. Now we find
\begin{align*}
-\la^2 \bar \vp_t (\rho,t) +   \mathcal L_{k} [\bar \vp(\rho,t)  ]
& \leq
(\mathcal L_{k}- \mathcal L_{-1}) [\bar \vp(\rho,t)  ]
\\
& \le
-  C\la_*^\nu (k^2 -1 + 2(k-1))\frac 1{\rho^2} (1+\rho)^{2-a}
\\
&
<   -  C (k^2 -1 + 2(k-1))\frac { \la_*^{\nu}} {(1+\rho)^{a}}  \inn \ttt \DD_{4R} .
\end{align*}
The latter quantity is negative provided that $|k|\ge 2$ and $k\ne -2$
 and hence we get the estimate
\[
|\phi_{k} (y , t)| \ \le\   \frac{C}{k^2} \| h\|_{a,\nu}  \la_*^{-\nu}\, \begin{cases}   R^{2-a} & \hbox{ if } a<2, \\   (1+\rho)^{2-a} & \hbox{ if } a>2,   \end{cases}   \inn \DD_{4R}.
\]
 The proof is concluded.
\qed

\bigskip
We can get gradient estimates for the solutions built in the above lemma by means of the following result.

\begin{lemma}\label{gradient}
Let $\phi$ be a solution of the equation
\begin{align}
\lambda^2 \pp_t  \phi     &=      L_{ W  }  [\phi]       + h(y, t )\inn  \DD_{4\gamma R}
\label{cv}
\\
\phi(\cdot, 0 ) &=   0 \inn B_{4\gamma R(0)} .
\nonumber\end{align}
Given numbers $a,b,\gamma $, there exists a $C$ such that if for some $M>0$ we have
\be\label{M}
  |\phi (y,t)|    +  (1+ |y|)^{2}  |h(y,t)|   \le  M\,  \lambda_*(t)^b  (1+|y|)^{-a} \inn \DD_{4\gamma R},
\ee
then
\be\label{MMM}
    (1+ |y|)   |\nn_y \phi (y,t)|  \ \le \  C \, M  \lambda_*(t)^b  (1+|y|)^{-a} \inn \DD_{ 3\gamma  R}
\ee
and we recall
$$
\DD_{\gamma R} =   \{(y,t)  \ /\  |y| <  \gamma R(t), \quad t \in (0,T) \} .
$$
If in addition we know that $\phi$ satisfies the boundary condition
$\phi(\cdot ,t) = 0$ on $ \pp B_{4\gamma R(t)} $ for all $ t \in (0,T)$ then estimate \equ{MMM} holds in the entire region $\DD_{ 4\gamma  R}$.
\end{lemma}

\proof
To prove the gradient estimates, we change the time variable, defining
\begin{align}
\label{tauu}
\tau(t) = \int_0^t \frac{ds}{\lambda(s)^2} ,
\end{align}
so that \equ{cv} becomes in the variables $(y,\tau)$
\begin{align}
\pp_\tau  \phi     &=      L_{ W  }  [\phi]       + h(y, \tau )\inn  \DD_{4\gamma R}
\nonumber\\
\phi(\cdot, 0 ) &=   0 \inn B_{4R(0)}
\nonumber\end{align}

Let $\tau_1>0$ and $y_1\in B_{3\gamma R(\tau_1)}(0)$.
Let $\rho = \frac{|y_1|}{5}+1$ so that $B_{\rho}(y_1) \subset B_{4\gamma R(\tau_1)}(0)$.
%
Let us define
$$
\ttt \phi( z, t) := \phi(  y_1 + \rho z,  \tau_1 +  \rho^2 s ) ,
\quad z\in B_1(0) , \quad s > -\frac{\tau_1}{\rho^2} .
$$

We distinguish two cases. First, when $\tau_1\geq \rho^2$, we use interior estimates for parabolic equations, while for the case  $\tau_1< \rho^2$, we use estimates for a parabolic equation with initial condition.

Assume $\tau_1\geq \rho^2$.
Then $\ttt \phi( z, s) $ satisfies an equation of the form
$$
\ttt \phi_s = \Delta_z \ttt \phi  + 
A \nn_z\ttt\phi +  B \ttt \phi   +     \ttt  h  (z,s)  \inn B_{1}(0) \times (-1,0] 
$$
with coefficients $A(z,s)$ and $B(z,s)$  uniformly bounded by $O((1+\rho)^{-2})$ in $B_1(0) \times (-1,0]$ and
$$
 \ttt  h  (z,s) =  \rho ^2 h( y_1  + \rho z , \tau_1 + \rho^2 s ) .
$$
Since $\rho \leq C R(\tau_1) $ and $R(\tau_1)^2 \ll \tau_1$ for $\tau_1$ large we get
\begin{align*}
\la_*(\tau_1 )^b \lesssim \la_*( \tau_1+\rho^2 s )^b  \lesssim \ \la_*(\tau_1)^b ,
\quad s \in (-1,0] .
\end{align*}
Standard parabolic estimates and assumption \equ{M} yield
  \begin{align*}  \|\nn_z \ttt \phi\|_{L^\infty ( B_{\frac 14 }(0) \times (1,2)) }  & \lesssim  \
  \|\ttt \phi\|_{L^\infty  B_{\frac 12 }(0) \times (0,2) }  + \| \ttt h\|_{L^\infty  (B_{\frac 12 }(0) \times (0,2)) } \\ &   \ \lesssim \  M\,   \la_*( \tau_1)^b  \rho^{2-a},
  \end{align*}
  so that in particular
$$
\rho |\nn_y \phi( y_1 , \tau_1  )|  \ = \  |\nn_z \ttt \phi(0, 1)| \lesssim   M\,  \la_*(\tau_1)^b  \rho^{2-a} .
$$
%

In the case $\tau_1\geq \rho^2$ the argument is similar, but the equation for $\tilde \phi$ holds in $B_1(0)\times (-\frac{\tau_1}{\rho^2},0]$ and has initial condition 0 at $s=-\frac{\tau_1}{\rho^2}$.
Finally, for the last assertion we argue in similar way but using boundary rather that interior gradient estimates.
\qed

\bigskip
In addition to estimate \equ{MMM} we have a H\"older gradient estimate which is more natural to express using the variable $\tau$ defined in \eqref{tauu}
as follows. We denote
$$ \mathcal B_\ell (y,\tau) =\{  (y',\tau') \ / \  |y-y'|^2 + |\tau' -\tau| < \ell^2 \}.$$
For a function $g(y,\tau)$, a number $0<\alpha<1$, and  a set $A$  we let
$$
[ g ]_{\alpha, A} :=  \sup \Big  \{  \frac { |f(y,\tau) - f(y',\tau') |} {(|y-y'|^2 + |\tau' -\tau|)^{\frac \alpha 2}}\ /\  (y,\tau) ,\, (y',\tau')\in A    \Big \}.
$$

\begin{corollary}
Let $\phi$ be a solution of the equation \eqref{cv} with $h(y,\tau) = \div H(y,\tau)$.
Given  $\alpha \in (0,1)$ and constants $a$, $b$, $\gamma $  there is $C$ such that if
\[
|\phi (y,\tau)|    +  (1+ |y|)  |H(y,\tau)|
+   (1+ |y|)^{1+\alpha} [  H ]_{\mathcal B_\ell(y) (y,\tau)\cap \DD_{ 4\gamma  R} }
\le  M\,  \lambda_*(\tau)^b  (1+|y|)^{-a}
\]
in $ \DD_{4\gamma R}$,
where $\ell(y) =  1+ \frac{|y|} 4 $, then
\begin{align}
\label{MMM2}
(1+ |y|)   |\nn_y \phi (y , \tau )|
+   (1+ |y|)^{1+\alpha} [\nn_y \phi ]_{\mathcal B_{\ell(y)} (y,\tau)\cap \DD_{ 4\gamma  R} }
\le  C \, M  \lambda_*(t)^b  (1+|y|)^{-a}
\end{align}
in $ \DD_{ 3\gamma  R} $.
If in addition we know that $\phi$ satisfies the boundary condition
$\phi(\cdot ,t) = 0$ on $ \pp B_{4\gamma R(t)} $ for all $ t \in (0,T)$ then estimate \equ{MMM2} holds in the entire region $\DD_{ 4\gamma  R}$.
\end{corollary}

\bigskip
Our next goal is to construct an inverse for modes $k=-1,0,1$ with a better control when subject to a certain solvability condition.

\bigskip

\subsection{Mode \texorpdfstring{$k=0$}{k=0}}
Let us consider again equation \equ{eqL00k} for $k=0$ and the functions $Z_{0j}(y)$ defined in \eqref{ZZ} .  
 We have the following result.

\begin{lemma}\label{step2}
Let assume that $2<a<3$, $k=0$ and
\be\label{ortogonal1}
\int_{\R^2} h_0(y,t) \cdot  Z_{0j}(y) \, dy \ =\ 0 \foral t\in [0,T)
\ee
for $j=1,2$.
Then there exist  a solution $\phi_0$ to equation $\equ{eqL00k}$ for $k=0$ that defines a linear operator of $ h_0 $ and satisfies the estimate in $\DD_{3R}$,
\be\label{modo0}
|\phi_0 (y, t)|  \lesssim  \|h_0\|_{a,\nu} R^{\frac{5-a} 2} \lambda_*^\nu  (1+|y|)^{-1} \,
 \min\{1,  R^{\frac{5-a}2} |y|^{-2} \}  \ .
\ee
\end{lemma}
A central feature of  estimate \equ{modo0} 
is that it matches the size of the solutions  obtained in
Lemma \ref{step1} for $k\ne 0,1$ when $|y| \sim  R$.

\begin{proof}
We observe that conditions \equ{ortogonal1} 
can be written as
\be\label{ortogonal}
\int_0^{2R} \ttt h_0(\rho,t) \, Z_0(\rho) \rho \, d\rho \ =\ 0 \foral \tau\in (0,T).
\ee
Let us consider the complex valued functions
$$
\ttt H_0(\rho ,t ) : =  -Z_0(\rho ) \int_{\rho}^\infty   \frac 1{ sZ_0(s)^2} \int_s^\infty   \ttt h_0(\zeta,t ) Z_0(\zeta) \zeta\, d\zeta , \quad k=0,1.
$$
They are well-defined thanks to \equ{ortogonal}.
Then the function
$$ H_0(y,t) :=  {\rm Re }\, (\ttt H_0(\rho ,\tau )  ) E_1 (y)+ {\rm Re }\, (\ttt H_0(\rho ,t ))   E_2 (y) $$
solves
\[
L_ W  [ H_0(y,\tau)  ] = h_0(y,\tau)\inn \DD_{4R}
\]
and satisfies
$$
|H_0(y,t)|\ \lesssim\ \la_*(t)^{\nu} (1+|y|)^{2-a}   \| h_0 \|_{a,\nu} \inn \DD_{4R}.
$$
Moreover, elliptic gradient estimates yield
$$
|\nn_y H_0(y,\tau)| \lesssim  \lambda_*(t)^\nu (1+|y|)^{1-a}   \| h_0 \|_{a,\nu} \inn \DD_{3R}.
$$
Let us consider the problem
\begin{align}
\label{eqL00kk}
\la^2 \Phi_t & =    L_ W  [\Phi]  + H_0(y,t)   \inn \DD_{4R},   \\
\Phi(y ,0) & =   0 \inn B_{4R}(0)   \nonumber\\ \nonumber
\Phi(y,t) & = 0 \quad \hbox{for all }  t\in (0,T), \quad y\in \pp B_{4R(0)}(0)
\end{align}
According to Lemma \ref{step1}, this problem has unique solution $\Phi= \Phi_0$ that satisfies the estimates
\begin{align*}
|\Phi_0(y,t) | \ \le  & \  C \|H_0\|_{a-2,\nu}   \lambda_*(\tau)^{\nu } (1+ |y|)^{-1}
\ \ R^{5-a } \inn \DD_{4R}  .
\end{align*}
Applying Lemma \ref{gradient} we deduce that, also,
\begin{align*}
| \nn_y \Phi_0(y,t) |  \lesssim    \|H_0\|_{a-2,\nu}   \lambda_*(\tau)^{\nu } (1+ |y|)^{-2}
\ \ R^{5-a } \inn \DD_{3R}
\end{align*}
Let us write
$$
\Phi_{0j} :=  \pp_{y_j} \Phi_0 , \quad   H_{0j} :=   \pp_{y_j} H_0
$$
Then we have
\begin{align*}
\la^2\pp_t {\Phi_{0j} } & =  L_ W  [\Phi_{0j} ]  +   \pp_{y_j} |\nn  W  |^2\Phi_0  + 2\nn \pp_{y_j} W  \nn \Phi_0  + H_{0j} (y,\tau)
\\
& \quad +   2(\nn \Phi_0 \pp_{y_j} \nn W) W+  2(\nn \Phi_0  \nn W)\pp_{y_j}W
 \inn \DD_{3R},
 \\
\Phi_{0j}(y,0) & =0 \quad \hbox{for all }  y\in  B_{3R(0)}(0)
\end{align*}
According to Lemma \ref{gradient} and the above estimates we obtain that
\begin{align*}
(1+ |y|) |\nn \Phi_{0j}(y,t)|  & \lesssim  \|h_0\|_{a,\nu}   \la_*(t)^{\nu } (1+ |y|)^{-2} R^{5-a } \\  &  \quad +
\|h_0\|_{a,\nu}   \la_*(t)^{\nu } (1+ |y|)^{4-a}\inn \DD_{3R} .
\end{align*}
Then we define
$$
\phi_0  := L_ W  [\Phi_0]
$$
so that $\phi = \phi_0$ solves
\begin{align*} 
\la^2\phi_t &  =    L_ W  [\phi ]  + h_0(y,t)   \inn \DD_{3R},   \\
\phi(y,0) &  = 0 \quad \hbox{for all }  \quad y\in  B_{3R(0)}(0)
\end{align*}
and defines a linear operator of the function $h_0$.
Moreover, observing that
$$
\left|L_ W  [\Phi_0]\right |  \lesssim  \left|D^2_y\Phi_0 \right| + O(\rho^{-4}) \left|\Phi_0 \right| + O(\rho^{-2}) \left|D_y \Phi_0 \right|
$$
we then get the  estimate
\be\label{Xxx0}
|\phi_0 (y, t)|  \lesssim      \|h_0\|_{a,\nu} R^{{5-a}} \la_*(t)^{\nu}(1+|y|)^{-3} \, .
\ee
To complete the proof of estimate \equ{modo0}, we let $\vp_0$ be the complex valued function defined as
$$\phi_0(y,t)\ =\  {\rm Re}\, (\vp_0(\rho,t))\,E_1 +  {\rm Im}\, (\vp_0(\rho,t))\,E_2$$
so that
letting $R'= R^{\frac{5-a}4}\ll R $,  using the notation in \equ{Llxx},  $\vp_0$ satisfies the equation
\begin{align}  \
\la^2 \pp_t \vp_{0}  &= \mathcal L_{0}[\vp_{0}]  + \ttt h_{0}  (\rho, t)  \inn  \ttt D_{R'} ,
\label{ppi} \\
\nonumber  \vp_0(0,\rho)  & = 0 \inn (0, R'), \nonumber
\end{align}
and from \equ{Xxx0},
we can find an explicit supersolution for the real and imaginary parts of equation \equ{ppi}, which also
dominates their boundary values at $R'$, which yields
$$
|\vp_0 ( y, t)|  \lesssim  \|h_0\|_{a,\nu}  \la_*^{\nu}\,  |R'|^2 (1+ |y|)^{-1} \, , \quad |y|< R' .
$$
Combining this estimate and \equ{Xxx0} yields the validity of \equ{modo0}. \end{proof}

We mention next a variant of Lemma~\ref{step2}, in which we weaken the hypothesis on the right hand side, allowing it to be a divergence of H\"older continuous function. This will be needed when analyzing estimates of the derivative with respect to $\lambda$ of operator $\mathcal T_{\lambda,2}$ (Proposition~\ref{prop02}).
\begin{lemma}
Let assume that $2<a<3$, $\nu>0$, and $k=0$.
Let $h_0$ have the form
\[
h_0(y,\tau) = \div H_0(y,\tau)
\]
such that
\[
(1+ |y|)  |H_0(y,\tau)|
+   (1+ |y|)^{1+\alpha} [  H_0 ]_{\mathcal B_\ell(y) (y,\tau)\cap \DD_{ 4  R} }
\le   \lambda_*(\tau)^\nu  (1+|y|)^{-a} ,
\]
in $ \DD_{4 R}$, where $\alpha \in (0,1)$ and  $\ell(y) =  1+ \frac{|y|} 4 $.
Assume also that

\begin{align*}
\int_{\R^2} h_0(y,t) \cdot  Z_{0j}(y) \, dy \ =\ 0 \foral t\in [0,T)
\end{align*}
for $j=1,2$.
Then there exist  a solution $\phi_0$ to equation $\equ{eqL00k}$ for $k=0$ that defines a linear operator of $ h_0 $ and satisfies
\begin{align*}
|\phi_0 (y, t)|  \lesssim  \|h_0\|_{a,\nu} R^{\frac{5-a} 2} \la_0^\nu  (1+|y|)^{-1}
\, \min\{1,  R^{\frac{5-a}2} |y|^{-2} \}  ,
\end{align*}
in $\DD_{3R}$.
\end{lemma}

\subsection{Mode \texorpdfstring{$k=-1$}{k=-1}}
Let us consider  equation \equ{eqL00k} for $k=-1$ and the functions $Z_{-1j}(y)$ defined in \eqref{Zmode-1} .
 We have the following result.

\begin{lemma}\label{step2b}
Let assume that $2<a<3$, $k=0$ and
\begin{align}
\nonumber
\int_{\R^2} h_{-1}(y,t) \cdot  Z_{-1j}(y) \, dy = 0 \foral t\in [0,T)
\end{align}
for $j=1,2$.
Then there exist  a solution $\phi_{-1}$ to equation $\equ{eqL00k}$ for $k=-1$ that defines a linear operator of $ h_0 $ and satisfies the estimate in $\DD_{3R}$,
\begin{align}
\nonumber
|\phi_{-1} (y, t)|  \lesssim  \|h_{-1} \|_{a,\nu}  \lambda_*^\nu  \,
 \min\{\log R,  R^{4-a} |y|^{-2} \}  .
\end{align}
\end{lemma}
\begin{proof}
The proof is essentially the same as that of Lemma~\ref{step2}.
\end{proof}

\subsection{Mode \texorpdfstring{$k=1$}{k=1}}
Now we deal with  \equ{eqL00k} for $k=1$.
For convenience we give the result for a right hand side more general than strictly need for the proof of Proposition~\ref{prop2}.
Let us assume that   $h_1$ is defined in entire $\R^2\times (0,T)$ and  that
\begin{align}
\label{formH}
h_1(y,t)  = {\rm div}_y\, G(y,t)
\end{align}
where
\begin{align}
\label{boundG}
|G(y,t)|\leq  \frac{\lambda_*(t)^\nu}{1+|y|^{a-1}} , \quad y\in \R^2, \  t\in (0,T),
\end{align}
for some $\nu>0$, $a \in (2,3)$. Then the following result holds.

\begin{lemma}\label{step3}
Let assume that $2<a<3$, $k=1$, $h_1$ has the form \eqref{formH} so that \eqref{boundG} holds and
\[
\int_{\R^2} h_1(y,t) \cdot  Z_1^j(y) \, dy \ =\ 0 \foral t\in (0,T)
\]
for $j=1,2$.
Then there exist  a solution $\phi_1$ to equation \equ{eqL00k} for $k=1$ that defines a linear operator of $ h_1 $ and
satisfies the estimate in $\DD_{3R}$,
\[
|\phi_1 (y, t)|  \lesssim \la_*(t)^\nu  (1+|y|)^{2-a }  .
 \]
\end{lemma}

From this we get directly the next result.

\begin{corollary}
Let assume that $2<a<3$, $k=1$ and
\[
\int_{B_{2R}} h_1(y,t) \cdot  Z_1^j(y) \, dy \ =\ 0 \foral t\in (0,T)
\]
for $j=1,2$.
Then there exist  a solution $\phi_1$ to equation \equ{eqL00k} for $k=1$ that defines a linear operator of $ h_1 $ and
satisfies the estimate in $\DD_{3R}$,
\[
|\phi_1 (y, t)| \ \lesssim \    \, \|h_1\|_{a,\nu} \la_*(t)^\nu  (1+|y|)^{2-a } \, \ .
 \]
\end{corollary}

Let us do the same change of the time variable as in \eqref{tauu}
so that \eqref{eqL00k} for $k=1$ in entire $\R^2$ becomes in the variables $(y,\tau)$
\begin{align}
\pp_\tau  \phi  \  &=   \   L_{ W  }  [\phi]       + h\inn  \R^2 \times (0, \infty ) ,
 \label{eqLL0011} \\ \nonumber
 \phi(\cdot,0 )\  &=  \  0 \inn \R^2  .
 \end{align}
Thus, we consider a function $h(y,\tau)$ defined in entire $\R^2\times (0,+\infty)$ of the form
\be  \label{modo1}
  h    =  {\rm Re}\, ( \ttt h e^{i\theta} ) \, E_1  +  {\rm Im}\, ( \ttt h e^{i\theta} ) \, E_2,
 \ee
 that satisfies the orthogonality conditions for $j=1,2$
\be \label{ortii}
\int_{\R^2}   h(\cdot, \tau )\cdot Z_{1j} \  =\  0  \foral \tau \in (0,\infty)
\ee
and such that $h(y,\tau) =0$ for $|y|\geq 2 R(\tau)$.

By standard parabolic theory, this problem has a unique solution, which is therefore of the form
\be  \label{modo11}
 \phi  =  {\rm Re}\, ( \vp  e^{i\theta} ) \, E_1  +  {\rm Im}\, ( \vp e^{i\theta} ) \, E_2,
 \ee
where the complex valued function $\vp(\rho, \tau)$  solves the initial value problem
\begin{align}  \label{Llxx1}
\pp_\tau \vp  \ =& \ \mathcal L_{1}[\vp]  + \ttt h  (\rho, \tau)  \inn  (0,\infty)\times (0,\infty),
\\ \vp(\rho,0)  \ = &\ 0 \inn (0, \infty ), \nonumber
\end{align}
$$
\mathcal L_1[\vp]   = \pp^2_\rho \vp   + \frac{ \pp_\rho \vp}{\rho}  -  (1 + 2\cos w  + \cos(2w) ) \frac{\vp} {\rho^2}.  $$

We have the validity of the following result.

\begin{lemma}
\label{lemmaAprioriMode1div}
Let $0<\sigma <1$, $\nu>0$.
Assume  that $h$ is mode 1, that is, has the form \equ{modo1},  satisfies the orthogonality conditions \equ{ortii}, and can be written as in \eqref{formH} with $g_j$ satisfying \eqref{boundG} where $b=1+\sigma$.
Then there exists a constant $C>0$ such that the solution $\phi$ of problem \equ{eqLL0011}  satisfies the estimate
\begin{align}
\label{boundPhiMode1div}
|\phi(y , t)| \leq C  \frac{\lambda_*(t)^\nu}{1+|y|^\sigma}.
\end{align}
\end{lemma}


For the proof of this result we will use the following Liouville type result.
\begin{lemma}
\label{lemmaLiouvilleMode1}
Let $0<\sigma<1$.
Suppose $ \ttt \phi$ satisfies
\begin{align*}
\ttt\phi_\tau    &=     L_{ W  }  [\ttt\phi]    \inn  \R^2 \times (-\infty,0] ,
\\
\int_{\R^2}  \ttt  \phi(\cdot, \tau )\cdot Z_{1j} & =  0  \foral \tau \in (-\infty,0],
\\
|\ttt \phi (y, \tau)| & \leq \frac 1{ 1+ |y|^\sigma }  \inn \R^2\times (-\infty , 0], \quad j=1,2,
\\
\ttt \phi(y, \tau) & =  {\rm Re}\, ( \ttt \vp(\rho, \tau)  e^{i\theta} ) \, E_1  +  {\rm Im}\, ( \ttt \vp (\rho , \tau)e^{i\theta} ) \, E_2 .
\end{align*}
Then  $\ttt \phi =0$.
\end{lemma}
\begin{proof}
By standard parabolic regularity  $\ttt \phi(y,\tau)$ is a smooth function.
A scaling argument shows that
\[
  (1+|y|)^{-1}|D_y \ttt\phi|  +  |\ttt \phi_\tau|  +  |D^2_y \ttt \phi|   \le C(1+|y|)^{-2-\sigma}.
\]
Differentiating the equation in $\tau$, we also get $ \pp_\tau \phi_\tau  = L_ W  [\phi_\tau ] $
and we find the estimates
\[
 (1+|y|)^{-1}|D_y \ttt\phi_\tau | \, + \,  |\ttt \phi_{\tau\tau}|\, + \, |D^2_y \ttt \phi_\tau |  \ \le \ C(1+|y|)^{-3-\sigma}.
\]
Testing suitably the equations (taking into account the asymptotic behaviors in $y$ in integrations by parts)  we find
\[
 \frac 12 { \pp_\tau}  \int_{\R^2}  |\ttt \phi_\tau|^2   + B( \ttt \phi_\tau, \ttt \phi_\tau) = 0 ,
\]
where
\[
B(\ttt \phi, \ttt \phi) =  -\int_{\R^2}    L_ W [\ttt \phi]\cdot \ttt \phi =   \int_{\R^2}   |\nn \ttt \phi|^2 -   |\nn  W  |^2 |\ttt \phi|^2 .
\]
It is useful to observe the following: since
\[
  \ttt \phi(y, \tau) \  =\   {\rm Re}\, ( \ttt \vp(\rho, \tau)  e^{i\theta} ) \, E_1  +  {\rm Im}\, ( \ttt \vp (\rho , \tau)e^{i\theta} ) \, E_2
\]
then we compute, using that $\mathcal L_1 [w_\rho]=0$,
$$
B(\ttt \phi, \ttt \phi)  =  - \int_0^\infty  \mathcal L_1 [\vp] \bar \vp \rho d\rho  \ = \     \int_0^\infty  |( w_\rho ^{-1}\ttt \vp  )_\rho |^2 w_\rho^2\rho\, d\rho \ \ge \ 0 .  \quad
$$
We also get
$$
  \int_{\R^2} |\ttt \phi_\tau |^2  = -  \frac 12  { \pp_\tau}  B(\ttt \phi, \ttt \phi).
$$
From these relations we find
$$
  \pp_\tau  \int_{\R^2}  | \ttt  \phi_\tau|^2  \le 0 , \quad   \int_{-\infty}^0  d\tau\int_{\R^2} |\ttt \phi_\tau |^2 \ < \ +\infty
$$
and hence
$
\ttt \phi_\tau = 0.
$
Thus $\ttt \phi $ is independent of $\tau$ and therefore  $L_ W  [\ttt \phi] = 0 $. Since $\ttt\phi$ is at mode 1, this implies that
$\ttt \phi $ is a linear combination of $Z_{1j}$, $j=1,2$.  Since
$\int_{\R^2} \ttt \phi\cdot  Z_{1j} = 0 $, $j=1,2$ we conclude that $\ttt \phi=0$, a contradiction.
\end{proof}

\medskip

\begin{proof}[Proof of Lemma~\ref{lemmaAprioriMode1div}]
Let us write
\begin{align*}
\| \phi \|_{b,\tau_1}  := \sup_{\tau \in (0,\tau_1)} \lambda_*(\tau)^{-\nu} \|  (1+|y|^{b})\phi \|_{L^\infty(\R^2)} .
\end{align*}

We claim that for any $\tau_1>0$  we have that
\begin{align}
\label{localBound}
\|\phi\|_{2+\sigma, \tau_1 } < + \infty.
\end{align}
Let us recall that with the transformations \eqref{modo11}
we have that the complex valued function $\vp(y, \tau)$  is radial in $y$ and solves the initial value problem
\begin{align*}
\pp_\tau \vp  &=  \Delta_{\R^2} \varphi  -  (1 + 2\cos w  + \cos(2w) ) \frac{\vp} {\rho^2}
 + \ttt h  (\rho, \tau) \quad  \text{in }  \R^2 \times (0,\infty),
\\
\vp(\cdot,0)  &=  0 \quad  \text{in }  \R^2
\end{align*}
where $\rho = |y|$, $y\in \R^2$ and $\tilde h$ is related to $h$ by \eqref{modo1}.
Let us write $\varphi = \varphi_a + \varphi_b$ where $\varphi_a$ is the unique solution to
\begin{align*}
\partial_\tau \varphi_a  &=  \Delta_{\R^2} \varphi_a +  \ttt h  (\rho, \tau) \quad  \text{in }  \R^2 \times (0,\infty),
\\
\varphi_a(\cdot,0)  &=  0 \quad  \text{in }  \R^2
\end{align*}
given by Duhamel's formula.
Using the heat kernel in $\R^2$ one readily shows that $\| \varphi_a \|_{2+\sigma, \tau_1 } < + \infty$.
Let
\begin{align*}
\partial_\tau \varphi_b  &=  \Delta_{\R^2} \varphi_b  -  (1 + 2\cos w  + \cos(2w) ) \frac{1} {\rho^2}
(\varphi_a+\varphi_b)  \quad \text{in } \R^2 \times (0,\infty),
\\
\varphi_b(\cdot,0)  &=  0 \quad  \text{in }  \R^2 .
\end{align*}
By standard linear parabolic theory $\phi_b(y,\tau)$ is
 locally bounded in time and space. More precisely, given  $R>0$ there is a $K= K(R,\tau_1)$ such that
 $$
 |\phi_b(y,\tau)| \le  K\inn B_R(0) \times (0, \tau_1].
 $$
 If we fix $R$ large and take $K_1$ sufficiently large, we see that  $K_1 \rho^{-\sigma}$ is  a supersolution for the real and imaginary parts of
 the equivalent complex valued equation \equ{Llxx1} in the region $\rho>R$. As a conclusion, we find that $|\phi_b| \le 2K_1\rho^{-\sigma}$, and therefore $\|\phi_b\|_{\sigma,\tau_1 } < +\infty $ for any $\tau_1 > 0$. This proves \eqref{localBound}.

Next we claim that
\be\label{blaa}
\int_{\R^2}   \phi(\cdot, \tau )\cdot Z_{1j}  = 0  \foral \tau \in (1,\tau_1),\quad j=1,2.
\ee
Indeed, let us test the equation against
$$
Z_{1j} \eta, \quad \eta(y)=  \eta_0 (R^{-1}|y|)
$$ where $\eta_0$ is a smooth cut-off function with $\eta_0(r)=1$ for $r<1$ and $=0$ for $r>2$ and
$R$ is an arbitrary large constant. We find that
\be\label{blabla}
\int_{\R^2} \phi(\cdot, \tau )\cdot Z_{1j}\eta \  =\
 \int_0^\tau ds \int_{\R^2}  \phi(\cdot,s) \cdot  ( L_{ W } [\eta Z_{1j}] + h \cdot Z_{1j} \eta ) .
\ee
On the other hand,
\begin{align*}
\int_{\R^2}  \phi\cdot   (L_{ W } [\eta Z_{1j}]  + h \cdot Z_{1j} \eta_R  ) \
& =
\int_{\R^2} \phi\cdot (Z_{1j} \Delta \eta + 2\nn \eta \cdot \nn Z_{1j} )  - h \cdot Z_{1j}(1- \eta_R)
\\
&=  O( R^{-2-\sigma})
\end{align*}
uniformly on $\tau\in (0,\tau_1)$.
Letting $R\to +\infty$ in \equ{blabla} we get that \equ{blaa} holds.

\medskip

Now we claim that there exists a constant $C$ such that for all $\tau_1>0$
we have the validity of the estimate
\begin{equation}
\label{esti01b}
\| \phi\|_{\sigma,\tau_1} \leq C  ,
\end{equation}
so that in particular estimate \equ{boundPhiMode1div} holds.

To prove \equ{esti01b}
we assume by contradiction the existence of sequences $\tau_1^n \to +\infty$ and $\phi_n$, $h_n$ of the form \equ{modo1}, \equ{modo11} satisfying
\begin{align*} \pp_\tau  \phi_n   &=     L_{ W  }  [\phi_n ]       + h_n \inn  \R^2 \times (1, \tau_1^n) , \\
\int_{\R^2}   \phi_n (\cdot, \tau )\cdot Z_{1j} & = 0  \foral \tau \in (1,\tau_1^n), \\
 \phi_n (\cdot,1 ) &=   0 \inn \R^2  ,
 \end{align*}
so that
\begin{align}
\label{normalization}
\| \phi_n\|_{\sigma,\tau_1^n} = 1
\end{align}
but
\[
h_n =\sum_{j=1}^2 \partial_{y_j} g_{j,n} ,
\quad
\| g_{j,n} \|_{1+\sigma,\tau_1^n}\to 0 , \quad \text{as }n\to\infty.
\]
We claim first that
\begin{align}
\label{to0uniformlyCompact}
\sup_{ 1<\tau < \tau_1^n } \tau^\nu |\phi_n(y,\tau)| \to 0
\end{align}
uniformly on compact subsets of $y \in \R^2$. If not, for some $M>0$ there are $|y_n|\le M$  and $1< \tau_2^n< \tau_1^n $
so that
\[
 (\tau_2^n)^\nu (1+|y_n|^\sigma)   |\phi( y_n , \tau_2^n )| \ge \frac 12  .
\]
Clearly we must have $\tau_2^n\to +\infty$. Let us define
\[
\ttt \phi_n (y, \tau) =   (\tau_2^n)^\nu  \phi_n ( y , \tau_2^n+ \tau ).
\]
Then
\[
\pp_\tau \ttt \phi_n  = L_ W   [\ttt \phi_n]  + \ttt h_n  \inn \R^2\times (1- \tau_2^n, 0]
\]
where $\tilde h_n\to 0$ has the form
\[
\tilde h_n = \sum_{j=1}^2 \partial_{y_j} \tilde g_{j,n} , \quad
|\tilde g_{j,n} (y,\tau) | \leq o(1) \frac{ ( \tau_2^n)^\nu}{(\tau_2^n + \tau)^\nu} \frac{1}{1+|y|^{1+\sigma}}
\]
and
\[
|\ttt \phi_n (y, \tau)| \le \frac 1{ 1+ |y|^\sigma }  \inn \R^2\times (1- \tau_2^n, 0] .
\]
From standard parabolic estimates, we find that passing to a subsequence,  $\ttt \phi_n\to \ttt \phi$ uniformly on compact subsets of $\R^2\times (-\infty , 0]$ where $\ttt \phi \ne 0$ and
\begin{align*}
\ttt\phi_\tau    &=     L_{ W  }  [\ttt\phi]    \inn  \R^2 \times (-\infty,0] ,
\\
\int_{\R^2}  \ttt  \phi(\cdot, \tau )\cdot Z_{1j} & =  0  \foral \tau \in (-\infty,0],
\\
|\ttt \phi (y, \tau)| & \leq \frac 1{ 1+ |y|^\sigma }  \inn \R^2\times (-\infty , 0], \quad j=1,2,
\\
\ttt \phi(y, \tau) & =  {\rm Re}\, ( \ttt \vp(\rho, \tau)  e^{i\theta} ) \, E_1  +  {\rm Im}\, ( \ttt \vp (\rho , \tau)e^{i\theta} ) \, E_2 .
\end{align*}
But then Lemma~\ref{lemmaLiouvilleMode1} implies that $\tilde \phi\equiv0$, which is a contradiction, and we conclude that \eqref {to0uniformlyCompact} indeed holds.

From \equ{normalization}, we have that for a certain $y_n$ with $|y_n|\to \infty $ and $\tau_2^n > 0$,
\[
 (\tau_2^n)^\nu  |y_n|^\sigma   |\phi_n( y_n , \tau_2^n )|  \ge  \frac 12  .
\]
Now we let
\[
\ttt \phi_n (z, \tau) :=   (\tau_2^n)^\nu |y_n|^\sigma \,\phi_n(   |y_n|^{-1} z    ,  |y_n|^{-2} \tau+ \tau_2^n )
\]
so that
\[
\pp_\tau \ttt \phi_n  = \Delta_z \ttt \phi_n  +  a_n \cdot \nn_z\ttt \phi_n + b_n \ttt \phi_n  +  \ttt h_n(z,\tau)
\]
where
\[
\ttt h_n(z,\tau) =  (\tau_2^n)^\nu |y_n|^{2+\sigma} \, h_n(    |y_n|^{-1} z    ,  |y_n|^{-2} \tau+ \tau_2^n ),
\]
and $|a_n|+|b_n|\to 0$ uniformly on compact sets of $\R^2\setminus \{0\}$.

Note that
\[
\ttt h_n = \sum_{j=1}^2 \partial_{z_j} \tilde g_{j,n}
\]
where
\[
 \tilde g_{j,n} (z,\tau) =  (\tau_2^n)^\nu |y_n|^{1+\sigma} \, g_{j,_n}(    |y_n|^{-1} z    ,  |y_n|^{-2} \tau+ \tau_2^n ),
\]
By assumption on $g_{j,n}$ we find that $ \tilde g_{j,n} \to 0$ uniformly on compact sets of $(\R^2 \setminus \{0\} ) \times (-\infty,0]$.
Besides $|\ttt \phi_n ( \frac {y_n}{|y_n|}, 0)| \ge \frac 12$
and
\[
|\ttt \phi_n (z, \tau)| \le     |   z|^{-\sigma} ( (\tau_2^n)^{-1} |y_n|^{-2} \tau+  1 )^{-\nu} .
\]
As a conclusion,  we may assume that  $\ttt \phi_n\to \ttt \phi\ne 0$ uniformly over compact subsets of $\R^2 \setminus\{0\}\times (-\infty ,0]$
where
\[
\ttt\phi_\tau  = \Delta_z \ttt\phi \inn \R^2 \setminus\{ 0 \}\times (-\infty ,0] .
\]
and
\[
|\ttt \phi (z, \tau)| \le    \,  |z  |^{-\sigma}  \inn \R^2 \setminus\{ 0 \}\times (-\infty ,0].
\]
Moreover, the mode 1 assumption for $\phi_n $ translates for $\ttt \phi$ into
\[
\ttt \phi (z, \tau) =  \left[\begin{matrix} \vp(\rho,\tau) e^{2i\theta}\\ 0    \end{matrix}  \right ] ,\quad   z=  \rho e^{i\theta}
\]
for a complex valued function $\vp$ that solves
\begin{equation}
\label{dss}
\vp_\tau  = \vp_{\rho\rho} + \frac{\vp_\rho}{\rho} - \frac {4\vp}{\rho^2} \inn (0,\infty)\times (-\infty, 0],
\end{equation}
\[
|\vp(\rho,\tau)| \ \le \ \rho^{-\sigma} \inn (0,\infty)\times (-\infty, 0].
\]
Let us set
$$
u(\rho ,t) =  (\rho^2 + t)^{-\sigma/2}     +\frac \ve {\rho^2}
$$
Then
$$
-u_t +  \Delta u -\frac {4u}{r^2}
<
    (\rho^2 + t)^{-\sigma/2 -1}  [  \sigma (\sigma +2)    - 4 +\frac \sigma 2] <0.
$$
It follows that the function $u(x,\tau +M)$ is a positive supersolution for the real and imaginary parts of  equation \equ{dss} in
$  (0,\infty)  \times [-M, 0] $. We find then that
$
|\vp(\rho,\tau)| \le 2 u(\rho,\tau +M)
$. Letting $M\to +\infty$ we find
$$
|\vp(\rho,\tau)| \le  \frac {2\ve}{\rho^2}
$$
and since $\ve$ is arbitrary we conclude $\vp=0$. Hence $\ttt\phi=0$, a contradiction that concludes the proof of the lemma.
\end{proof}

\medskip

\begin{proof}[Proof of Lemma~\ref{step3}]
We take $h$ to be the extension as zero of the function $h_1$ as in the statement of the lemma.
Then we let $\phi$ be the unique solution of the initial value problem \equ{eqLL0011}, which clearly defines a linear operator of $h_1$.
From Lemma \ref{lemmaAprioriMode1div}, expressing the resulting estimate in the variables $(y,t)$,  we have that for any $ t_1 \in (0,T)$
$$
|\phi (y, t)| \le  C \lambda_*(t)^\nu  (1+ |y|)^{-\sigma} \| h\|_{2+\sigma,t_1 }  \foral t \in (0,t_1),
\quad y\in \R^2.
$$
Then letting $\phi_1 := \phi \big|_{\DD_{3R}}$ and letting $t_1 \uparrow	T$ the result follows.
\end{proof}

\subsection{Proof of Proposition~\ref{prop2}}

We let $h$ be defined in $\DD_{2R}$ with $\|h\|_{a,\nu}<+\infty$, with $ a \in (2,3)$, $\nu>0$.
We consider  the problem
\begin{align*}
\lambda^2 \pp_t  \phi   &=      L_{ W  }  [\phi]       + h\inn  \DD_{4R}
\phi(\cdot,0 )  &  \inn B_{4R(0)} ,
\end{align*}
(recall that $h$ is assumed to be defined in $\R^2 \times (0,T)$.
Let $\phi_k$ be the solution estimated in Lemma \ref{step1}
of
\begin{align*}
\lambda^2 \pp_t  \phi_k    &=      L_{ W  }  [\phi_k]       + h_k\inn  \DD_{4R} \\
 \phi(\cdot,t )   &=    0 \onn \pp B_{4R} \foral t\in (0,T) ,\\
 \phi(\cdot,0 )  &=    0 \inn B_{4 R(0)} .
 \end{align*}
 In addition we let $\phi_{01}$, $\phi_{11}$, $\phi_{-11}$ solve
\begin{align*}
\lambda^2 \pp_t  \phi_{k1}    &=     L_{ W  }  [\phi_{k1}]       +  \bar h_k\inn  \DD_{4R} \\
 \phi_{k1}(\cdot,t )  &=    0 \onn \pp B_{4R} \foral t\in (0,T) ,\\
 \phi_{k1}(\cdot,0 )  &=    0 \inn B_{4 R(0)}
 \end{align*}
 for $k=0,1,-1$.
Let us consider the functions $\phi_{02}$ constructed in Lemma \ref{step2},
$\phi_{-1,2}$ constructed in Lemma \ref{step2b},
and
$\phi_{12}$ constructed in Lemma \ref{step3},
 that solve for $k=0,1,-1$
\begin{align*}
\lambda^2 \pp_t  \phi_{k2}    &=      L_{ W  }  [\phi_{k2}]       + h_k-  \bar h_k\inn  \DD_{3R} \\
 \phi_{k2}(\cdot,0 )  &=    0 \inn B_{3 R(0)} .
 \end{align*}
We define
\[
\phi  : =   \sum_{k=0,1,-1} (\phi_{k1} +\phi_{k2})  +  \sum_{k\ne 0,1,-1}  \phi_k
\]
which is a bounded solution of the equation
\[
\lambda^2 \phi_t  =     L_ W  [\phi ]  + h (y,t)\inn \DD_{3R}
\]
that defines a linear operator of $h$.  Applying the estimates for the components in Lemmas \ref{step1}, \ref{step2}, \ref{step2b}, and \ref{step3}
we obtain
\begin{align*}
 |\phi(y,t)| &  \lesssim       \frac{ \lambda_*(t)^{\nu}  \log R(t) }{1+|y|^{a-2}}\,   \|h^\perp \|_{a,\nu}
\\
&  \quad
+   \frac{ \lambda_*(t)^{\nu} }{ 1+ |y|^{a-2} }\, \left \| h_1 - \bar h_1\right  \|_{a,\nu}     +   \frac{ \lambda_*(t)^{\nu} R^{4}} {1+ |y|^2} \left \| \bar h_{1} \right  \|_{\nu,a}  \\
& \quad
+
  \frac {  \lambda_*(t)^{\nu}  R^{\frac{5-a}2} } { 1+|y| } \,
 \min\{1,  R^{\frac{5-a}2} |y|^{-2} \}
 \, \| h_0 -\bar h_{0} \|_{a,\nu} +     \frac{   \lambda_*(t)^{\nu}   R^{2}} {1+ |y|}  \|\bar h_0\|_{a,\nu}
 \\
& \quad
+
\lambda_*^\nu  \min \{ \log R , R^{4-a} |y|^{-2} \}
 \, \| h_{-1} -\bar h_{-1} \|_{a,\nu} +     \lambda_*(t)^\nu  \log R \|\bar h_{-1}\|_{a,\nu}  ,
\end{align*}
in $\DD_{3R}$. Finally, Lemma \ref{gradient} yields that the same bound is valid for
$ (1+|y|)  |\nn_y \phi | $
in $\DD_{2R}$.
The function $\phi\big |_{\DD_{2R}}$ solves \equ{eqL00}, it defines a linear operator of $h$
and satisfies the required estimates.
\qed

\subsection{Modified theory for mode 0}
Let us consider the problem
\begin{align}
\label{linear1}
\left\{
\begin{aligned}
& \lambda^2 \varphi_t =  L_{W} \varphi  + h(y,t) + \sum_{j=1,2} \tilde c_{0j} Z_{0j} w_\rho^2
\quad \text{in } \mathcal D_{2R}
\\
& \varphi \cdot W = 0 \quad \text{in } \mathcal D_{2R} \\
& \varphi =0 \quad \text{on } \partial B_{2R}\times (0,T)\\
& \varphi(\cdot,0)= 0 \quad \text{in }   B_{2R(0)}  ,
\end{aligned}
\right.
\end{align}
in mode 0.
The result here is the following.

\begin{prop}
\label{prop1-modified}
Let $\sigma\in (0,1)$, $\delta \in (0,1)$, $\nu>0$.
Assume $\| h \|_{\nu,2+\sigma}<\infty$. Then there is a solution $\phi$, $\tilde c_{0j}$  of \eqref{linear1}, which is  linear in $h$, such that
\begin{align*}
| \varphi(y,t) |
+ (1+|y|) |\nabla_y \varphi(y,t)|
\leq C \lambda_*^{\nu} \| h \|_{\nu,2+\sigma}
\begin{cases}
\frac{R^{\delta(3-\sigma)}}{(1+|y|)^3}  & |y| \leq 2 R^\delta \\
\frac{1}{(1+|y|)^\sigma } &   2 R^\delta \leq |y| \leq R ,
\end{cases}
\end{align*}
and such that
\[
\tilde c_{0j}[h] =  -\frac{\int_{B_{\R^2}} h \cdot Z_{0j} }{\int_{\R^2} w_\rho^2 |Z_{0j}|^2} - G[h]
\]
where $G$ is a linear operator of $h$ satisfying the estimate
\begin{align}
\label{est-G}
|G[h]| \leq
C  \lambda_*^{\nu}  R^{-\delta\sigma'} \| h \|_{\nu,2+\sigma} ,
\end{align}
with $0<\sigma'<\sigma$.
\end{prop}

%

%
%

We are using the terminology {\em mode 0} from \S\ref{sectLinearTheory}, which means that $\varphi$ has the form
\begin{align*}
\varphi= \Re( \tilde \varphi e^{i\theta} ) E_1 + \Im( \tilde \varphi e^{i\theta} ) E_2
\end{align*}
where $\tilde \varphi$ is a complex valued  function of $\rho$ and $t$.
The equation $ \lambda^2 \varphi_t =  L_{W} \varphi  + h(y,t) $ (wit $h$ also in mode 0) becomes
\begin{align*}
\lambda^2 \partial_t \tilde\varphi = \mathcal L_0 \tilde \varphi + \tilde h ,
\quad
\text{where}
\quad
\mathcal L_0[\tilde \varphi]   := \partial^2_\rho \tilde \varphi
+ \frac{ 1}{\rho} \partial_\rho \tilde \varphi
-  \frac{ \cos(2w) } {\rho^2} \tilde \varphi ,
\end{align*}
and we have a similar definition for $\tilde h$.
Note that the operator $\mathcal L_0$  at $\rho = 0$ and $\rho=\infty$ is given by
$
\partial^2_\rho \tilde \varphi
+ \frac{ 1}{\rho} \partial_\rho \tilde \varphi -   \frac{1} {\rho^2} \tilde \varphi.
$
The last equation can be written as a regular parabolic PDE by setting $\hat\varphi (y,t)  =  \tilde \varphi(\rho, t)e^{-i\theta} $,  $y= \rho e^{i\theta}$,
\begin{align*}
\lambda^2 \partial_t
\hat \phi &= \Delta_y \hat \varphi+ \frac {16\hat \varphi}{(1+ |y|^2)^2} + \hat h(y,t) .
\end{align*}
Thus, instead of \eqref{linear1} we will construct a solution to (changing the notation to $\varphi$ and $h$)
\begin{align}
\label{linear1-b}
\left\{
\begin{aligned}
& \lambda^2
\varphi_t =   \Delta_y  \varphi+ \frac {16}{(1+ |y|^2)^2} \varphi + h(y,t)  +  \tilde c_0 \rho w_\rho^3
\quad \text{in } \mathcal D_{2R}
\\
& \varphi =0 \quad \text{on } \partial B_{2R}\times (0,T)\\
& \varphi(\cdot,0)= 0 \quad \text{in }   B_{2R(0)}  ,
\end{aligned}
\right.
\end{align}
with $\varphi $ complex valued of the form $\varphi (y)= e^{i\theta} \tilde \varphi(\rho,t)$ (and the same for $h$).
Here $\tilde c_0$ is complex and related to $\tilde c_{0j}$ in \eqref{linear1} by $\tilde c_0 = \tilde c_{01} + i \tilde c_{02}$.

We will construct $\varphi$ solving  \eqref{linear1-b} of the form
\[
\varphi = \eta \phi  + \psi
\]
where $
\eta(y,t) = \eta_1\Bigl( \frac{|y|}{R_1} \Bigr) $ and $\eta_1(r) = 1$ for $r\leq 1$, $\eta_1(r)=0$ for $r\geq 2$.
Here $R_1 = R^\delta$.
We find a solution to \eqref{linear1-b} if we get $\phi $, $\psi$ solving  the system
\begin{align}
\label{eqPhi}
\left\{
\begin{aligned}
\lambda^2 \partial_t \phi &=  \Delta  \phi  + B \phi  +B \psi
+ h(y,t)
+ c_{0} \rho w_\rho^3
\quad \text{in } \mathcal D_{2 R_1}
\\
\phi( \cdot ,0) & = 0 \quad \text{in } B_{2R_1(0)}  ,
\end{aligned}
\right.
\end{align}
\begin{align}
\label{eqPsi}
\left\{
\begin{aligned}
\lambda^2 \partial_t \psi & =  \Delta \psi + (1-\eta) B \psi + A \phi + (1-\eta) h(y,t)
\quad \text{in } \mathcal D_{2 R}
\\
\psi  & = 0 \quad \text{on } \partial B_{2R}\times (0,T)
\\
\psi (\cdot ,0) & = 0  \quad \text{on } B_{2R(0)}  ,
\end{aligned}
\right.
\end{align}
where
\begin{align*}
B & =  \frac {16}{(1+ |y|^2)^2} , \quad
A\phi  =\phi \Delta \eta + 2 \nabla \phi \nabla \eta - \phi \eta_t .
\end{align*}

Consider
\begin{align}
\label{eqPsi2}
\left\{
\begin{aligned}
\lambda^2 \partial_t \psi & =  \Delta \psi + (1-\eta) B \psi + h(y,t)
\quad \text{in } \mathcal D_{2 R}
\\
\psi  & = 0  \quad \text{on } \partial B_{2R}\times (0,T) ,
\\
\psi (y,0) & = 0 \quad \forall y\in  B_{2R(0)}  ,
\end{aligned}
\right.
\end{align}
with $\psi$ and $h$ of the form $\psi = \tilde \psi(\rho,t)e^{i\theta}$.
Let
\[
\| \psi \|_{\nu,\sigma}^{(1)}
=
\sup_{\mathcal D_{2R}}
\bigl\{
\lambda_*^{-\nu}(t) (1+|y|)^{\sigma}\,
\bigl[
\,
| \psi (y,t)| + (1+|y|)  |\nabla_y \psi (y,t)|
\,
\bigr]
\bigr\} .
\]
\begin{lemma}
\label{lemma1-psi}
Let $\sigma \in (0,1)$, $\nu>0$ and let $\psi$ solve \eqref{eqPsi2}.
If $R_1$ is sufficiently large, then
\begin{align}
\label{estLemma1}
\| \psi \|_{\nu,\sigma}^{(1)} \leq C \| h \|_{\nu,2+\sigma}  .
\end{align}
If in \eqref{eqPsi2} $h$ is replaced by $(1-\eta)h$ we get the additional estimate
\begin{align*}
|\psi(y,t)| + R_1 |\nabla \psi(y,t)|
\leq
C \lambda_*^\nu \frac{1}{R_1^{\sigma} } ,
\quad |y| \leq 2 R_1 .
\end{align*}
\end{lemma}
\begin{proof}
To prove this lemma, we first claim that for the equation
\begin{align*}
\left\{
\begin{aligned}
\lambda^2 \partial_t \psi & =  \Delta \psi +  h(y,t)
\quad \text{in } \mathcal D_{2 R}
\\
\psi  & = 0  \quad \text{on } \partial B_{2R}\times (0,T) ,
\\
\psi (y,0) & = 0 \quad \forall y\in  B_{2R(0)}  ,
\end{aligned}
\right.
\end{align*}
with $\psi$ and $h$ of the form $\psi = \tilde \psi(\rho,t)e^{i\theta}$.
we have
\begin{align}
\label{est1}
\|\psi\|_{\nu,\sigma}^{(1)} \leq C \|h\|_{\nu,2+\sigma} .
\end{align}
This is obtained using a barrier for the real and imaginary parts of $\tilde \psi$, which satisfies
\[
\lambda^2
\partial_t \tilde \psi = \partial_{\rho\rho}\tilde \psi + \frac{1}{\rho} \partial_\rho \tilde \psi -\frac{1}{\rho^2} \tilde \psi + \tilde h.
\]

To find the estimate for the solution of \eqref{eqPsi2} we need to estimate $\| (1-\eta) B \psi \|_{\nu,2+\sigma}$.
We have that
\begin{align*}
(1-\eta) B \, |\psi |
& \leq
(1-\eta)
\lambda_*^{\nu} (1+|y|)^{-4-\sigma}  \|\psi\|_{\nu,\sigma}^{(1)}
\\
& \leq
R_1(0)^{-2} \lambda_*^{\nu} (1+|y|)^{-2-\sigma}  \|\psi\|_{\nu,\sigma}^{(1)} ,
\end{align*}
and  therefore
\[
\| (1-\eta) B \psi \|_{\nu,2+\sigma} \leq C R_1(0)^{-2} \|\psi\|_{\nu,\sigma}^{(1)}.
\]
Then, if $\psi$ satisfies \eqref{eqPsi2}, using \eqref{est1} we get
\begin{align*}
\|\psi\|_{\nu,\sigma}^{(1)}
& \leq C \| (1-\eta) B \psi + h \|_{\nu,2+\sigma}
\leq C R_1(0)^{-1} \|\psi\|_{\nu,\sigma}^{(1)}  + C  \| h \|_{\nu,2+\sigma} .
\end{align*}
If $R_1(0)$ is large enough, we obtain \eqref{estLemma1}.
\end{proof}

\begin{proof}[Proof of Proposition~\ref{prop1-modified}]
We use Lemma~\ref{lemma1-psi} to find a solution
$\psi[\phi]$ of \eqref{eqPsi2} with $h$ replaced by $A\phi$,
and a solution $\psi[h]$ of \eqref{eqPsi2} with $h$ replaced by $(1-\eta) h$,
so that  $\psi[\phi] + \psi[h]$ is the solution of \eqref{eqPsi}.

Let $\sigma_1 \in (0,1)$.
We also get the estimate
\begin{align}
\label{estPsi1}
\|\psi[\phi]\|_{\nu,\sigma_1}^{(1)} \leq C \|A \phi\|_{\nu,2+\sigma_1} .
\end{align}

We take $R_1 = R^\delta$ and  construct a solution of the system \eqref{eqPhi}, \eqref{eqPsi}. For this it suffices to find $\phi$ such that
\begin{align}
\label{eqPhi3}
\left\{
\begin{aligned}
\lambda^2
\partial_t \phi &=  \Delta \phi  + B\phi + B\psi[\phi] + B\psi[h]
+ h(y,t)
+ c_0 \rho w_\rho^3
\quad \text{in } \mathcal D_{2 R_1}
\\
\phi( \cdot ,0) & = 0 \quad \text{in } B_{2R_1(0)}  .
\end{aligned}
\right.
\end{align}
Let $\mathcal T$ denote the linear operator given by Lemma~\ref{step2},
Applied in $\mathcal D_{2 R_1}$.
Then to solve \eqref{eqPhi3} we consider the fixed point problem
\[
\phi  = \mathcal T[ B \psi[\phi] + B\psi[h] + h ] .
\]

Let $\sigma \in (0,1)$.
By Lemma~\ref{step2},
\begin{align}
\label{estMode0Main}
\| \mathcal T [ g ] \|_{*,\nu,2+\sigma} \leq \|g\|_{\nu,2+\sigma},
\end{align}
where
\begin{align*}
\|\phi\|_{*,\nu,\sigma} &= \sup \,  \frac{ \lambda_*^{-\nu} (1+|y|)^{3}}{R_1^{3-\sigma}}
\left[ | \phi(y,t)| +(1+|y|) |\nabla_y \phi(y,t)|
\right] .
\end{align*}

We claim that if $ \sigma_1< \sigma $
then
\begin{align}
\label{estAphi}
\| A \phi \|_{\nu,2+\sigma_1}  \leq C R_1(0)^{\sigma_1-\sigma} \|\phi\|_{*,\nu,\sigma}.
\end{align}
Indeed, we have
\begin{align*}
|\phi \Delta \eta|
& \leq
\frac{1}{R_1^2} \lambda_*^{\nu} \frac{R_1^{3-\sigma}}{(1+|y|)^3} |\Delta \eta_1| \|\phi\|_{*,\nu,\sigma}
\leq
C \lambda_*^{\nu} \frac{R_1^{ \sigma_1  - \sigma }}{(1+|y|)^{2+\sigma_1}}  \|\phi\|_{*,\nu, \sigma }
\\
& \leq
C R_1(0)^{ \sigma_1  -  \sigma  }\lambda_*^{\nu} \frac{1}{(1+|y|)^{2+\sigma_1}}  \|\phi\|_{*,\nu, \sigma } .
\end{align*}
Similarly
\begin{align*}
|\nabla \phi \nabla \eta |
&\leq
\frac{1}{R_1} \lambda_*^{\nu} \frac{R_1^{3-\sigma}}{(1+|y|)^4} |\nabla \eta_1|
\|\phi\|_{*,\nu,\sigma}
\leq
C \lambda_*^{\nu} \frac{R_1^{ \sigma_1-\sigma }}{(1+|y|)^{2+\sigma_1}}
\|\phi\|_{*,\nu,\sigma} .
\end{align*}
Similar estimates for the remaining terms in $A$ prove \eqref{estAphi}.

From \eqref{estPsi1} and  \eqref{estAphi} we find
\begin{align}
\label{estPsi2}
\|\psi [\phi] \|_{\nu,\sigma_1}^{(1)}
\leq   C
R_1( 0 )^{\sigma_1 - \sigma}\|\phi\|_{*,\nu,\sigma} .
\end{align}

\medskip

Now we claim that
\begin{align}
\label{estBpsi}
\| B \psi \|_{\nu,2+\sigma} \leq C \| \psi \|_{\nu,\sigma_1}^{(1)}.
\end{align}
Indeed,
\begin{align*}
B \, |\psi |
& \leq C \frac{ \lambda_*^{\nu}}{ (1+|y|)^{4+\sigma_1 } } \| \psi \|_{\nu,\sigma_1}^{(1)}
\leq
C \frac{ \lambda_*^{\nu}}{ (1+|y|)^{2+ \sigma } }  \| \psi \|_{\nu,\sigma_1}^{(1)}
\end{align*}
so \eqref{estBpsi} follows.
Combining \eqref{estBpsi} and \eqref{estPsi2} we get
\begin{align}
\nonumber
\| B \psi[\phi \|_{\nu,2+ \sigma }
& \leq
C
\|\psi [\phi] \|_{\nu,\sigma_1}^{(1)}
 \leq   C R_1( 0 )^{\sigma_1- \sigma }  \|\phi\|_{*,\nu, \sigma } .
\end{align}
From the above inequality and  \eqref{estMode0Main} we then get
\[
\| \mathcal T [ B \psi[\phi]  ] \|_{*,\nu, \sigma }
\leq C  R_1(0)^{\sigma_1-\sigma }  \|\phi\|_{*,\nu, \sigma } ,
\]
which shows that the operator $\phi \mapsto \mathcal T [ B \psi[\phi]  + B\psi[h] + h ] $ is a contraction if $R_1(0)$ is sufficiently large, and we find a unique fixed point, which satisfies the estimate
\begin{align}
\nonumber
\|\phi\|_{*,\nu, \sigma } \leq C  \| \mathcal T [ B\psi[h] + h ]  \|_{*,\nu,\sigma} .
\end{align}

Next we estimate $ \| \mathcal T [ B\psi[h] + h ]  \|_{*,\nu,\sigma} $.
We have by \eqref{estMode0Main}
\begin{align*}
\| \mathcal T [ B\psi[h] + h ]  \|_{*,\nu,\sigma}
&\leq
C \| B\psi[h] + h \|_{\nu,2+\sigma}
\\
&\leq
C \| \psi[h] \|_{\nu,\sigma}^{(1)} + \| h \|_{\nu,2+\sigma}
\leq
C \| h \|_{\nu,2+\sigma}  ,
\end{align*}
and hence
\begin{align}
\label{estPhi3}
\|\phi\|_{*,\nu,\sigma} \leq C \| h \|_{\nu,2+\sigma}  .
\end{align}

Similar to \eqref{estPsi2} we have
\begin{align*}
\|\psi [\phi] \|_{\nu,\sigma}^{(1)} \leq   C \|\phi\|_{*,\nu,\sigma}  \leq C \| h \|_{\nu,2+\sigma}
\end{align*}
and
\begin{align*}
\|\psi [h] \|_{\nu,\sigma}^{(1)}  \leq C \| h \|_{\nu,2+\sigma}  .
\end{align*}

Recalling that  $\varphi = \eta \phi + \psi$ and $R_1 = R^\delta$, we get
\begin{align*}
| \varphi(y,t) |
+ (1+|y|) |\nabla_y \varphi(y,t)|
\leq C \lambda_*^{\nu} \| h \|_{\nu,2+\sigma}
\begin{cases}
\frac{R^{\delta(3-\sigma)}}{(1+|y|)^3}  & |y| \leq 2 R^\delta \\
\frac{1}{(1+|y|)^\sigma } &   2 R^\delta \leq |y| \leq R .
\end{cases}
\end{align*}

Finally, thanks to Lemma~\ref{step2}, we have that
\[
c_{0j}[h] =  -
\frac{1}{\int_{B_1} |Z_{0j}|^2}
\left[
\int_{B_{2R_1}} h \rho w_\rho
+
\int_{B_{2R_1}} ( B \psi[\phi] + B\psi[h] ) \rho w_\rho
\right]
\]
The last term is a linear operator of $h$, which we estimate next.
A similar computation as in \eqref{estAphi} shows that
\begin{align*}
\| A \phi \|_{\nu+\delta(\sigma-\sigma_1) , 2+\sigma_1}
\leq C \|\phi \|_{*,\nu,\sigma} .
\end{align*}
This implies
\[
\| \psi[\phi] \|_{\nu+\delta(\sigma-\sigma_1) , \sigma_1 }\leq C\|\phi \|_{*,\nu,\sigma}
\]
and therefore
\begin{align*}
\left|
\int_{B_{2R_1}} B \psi[\phi] \cdot Z_{0j}
\right|
&\leq C \lambda_*^{\nu }  R^{\sigma_1-\sigma}
\|\phi \|_{*,\nu,\sigma}
\end{align*}
and using \eqref{estPhi3}
\begin{align*}
\left|
\int_{B_{2R_1}} B \psi[\phi] \cdot Z_{0j}
\right|
C \lambda_*^{\nu }  R^{\sigma_1-\sigma}
 \| h \|_{\nu,2+\sigma} .
\end{align*}

We have
for $|y|\leq 2 R^\delta$
\[
|\psi[h](y,t)| + (1+|y|) |\nabla_y \psi[h](y,t) |
\leq C R_1^{-\sigma} \| h \|_{\nu,2+\sigma}.
\]
Then for $|y|\leq 2 R^\delta$ we have
\begin{align*}
| B \, |\psi[h] |
& \leq C \lambda_*^{\nu} (1+|y|)^{-4}  R_1^{- \sigma} \| h \|_{\nu,2+\sigma} ,
\end{align*}
and hence
\begin{align*}
\left|
\int_{B_{2 R_1}} B \psi[h] \rho w_\rho
\right|
&\leq C \lambda_*^{\nu} R_1^{-\sigma} \| h\|_{\nu,2+\sigma}.
\end{align*}
We would like to have the orthogonality condition defined as an integral in $\R^2$. Note that
\begin{align*}
\left|\int_{ ( B_{2R^\delta})^c}  h \rho w_\rho \right|
&\leq
C \|h\|_{\nu,2+\sigma} \lambda_*^\nu
\int_{ ( B_{2R^\delta})^c}  \frac{1}{(1+|y|)^{3+\sigma}} \,dy
\\
&\leq
C \|h\|_{\nu,2+\sigma} \lambda_*^\nu
R^{-\delta (1+\sigma)} .
\end{align*}
Then, going back to the original notation, we get
\[
c_{0j} [h]  = - \frac{\int_{\R^2} h \cdot Z_{0j} }{\int_{\R^2} w_\rho^2 |Z_{0j}|^2} - G[h]
\]
where $G$ satisfies \eqref{est-G}.
\end{proof}

\subsection{Lipschitz bounds with respect to \texorpdfstring{$\lambda$}{lambda}}
\label{lips}
Let us consider the linear operator we constructed  in Proposition \ref{prop2} as a solution $\phi[h]= \mathcal T_{\la,1} [h]$
of problem \eqref{eqL00},
\begin{align*}
\lambda^2 \pp_t  \phi   &=     L_{ W  }  [\phi]       + h(y,t)\inn  \DD_{2R}
\\
\nonumber
 \phi(\cdot,0 ) &=   0 \inn B_{2R( 0)}
\\
\nonumber
\phi \cdot  W  &=  0 \inn   \DD_{2R}
\end{align*}
where $ \DD_{2R} =  \{ (y,t) \ /\  t \in ({0}, T) ,\  y \in  B_{2R(t)}(0) \} $,
and we assume $  h \cdot  W  =   0$ in  $\DD_{2R}$.
The purpose in this section is find estimates for directional derivatives of the operator $\mathcal T_{\la,1} [h]$ with respect
 to the parameter function  $\lambda$. Examining the construction of $\mathcal T_{\la,1} [h]$ as the superposition of the unique solutions of
 different problems, it is not hard to see that the directional derivative
\begin{align*}
\phi_\la :=  (\pp_\la \mathcal T_{\la,1} )[h] [\la_1 ] = \frac{d}{ds}  \mathcal T_{\la + s\la_1,1}[h]\big |_{s=0} \end{align*}
satisfies the equation
\begin{align*}
\lambda^2 \pp_t  \phi_\la    &=     L_{ W  }  [\phi_\la ]       - 2\frac{\la_1}{\la} ( L_{ W  }  [\phi ] +  h(y,t))\inn  \DD_{2R}
\\
\nonumber
\phi_\lambda(\cdot,0) &=0 \inn B_{2R(0)}
\end{align*}
with $\phi = \mathcal T_{\la,1} [h]$. We will find estimates for this quantity inherited from those we have already established for $\phi$.
We assume that for some positive numbers $a,b,c$  independent of $T$
we have that
$$
a\la_* (t)  \le \la (t)\le b \la_* (t) , \quad   |\la_1(t)| \le c \la_*(t) \foral t\in (0,T).
$$
The following estimate holds.

\begin{prop}
\label{propDerPhiLambda}
The function $\phi_\la$ is well defined and satisfies the estimate
\begin{align*}
& (1+|y|)\left |\nn_y \phi_\lambda(y,t)\right |  +  |\phi_\lambda(y,t)|
\\
& \lesssim\
\la_*^\nu
\frac{ R^{1+ \frac{5-a}2} \log R  } {1 + |y|}
\min \Bigl\{  \frac{ R^{1+ \frac{5-a}2}   } { |y|^2} , 1 \Bigr\}\,
\|h\|_{a,\nu}  \left \|\frac {\la_1}{\la}  \right \|_\infty  \inn \DD_{2R}.
\end{align*}
\end{prop}

\noindent
{\bf Proof of Proposition \ref{propDerPhiLambda}.}
  We recall that $\phi[h]= \mathcal T_{\la,1} [h]$ was constructed mode by mode. According to the decomposition \eqref{hh1}, \eqref{fourierH}, \eqref{hh3}, \eqref{hh4}, we can write
\begin{align}
\nonumber
\phi  = \phi_0  + \phi_1 + \phi_{-1} + \phi^\perp , \quad   h=  h_0 + h_1+h_{-1}+  h^\perp,
\end{align}
where we can assume for $k=0,1$, $j=1,2$,
$$
\int_{B_{2R}} h_k(y,t) \cdot Z_{kj}(y) dy \ =\ 0.
$$
We will give the estimates for $\phi_\lambda$ in each mode separately, writing
$$
\phi_\la  = \phi_{0\la}  + \phi_{1\la} + \phi_{-1\lambda} +  \phi^\perp_\la .
$$
We will estimate each of the terms $\phi_{0\la}$, $ \phi_{1\la}$,
$ \phi_{-1\lambda} $, $ \phi^\perp_\la $ separately.

First we give some estimates for the equation in entire space with some suitable right hand side.

\begin{lemma}
\label{lemmaDuhamelDiv}
Let $\phi$ be the solution of
$$ \begin{aligned}
\partial_\tau \phi &= \Delta_y \phi + g(y,\tau)   
\quad \text{in }\R^2 \times (0,\infty) \\
\phi (\cdot,0) &= 0
\quad \text{in }\R^2
\end{aligned}$$
given by Duhamel's formula. The following holds: let $\nu \in (0,1)$, $a\in (2,3)$.
Assume that  $g(y,\tau)= \text{\rm div}_y G(y,\tau)$
 where
$
| G(y,\tau) | \leq \frac{1}{(1+\tau^\nu) (1+|y|^{a-1})}.
$
Then
\[
|\phi(y,\tau) |\leq \frac{C (1+\log_+\tau))}{(1+\tau^\nu) (1+|y|^{a-2})}
\]
where $\log_+\tau = \max( 0  , \log\tau)$. If instead, $g$ satisfies
$
| g(y,\tau) | \leq \frac{1}{(1+\tau^\nu) (1+|y|^a)}
$,
then
\[
|\phi(y,\tau) |\leq \frac{C (1+\log_+\tau))}{1+\tau^\nu}.
\]

\end{lemma}

\begin{proof}
The proof of the first estimate directly follows from the representation formula
\begin{align*}
\phi(y,\tau)
&= \frac{1}{4\pi} \int_0^\tau \frac{1}{\tau-s} \int_{\R^2}
e^{-\frac{|y-z|^2}{4(\tau-s)}}  \text{\rm div}_z  G(z,s)\, d z \, d s
\\
&= C \int_0^\tau \frac{1}{\tau-s} \int_{\R^2}
e^{-\frac{|y-z|^2}{4(\tau-s)}} \frac{y-z}{\tau-s} \cdot  G(z,s)\, d z \, d s
\end{align*}
The second estimate is treated similarly.
\end{proof}

\medskip
\subsection{Mode 0.  Estimate of  \texorpdfstring{$\phi_{0\lambda}$}{phi0lambda}}
We claim that
\begin{align}
\label{derPhi0Lambda}
& (1+|y|)|\nn \phi_{0\la}(y,t)| + |\phi_{0\la}(y,t)|
\\
\nonumber
&  \quad \lesssim \la_*^\nu \left \|\frac{\la_1}{\la} \right \|_\infty \|h_0\|_{a,\nu}\frac{ R^{1 + \frac {5-a}2 }\log R } {|y|+1 } \begin{cases}
 1 & \hbox{ if } |y|< R^{\frac 12} \\ \frac{ R} { |y|^2} & \hbox{ if } |y|> R^{\frac 12} .  \end{cases}
\end{align}

\begin{proof}
We refer to the notation in the proof of Lemma \ref{step2} on the construction of $\phi_0$.
We recall that
$\phi_0= L_W[\Phi_0]$ where $\Phi_0$ is the unique solution of the problem \equ{eqL00kk},
\begin{align*}
  \la^2 \Phi_t  & =    L_ W  [\Phi]  + H_0(y,t)   \inn \DD_{4R},   \\
  \Phi(y ,0) & =   0 \inn B_{4R}(0)   \nonumber\\ \nonumber
   \Phi(y,\tau) &  =  0 \quad \hbox{for all }  t\in (0,T), \quad y\in \pp B_{4R(0)}(0).
\end{align*}
Then $\phi_{0\la} = L_W[\Phi_{0\la}]$ where $\Phi_{0\la}$ solves
\begin{align}\label{eqL00kk2}
\la^2 \pp_t \Phi_{0\la}  & =    L_ W  [\Phi_{0\la}]  - 2\frac{\la_1}{\la} (\phi_0  + H_0(y,t) )  \inn \DD_{4R},
\\
\Phi_{0\la}(y ,0) & =   0 \inn B_{4R}(0)
\nonumber
\\
\nonumber
\Phi_{0\la}(y,\tau) &  = 0 \quad \hbox{for all }  t\in (0,T), \quad y\in \pp B_{4R(0)}(0).
\end{align}

We recall that we obtained
\[
|\phi_0 (y, t)| \lesssim \|h_0\|_{a,\nu} R^{{5-a}} \la_*(t)^{\nu}(1+|y|)^{-3} ,
\]
and a posteriori the better estimate
\[
|\phi_0 (y, t)|  \lesssim  \|h_0\|_{a,\nu} \frac{ R^{\frac {5-a}2 }\la_*^{\nu}}{1+ |y|} \begin{cases}    1 & \quad \hbox{ if }  |y| \le  R^{\frac {5-a}4 }, \\ \frac {R^{\frac {5-a}2 }}{ |y|^2}  & \quad \hbox{ if }  |y| > R^{\frac {5-a}4 } .
\end{cases}
\]
The use of an explicit barrier in \equ{eqL00kk2} then yields
$$
|\Phi_{0\la}|  \lesssim     \la_*^\nu \|h_0\|_{a,\nu} \left \|\frac{\la_1}{\la} \right \|_\infty \frac{ R^{\frac {5-a}2 +2 }\log R} {1+ |y|}
$$
and then, arguing similarly as in the construction of $\phi_0$ we obtain the estimate for $\phi_{0\la}= L_W[\Phi_{0\la}]$,
\begin{equation}
\label{jk01}
|\phi_{0\la}(y,t)|
\lesssim  \lambda_*^\nu \|h_0\|_{a,\nu} \left \|\frac{\la_1}{\la} \right \|_\infty
\frac{ R^{\frac {5-a}2 +2 }\log R} {1+ |y|^3} .
\end{equation}

Next we want to improve this estimate, as was done in Lemma~\ref{step2}.
We have that $\phi_{0\la}$ satisfies the equation
$$
\lambda^2 \partial_t \phi_{0\la}  =   L_ W  [\phi_{0\la}]  + g(y,t)
$$
where
\begin{equation}
\label{defG111}
g=- 2\frac{\la_1}{\la} (L_W[\phi_0]  + h_0(y,t) ) .
\end{equation}
We have that
$ g(y,t) =  {\rm div}_y\,  G_0(y, t)  + G_1(y,t)$ in $ \DD_{4R}$, where
\begin{align}
\label{boungsG1111}
&  (1+|y|)|G_1(y,t) | + (1+|y|)^\alpha [  G_0 ] _{ B_\ell (y,\tau)\cap \DD_{4R}} +
|G_0(y,t) |
\\
\nonumber
& \quad
\lesssim
\left \|\frac{\la_1}{\la} \right \|_\infty \|h_0\|_{a,\nu}
\frac{ R^{\frac {5-a}2 }\la_*^{\nu}}{1+ |y|^2}
\begin{cases}    1 & \quad \hbox{ if }  |y| \le  R^{\frac {5-a}4 }, \\
\frac {R^{\frac {5-a}2 }}{ |y|^2}  & \quad \hbox{ if }  |y| > R^{\frac {5-a}4 } .\end{cases}
\end{align}
We write
\[
\phi_{0\lambda} = \phi_b + \phi_c
\]
where $\phi_b$ is given by the Duhamel formula
\begin{equation}
\nonumber
\phi_b(y,t)
=  \int_0^\tau  \frac 1{4\pi (\tau-s) } ds\, \int_{\R^2}  e^{-\frac {|y-z|^2 }{4(\tau -s)} } g(z,t_\la(s)) \,dz
\end{equation}
with $g$ given by \eqref{defG111} and   $\tau$ by \eqref{tauu},
and  let $\phi_c$ solve
\begin{equation}
\nonumber
\left\{
\begin{aligned}
\lambda^2 \pp_t \phi_c   &=      L_W  [\phi_c]
+|\nabla W|^2 \phi_b + 2 (\nabla W \cdot \nabla \phi_b) W
\quad \text{in }  \DD_{4R}
\\
 \phi_c (\cdot,t )   &=    -\phi_b \onn \pp B_{4R} \foral t\in (0,T) ,
\\
 \phi_c (\cdot,0 )  &=    0 \inn B_{4 R(0)} .
\end{aligned}
\right.
\end{equation}
Using Lemma~\ref{lemmaDuhamelDiv} we find that
\begin{equation}
\label{jk03}
|\phi_b(y,t)|   + (1+|y|) |\nn \phi_b(y,t)|
\lesssim
\left \|\frac{\la_1}{\la} \right \|_\infty \|h_0\|_{a,\nu}
\, \la_*^{\nu} R^{\frac {5-a}2 } \log R
\end{equation}
for $|y|\leq 5R$.
The above estimate implies that
\begin{equation}
\label{RHSphiC}
|\nabla W|^2 |\phi_b| + 2\left| (\nabla W \cdot \nabla \phi_b) W\right|
\lesssim
\left \|\frac{\la_1}{\la} \right \|_\infty \|h_0\|_{a,\nu}
\, \la_*^{\nu} R^{\frac {5-a}2 } \log R  ( 1+|y|)^{-3}
\end{equation}
Let $\varphi_c$ be the complex valued function defined by
\[
\phi_c(y,t)\ =\  {\rm Re}\, (\varphi_c(\rho,t))\,E_1 +  {\rm Im}\, (\varphi_c(\rho,t))\,E_2
\]
so that using the notation in \equ{Llxx},  $\varphi_c$ satisfies the equation
\begin{align}
\label{ppi2}
\left\{
\begin{aligned}
\lambda^2 \pp_t \varphi_c  & = \mathcal L_{0}[ \varphi_c ]  + \tilde g_c  (\rho, t)  \inn  \ttt D_{4 R} ,
\\
\varphi_c(0,\rho)  & =  0 \inn (0, 4 R),
\end{aligned}
\right.
\end{align}
where by \eqref{RHSphiC} $ \tilde g_c $ satisfies
\[
|\tilde g_c | \lesssim
\left \|\frac{\la_1}{\la} \right \|_\infty \|h_0\|_{a,\nu}
\, \la_*^{\nu} R^{\frac {5-a}2 } \log R  ( 1+|y|)^{-3}.
\]
We can find an explicit supersolution for the real and imaginary parts of equation \eqref{ppi2} in $\tilde D_{R^{1/2}}$ of the form
\[
\bar \varphi_c =
d(t)
Z_0(\rho) \int_\rho^{2 R^{1/2}} \frac{1}{Z_0(r)^2 r } \int_0^r Z_0(s) (1+s)^{-3} s\,dsdr
\]
where $d(t) = \left \|\frac{\la_1}{\la} \right \|_\infty \|h_0\|_{a,\nu}
\, \la_*^{\nu} R^{\frac {5-a}2 } \log R $ and $Z_0$ is defined in \eqref{Z000}.
We note that at $\rho  = R^{1/2}$ the value of $\phi_c$
satisfies, by \eqref{jk01} and \eqref{jk03}
\begin{align*}
|\phi_c(R^{1/2},t) |
&  \leq |\phi_{0\lambda} (R^{1/2},t) | +  |\phi_b(R^{1/2},t) |
\\
& \lesssim
\lambda_*^\nu \|h_0\|_{a,\nu} \left \|\frac{\la_1}{\la} \right \|_\infty
 R^{\frac {6-a}2  }\log R
\end{align*}
and on the other hand
\[
|\bar\phi_c(R^{1/2},t) |   \geq c  \left \|\frac{\la_1}{\la} \right \|_\infty \|h_0\|_{a,\nu}
\, \la_*^{\nu} R^{\frac {5-a}2 } \log R R^{1/2}
\]
for some $c>0$.
This yields
\begin{align*}
|\phi_c ( y, t)|
\lesssim
 \left \|\frac{\la_1}{\la} \right \|_\infty \|h_0\|_{a,\nu}
\, \la_*^{\nu} R^{\frac {5-a}2 +1} \log R (1+|y|)^{-1}
\quad |y|< R^{1/2} .
\end{align*}
and combining with \eqref{jk03} we get
\begin{align*}
|\phi_{0\lambda} ( y, t)|
\lesssim
 \left \|\frac{\la_1}{\la} \right \|_\infty \|h_0\|_{a,\nu}
\, \la_*^{\nu} R^{\frac {5-a}2 +1} \log R (1+|y|)^{-1}
\quad |y|< R^{1/2} .
\end{align*}
Using Schauder estimates together with \eqref{boungsG1111} we obtain \eqref{derPhi0Lambda}.

%
%
%

\end{proof}

\subsection{Mode 1. Estimate of  \texorpdfstring{$\phi_{1\lambda}$}{phi1lambda}}
From a similar argument we obtain the following estimate.
\[
(1+|y|)|\nabla_y\phi_{1\lambda}(y , t)| + |\phi_{1\lambda}(y , t)|
\leq C  \lambda_*(t)^\nu  (1+|y|)^{2-a }
\|h\|_{a,\nu}  \left \|  \frac {\la_1}{\la}  \right \|_\infty  \inn \DD_R.
\]

\subsection{Estimate of \texorpdfstring{$\phi^\perp_\lambda$}{phiPerpLambda} and \texorpdfstring{$\phi_{-1\lambda}$}{}}
We claim that for any $\sigma\in (0,1)$ we have
\begin{align}
\nonumber
(1+|y|) |\nabla \phi^\perp_\lambda(y,t)|
+
|\phi^\perp_\lambda(y,t)|
& \lesssim  \lambda_*(t)^\nu
R^{a-2} \log R
(1+|y|)^{2-a} \|h^\perp\|_{a,\nu}
\left \|\frac{\lambda_1}{\lambda} \right \|_\infty
\\
\nonumber
(1+|y|) |\nabla \phi_{-1\lambda}(y,t)|
+
|\phi^\perp_{-1\lambda}(y,t)|
&
\lesssim  \lambda_*(t)^\nu
R^{a-2+\sigma}
(1+|y|)^{2-a} \|h_{-1} \|_{a,\nu}
\left \|\frac{\lambda_1}{\lambda} \right \|_\infty .
\end{align}

\section{The \texorpdfstring{$\lambda$-$\omega$}{lambda-omega} system}
\label{secLambda}

In this section we prove Proposition~\ref{propIntegralOp}, on approximate solvability of the equation
\[
\mathcal B_0[p] (t) = a(t) , \quad t\in [0,T),
\]
where $\mathcal B_0$ is the operator defined in \eqref{defB0-new} and $a:[0,T] \to \C$ is a given continuous function.
We will also derive Lipschitz estimates that will be crucial in solving for the final adjustment of parameters $p,\xi$ by a fixed point argument in the next section.

\medskip
Consistently with the discussion in section~\ref{informal}, we assume that $\frac{1}{C_1} \leq   | a(T)|  \leq C_1$ for some $C_1$ independent of $T$.
We will construct an operator $\mathcal P$ that to a function $a$ in a suitable class assigns $p = \mathcal P[a]$ such that
\begin{align}
\label{modifiedEq}
\mathcal B_0[p ](t)  = a(t) + \Rem[a](t) , \quad \text{in } [0,T).
\end{align}
so that $\Rem[a](t)$ is a suitably small.






We construct the function $p $ in Proposition~\ref{propIntegralOp} by linearization, and the first approximation is a function $p_\kappa$ that deals with the case of  constant $a$.

First we  introduce some notation.
We work with $\kappa\in \C$ and let $p_{0,\kappa}$ be the function
\begin{align}
\label{p0kappa}
p_{0,\kappa}(t) =
\kappa |\log T|
\int_t^T \frac{1}{|\log(T-s)|^2}
\,ds
, \quad  t\leq T  ,
\end{align}
so that
\begin{align}
\label{derP0Kappa}
\dot p_{0,\kappa} (t) = - \frac{\kappa|\log T|}{|\log(T-t)|^2}.
\end{align}
We will always assume that for a large, fixed constant $C_1$ we have
\begin{align}
\label{condKappa}
\frac{1}{C_1}\leq |\kappa| \leq C_1,
\end{align}
so that we also have $
{\tilde  C}^{-1} \lambda_*\leq | p_{0,\kappa}| \leq {\tilde C}\lambda_* .$
The first term in the function $p$ constructed in Proposition~\ref{propIntegralOp} is a function close to $p_{0,\kappa}$ that actually more or less solves \eqref{modifiedEq} in the case that $a$ is constant.

\begin{lemma}
	\label{lemma1}
Given $\kappa \in \C$ satisfying \eqref{condKappa}, there is a function $p_\kappa :[-T,T]\to \C$, a  constant $c(\kappa) \in \C$, and $\mathcal R_1(\kappa)(t)$ such that
\begin{align}
\label{estLemma01}
\mathcal B_0[p_\kappa](t) = c(\kappa) + \mathcal R_1(\kappa)(t)
\end{align}
for $t\in [0,T]$,
where $\mathcal R_1(\kappa)(t)$ satisfies
\begin{align}
\label{est-rem1}
| \mathcal R_1(\kappa)(t) | \leq C \lambda_*^{\alpha_0}
\end{align}
for some $ \alpha_0 > 0$ .
\end{lemma}
We have additional estimates for  $p_\kappa$ and the remainder $\mathcal R_1(\kappa)$  constructed above. The function $p_\kappa$ can be decomposed as
\begin{align*}
p_\kappa = p_{0,\kappa} + p_{1,\kappa} .
\end{align*}
Here $p_{0,\kappa}$ is defined in \eqref{p0kappa}. The function $p_{1,\kappa}$ satisfies: given $k\in (1,2)$ there is $C$ such that
\begin{align}
\label{est-p1}
\| p_{1,\kappa} \|_{*,k+1} \leq C |\log T|^{k-1} \log^2(|\log T|)
\end{align}
and
\begin{align}
\label{lip-p1}
\| p_{1,\kappa_1} - p_{1,\kappa_2} \|_{*,k+1} \leq C |\log T|^{k-1} \log^2(|\log T|)
|\kappa_1-\kappa_2|
\end{align}
for $\kappa_1$, $\kappa_2$ satisfying \eqref{condKappa}, where the norm
$\| \  \|_{*,k}$ is defined for $g \in C([-T,T];\C)\cap C^1([-T,T);\C)$ with
$$
g(T) = 0
$$
and $k>0$ by
\begin{align}
\label{norm g}
\|g\|_{*,k} = \sup_{t\in [-T,T]}  |\log(T-t)|^{k} |\dot g(t)|,
\end{align}
(here $\dot g = \frac{d}{dt} g$).


The remainder, satisfies together with \eqref{est-rem1} the estimate for the derivative in $t$:
\begin{align}
\label{est-der-R1}
\left|
\frac{d}{dt} \mathcal R_1(\kappa)(t)
\right| \leq C \lambda_*^{\alpha_0-1}
\end{align}
and Lipschitz estimates
\begin{align}
\label{lip-R1}
|\mathcal R_1(\kappa_1)(t)- \mathcal R_1(\kappa_2)(t) |
& \leq C \lambda_*^{\alpha_0} |\kappa_1-\kappa_2|
\\
\label{lip-ddt-R1}
\left|
\frac{d}{dt} \mathcal R_1(\kappa_1)(t)
-
\frac{d}{dt} \mathcal R_1(\kappa_2)(t)
\right|
& \leq C \lambda_*^{\alpha_0-1}
|\kappa_1-\kappa_2|,
\end{align}
for $\kappa_1$, $\kappa_2$ satisfying \eqref{condKappa}.
The proof of Lemma~\ref{lemma1} and estimates
\eqref{est-p1}, \eqref{lip-p1}, \eqref{est-der-R1}, \eqref{lip-R1}, and \eqref{lip-ddt-R1} are in section~\ref{appProofLemma01}.

\medskip

For the proof of Proposition~\ref{propIntegralOp} and Lemma~\ref{lemma1}
it will be useful to isolate the main part of the operator $\mathcal B_0$, defined in \eqref{defB0-new}.
Given the asymptotic expansion of $\Gamma_l$ in \eqref{GammaNear0} we write
\begin{align*}
\mathcal B_0[p]
=
\mathcal I[p] + \tilde{\mathcal B}[p]  ,
\end{align*}
where
\begin{align}
\label{defI}
\mathcal I[p] & := \int_{-T}^{t-\lambda_*(t)^2} \frac{  \dot p(s)   }{t-s}	
\,ds ,
\qquad
\tilde{\mathcal B}[p] :=
\tilde{\mathcal B}_1[p] + \tilde{\mathcal B}_2[p]
-\Re (\dot p(t)),
\end{align}
where
\begin{align*}
\tilde B_1[p](t) &=
 e^{i \omega(t)}
\left[
\int_{-T}^{t-\lambda_*(t)^2}
\frac{\Re( \dot p (s) e^{-i\omega(t)}  )}{t-s} \Bigl( \Gamma_1( \frac{\lambda(t)^2} {t-s}) -1\Bigr)\,ds
+
\int_{t-\lambda_*(t)^2}^t
\frac{\Re( \dot p (s) e^{-i\omega(t)}  )}{t-s} \Gamma_1( \frac{\lambda(t)^2} {t-s}) \,ds\right]
\\
\tilde B_2[p](t) &=
i  e^{i \omega(t)}
\left[
\int_{-T}^{t-\lambda_*(t)^2}
\frac{\Im( \dot p (s) e^{-i\omega(t)}  )}{t-s} \Bigl( \Gamma_2( \frac{\lambda(t)^2} {t-s}) -1\Bigr)\,ds
+\int_{t-\lambda_*(t)^2}^t
\frac{\Im( \dot p (s) e^{-i\omega(t)}  )}{t-s} \Gamma_2( \frac{\lambda(t)^2} {t-s}) \,ds
\right]
\end{align*}
%
and we use the notation $p(t) = \lambda(t) e^{i \omega(t)}$.
To prove  Proposition~\ref{propIntegralOp}, we take $p$ of the form
\[
p = p_\kappa + p_2,
\]
where $p_\kappa$ is the function constructed in Lemma~\ref{lemma1}, for some $\kappa \in \C$ to be determined.
The function $p_2(t)$ will have the property
$
p_2(t) = o(p_\kappa(t)),
$
as $t\to T$.
We would like that
\begin{align}
\label{goal1a}
&
\mathcal I[p_{\kappa}](t)
+
\mathcal I[ p_2](t) +
 \tilde{\mathcal B}[p_{\kappa} + p_2 ](t)
\approx a(t)  .
\end{align}
Given $\alpha >0$,  let us decompose
$
\mathcal I [p] = S_\alpha[\dot p] + R_\alpha[\dot p]
$
where $S_\alpha $, $R_\alpha $ are defined as in \eqref{defSalpha}, \eqref{defRem}, that is
The idea is to replace $ \mathcal I[  p_2] $  by $S_{\alpha }[ \dot p_2]$
in \eqref{goal1a} to make this equation more manageable, that is, we consider
$
S_\alpha[\dot p_2] +
\tilde {\mathcal B}[p_\kappa+p_2] -
\tilde{\mathcal B}[p_{\kappa} ] +
\mathcal R_1(\kappa) = a(t) , \  t\in [0,T],
$
where we have used \eqref{estLemma01}.
We introduce one more modification, so as to have a more convenient problem to treat.
Let us split
$
S_\alpha [g]
= L_0[g] + L_1[g]
$
where
\begin{align}
\nonumber
L_0[g]
&= (1-\alpha )|\log(T-t)| g(t)
\\
\nonumber
L_1[g]
&=
( 4\log(|\log(T-t)|) -2\log(\kappa) - 2 \log(|\log(T)|) ) g(t)
+
\int_{-T}^{t-(T-t)^{1+\alpha }} \frac{g(s)}{t-s}\,ds.
\end{align}

We actually introduce one more modification to \eqref{goal1a}.
For this, it is convenient that $a$ is defined in $[-T,T]$. So, given a function $a:[0,T]\to \C$ satisfying the hypotheses of Proposition~\ref{propIntegralOp}, we extend $a$ continuously by constant for $t\leq 0$.

Let $\eta$ be a smooth cut-off function such that
$\eta(s) =1\quad \text{for } s\geq 0$,
$\eta(s) =0\quad \text{for } s\leq -\frac{1}{4}$.
The equation that we are going to solve is the following one:
\begin{align}
\label{eqP1aModified00}
L_0[\dot p_2]  + \eta(\frac{t}{T})  L_1[\dot p_2]
+\tilde{\mathcal B}[ p_\kappa +  p_2]-\tilde{\mathcal B}[ p_\kappa ]
= a(t) - \mathcal R_1(\kappa) + c \quad \text{in }[-T,T]
\end{align}
for some constant $c$. Later on we shall show that it is possible to adjust $\kappa$ so that $c=0$.

\subsection{Construction of a solution to  \texorpdfstring{\eqref{eqP1aModified00}}{eqref:eqP1aModified00}}

Since in \eqref{eqP1aModified00}  the terms $ a(t) $ and $ \mathcal R_1(\kappa)$ have similar behavior, we will consider just
\begin{align}
\label{eqP1aModified}
L_0[\dot p_2]  + \eta(\frac{t}{T})  L_1[\dot p_2]
+\tilde{\mathcal B}[ p_\kappa +  p_2]-\tilde{\mathcal B}[ p_\kappa ]
= a(t) + c \quad \text{in }[-T,T]
\end{align}

Consider the norm $\| \ \|_{\mu,l}$ defined in \eqref{normG0}.
\begin{lemma}
\label{lemma03}
Let $\mu, \alpha  \in (0,\frac{1}{2}) $ and $l\in \R$.
Assume that $\frac{1}{C_1} \leq | a(T) | \leq C_1$ and
\begin{align}
\label{hypA000}
T^{\mu} |\log T|^{1+\sigma-l}
\| a(\cdot) - a(T) \|_{\mu,l-1} \leq C_1 ,
\end{align}
for some $\sigma>0$ fixed.
Then if $T>0$ is small there is a solution $p_2$ to \eqref{eqP1aModified}
for some $c\in \C$. Moreover this solution satisfies
\begin{align}
\label{estP1}
\| \dot p_2 \|_{\mu,l} \leq C  \| a(\cdot) - a(T) \|_{\mu,l-1} .
\end{align}
\end{lemma}

For the proof of this lemma
we consider the  linear equation
\begin{align}
\label{eq3modified0}
L_0[g]  + \eta(\frac{t}{T})  L_1[g] =   f + c \quad \text{in } [-T,T].
\end{align}
We will assume that $f(T) = 0$, and hence $c=L_1[g](T)$  because all other terms in the equation vanish at $T$.
Thanks to the cut-off function $\eta(\frac{t}{T}) $, we need only to consider the values of $L_1[g](t)$
for $t\geq -\frac{T}{4}$. Then in the definition of $L_1[g]$,
$t-(T-t)^{1+\alpha } \geq t - \frac{1}{2}(T-t) \geq -T$ of $T>0$ is small.

For the right hand  side of \eqref{eq3modified0} we take the space $C([-T,T];\C)$ with $f(T)=0$ and the norm $\| f \|_{\mu,l-1}$.

The next lemma asserts the solvability of \eqref{eq3modified0}  in the weighted spaces introduced above.
\begin{lemma}
\label{lemma02}
Let  $\alpha  \in (0,\frac{1}{2})$ and $T>0$ be sufficiently small.
Assume $\|f\|_{\mu,l-1}<\infty$ where $\mu \in (0,1)$, $l \in \R$.
Then for $T>0$ small there is a solution $S[f]$ of \eqref{eq3modified0}
that defines a linear operator of $f$ and such that
\begin{align}
\label{estLemma02}
\| S[f] \|_{\mu,l} \leq C \|f\|_{\mu,l-1} .
\end{align}
\end{lemma}
\begin{proof}
We consider \eqref{eq3modified0} as a fixed point problem of the form
\begin{align}
\nonumber
g  = L_0^{-1}\left[ f  -\eta(\frac{t}{T}) \left( L_1[g](t) - L_1[g](T) \right) \right] ,
\end{align}
where $L_0^{-1}$ is defined the formula
\begin{align}
\nonumber
L_0^{-1}[f](t) = \frac{ 1 }{(1-\alpha )} \frac{ f(t) }{ |\log(T-t)|}  .
\end{align}
	
It is clear that
\begin{align}
\label{estL0inv}
\| L_0^{-1}[f] \|_{\mu,l} \leq \frac{1}{1-\alpha } \| f\| _{\mu,l-1}.
\end{align}
and a calculation shows that
\begin{align}
\label{estmu100}
\| L_1[g](\cdot) - L_1[g](T)\|_{\mu,l-1}
\leq \Bigl(\alpha  + \frac{C \log|\log T|}{|\log T |} \Bigr) \|g\|_{\mu,l}.
\end{align}
	
To estimate the integral term we decompose
\begin{align*}
\int_{-T}^{t-(T-t)^{1+\alpha }}
\frac{g(s)}{t-s}\,ds
-
\int_{-T}^T
\frac{g(s)}{T-s}\,ds
&= I_1 + I_2 +I_3
\end{align*}
where
\begin{align*}
I_1 &= \int_{t-(T-t)/2}^{t-(T-t)^{1+\alpha }} \frac{g(s)}{t-s}\,ds, \quad
I_2= \int_{-T}^{t-(T-t)/2} g(s) \left( \frac{1}{t-s}-\frac{1}{T-s}\right)\,ds , \quad
I_3 = \int_{t-(T-t)/2}^T \frac{g(s)}{T-s} \, ds .
\end{align*}
Then
\begin{align*}
| I_1 |
& \leq
\|g\|_{\mu,l}
\int_{t-(T-t)/2}^{t-(T-t)^{1+\alpha }} \frac{(T-s)^\mu}{|\log(T-s)|^l (t-s)}\,ds
\leq
\|g\|_{\mu,l}
\frac{(T-t)^\mu}{|\log (T-t)|^l}
( \alpha  |\log(T-t)| + C).
\end{align*}
and similarly
\begin{align*}
|I_2|
&\leq
C \|g\|_{\mu,l}
\frac{(T-t)^\mu}{|\log(T - t)|^l  }
, \qquad
| I_3 |
\leq C \|g\|_{\mu,l} \frac{(T-t)^\mu}{|\log(T - t)|^l  }.
\end{align*}
These estimates imply \eqref{estmu100}.
Then this inequality combined with  \eqref{estL0inv} shows that
\[
\left\|
L_0^{-1}\left[ \eta(\frac{t}{T}) \left( L_1[g](t) - L_1[g](T) \right) \right]
\right\|_{\mu,l}
\leq
\frac{1}{1-\alpha }
\Bigl(\alpha  + \frac{C \log|\log T|}{|\log T |} \Bigr) \|g\|_{\mu,l}.
\]
Then for $\alpha  \in (0,\frac{1}{2})$ and $T>0$ sufficiently small this operator is a contraction and we obtain the conclusion of the lemma.
\end{proof}

\begin{proof}[Proof of Lemma~\ref{lemma03}]
Let $S$ denote the linear operator constructed in Lemma~\ref{lemma02}.

Then to find a solution to \eqref{eqP1aModified} it is sufficient to find a solution $p_2$ of the fixed point problem
\begin{align}
\label{fixedP10}
p_2 = \mathcal A[p_2]
\end{align}
where $\tilde p = \mathcal A[p_2]$ is defined by  $\tilde p(T)=0$ and
\begin{align*}
\frac{d \tilde p }{dt}
=
S\left[
- \big( \tilde{\mathcal B}[  p_\kappa +  p_2]
-\tilde{\mathcal B}[ p_\kappa ] \big)
+   a(t) -a(T)
\right] .
\end{align*}

Let $ M_1  = C_0  \|  a(\cdot ) -a(T)  \|_{\mu,l-1}$, where $C_0$ is a sufficiently large fixed constant.
We claim that if $T>0$ is sufficiently small then $\mathcal A$ is a contraction in ball $\overline B_{M_1}$ of the space of complex valued functions $p_2 \in C^1([-T,T])$  with $p_2(T) = 0$ and with the norm $\| \dot p_2 \|_{\mu,l}$. Note that with this norm we have
\begin{align}
\nonumber
|\p_2(t) |\leq C\| \dot \p_2 \|_{\mu,l} \frac{(T-t)^{\mu+1}}{|\log(T-t)|^l}.
\end{align}
In particular,  thanks to \eqref{hypA000},  if $\|\dot\p_2\|_{\mu,l}\leq M_1$,  then
\begin{align*}
\left|   \frac{\p_2}{\lambda_*} \right|
+
\left|   \frac{\dot\p_2}{\dot\lambda_*} \right| \ll 1
\end{align*}
for $T>0$ small.
	
Let us verify that $\mathcal A$ maps $\overline B_{M_1}$ into itself.
Let $\p_2 \in \overline B_{M_1}$.
By \eqref{estLemma02} we have
\begin{align}
\label{Aestimate}
\| \mathcal A[\p_2] \|_{\mu,l}
& \leq C
\Bigl(
\| \tilde{\mathcal B}[ \p_\kappa +  \p_2] - \tilde{\mathcal B}[ \p_\kappa ] \|_{\mu,l-1}
+  \|  a(\cdot ) -a(T)  \|_{\mu,l-1}
\Bigr) .
\end{align}

After some computations,
we can check the validity of the following estimate:
 for $p_1,p_2 \in $ $\overline B_{M_1}$ we have
\begin{align}
\label{lipTildeB}
\| \tilde{\mathcal B}[ \p_\kappa +  p_1] - \tilde{\mathcal B}[ \p_\kappa  + \p_2] \|_{\mu,l-1}
\leq
C
\frac{1}{|\log T|}
\| \dot p_1 -\dot p_2 \|_{\mu,l} .
\end{align}
	
Assuming  for now this estimate  let us continue with proving that $\mathcal A$ maps $\overline B_{M_1}$ into itself.
Let $p_1 \in \overline B_{M_1}$.
By \eqref{Aestimate}
and \eqref{lipTildeB}
\begin{align*}
\| \mathcal A[p_2] \|_{\mu,l}
& \leq
C \frac{M_1}{|\log T|}
+ C  \|  a(\cdot ) -a(T)  \|_{\mu,l-1}
\leq M_1,
\end{align*}
if $T>0$ is small.
Also thanks to \eqref{estLemma02} and \eqref{lipTildeB} we see that $\mathcal A$ is a contraction in  $\overline B_{M_1}$.
This finishes the proof of the lemma.
\end{proof}

We also have a Lipschitz property of the solution constructed in Lemma~\ref{lemma03}.
\begin{lemma}
Let $\mu, \alpha  \in (0,\frac{1}{2}) $ and $l\in \R$.
Assume that for $j=1,2$, $a_j$ satisfies  $\frac{1}{C_1} \leq | a_j(T) | \leq C_1$  and \eqref{hypA000}, and let $\kappa_1$, $\kappa_2$ satisfy \eqref{condKappa}.
Then for $T>0$ is small the solution $p_2[a,\kappa]$ to \eqref{eqP1aModified} constructed in Lemma~\ref{lemma03} satisfies
\begin{align}
\nonumber
\| \dot p_2[a_1,\kappa_1] - \dot p_2[a_2,\kappa_1] \|_{\mu,l}
&  \leq C  \| a_1(\cdot) - a_1(T)
-
( a_2(\cdot) - a_2(T) )
\|_{\mu,l-1}
\\
\label{lip-p2-kappa-linfty}
\| \dot p_2[a_1,\kappa_1] - \dot p_2[a_1,\kappa_2] \|_{\mu,l}
& \leq C  \| a_1(\cdot) - a_1(T)  \|_{\mu,l-1} |\kappa_1 - \kappa_2|.
\end{align}
\end{lemma}

\subsection{H\"older estimate of the solution}
We will  show in this section that the solution constructed in Lemma~\ref{lemma03} has some H\"older regularity inherited from the one of $a$.

We then have the following result, where the H\"older semi norm $[ \ ]_{\gamma,m,l}$ is defined in \eqref{normG1}.
\begin{lemma}
\label{lemma05}
Let $\alpha  \in (0,\frac{1}{2})$,  $\mu  , \gamma \in (0,1)$, $m \leq   \mu - \gamma$, $l\in \R$.
Assume that  $\frac{1}{C_1} \leq | a(T) | \leq C_1$   and
\begin{align}
\nonumber
T^\mu |\log T|^{1+\sigma-l} \| a(\cdot) - a(T) \|_{\mu,l-1}
+ [a]_{\gamma,m,l-1}
\leq C_1 ,
\end{align}
for some $\sigma>0$.
Then the solution $p_2$ constructed in Lemma~\ref{lemma03} satisfies
\begin{align}
\nonumber
[ \dot p_2 ]_{\gamma,m	,l}
&\lesssim
\frac{T^\mu}{|\log T|}
\left(
T^{-\gamma-m}  +  \log |\log T|  \right)
\| a(\cdot) - a(T) \|_{\mu,l-1}
\\
\nonumber
& \quad
+  [ a(\cdot ) -a(T)   ]_{\gamma,m,l-1} .
\end{align}
\end{lemma}

%

The proof follows from the fixed point representation \eqref{fixedP10} and estimates in the weighted H\"older norms for the operators involved there.

We will also need a Lipschitz estimate of $p_2$ as a function of $\kappa$ and $a(t)$ in the semi norm  $[ \ ]_{\gamma,m,l}$.

\begin{lemma}
Let $\alpha  \in (0,\frac{1}{2})$,  $\mu  , \gamma \in (0,1)$, $m \leq   \mu - \gamma$, $l\in \R$.
Assume that for $j=1,2$, we have $\frac{1}{C_1} \leq | a_j(T) | \leq C_1$ and
\[
T^\mu |\log T|^{1+\sigma-l} \| a_j(\cdot) - a_j(T) \|_{\mu,l-1}
+ [a_j]_{\gamma,m,l-1}
\leq C_1 ,
\]
for some $\sigma>0$, and that $\kappa_1$, $\kappa_2$ satisfy \eqref{condKappa}.
Then the solution $p_2 = p_2[a,\kappa]$ constructed in Lemma~\ref{lemma03} satisfies
\begin{align}
\nonumber
&
[ \dot p_2[a_1,\kappa_1] - \dot p_2[a_2,\kappa_1] ]_{\gamma,m,l}
\\
\nonumber
&  \quad \lesssim
[ a_1 - a_2]_{\gamma,m,l-1}
\\
\nonumber
& \qquad
+    T^{\mu-m-\gamma} \frac{\log |\log T|}{|\log T|}
\| a_1(\cdot) - a_1(T)
-
( a_2(\cdot) - a_2(T) )\|_{\mu,l-1}  ,
\end{align}
and
\begin{align}
\nonumber
[ \dot p_2[a_1,\kappa_1] - \dot p_2[a_1,\kappa_2] ]_{\gamma,m,l}
\leq
C
\frac{T^{\mu-\gamma-m}}{|\log T|}  \|a_1(\cdot)-a_1(T)\|_{\mu,l-1}
|\kappa_1 - \kappa_2|.
\end{align}
\end{lemma}

\begin{proof}[Proof of Proposition~\ref{propIntegralOp}]
By Lemma~\ref{lemma03} there is $\p_2$ satisfying  \eqref{eqP1aModified00}, where we have used this lemma with $a$ replaced by $a- \mathcal R_1(\kappa)$,  with $ \mathcal R_1(\kappa)$ being the remainder  appearing in \eqref{estLemma01}.

Note that by \eqref{est-rem1} and using the assumption $\Theta<\alpha_0$, we have
\begin{align}
\label{weighted-norm-R1}
\| \mathcal R_1(\kappa) \|_{\Theta,l-1} \leq T^{\alpha_0-\Theta} |\log T|^{l-1} .
\end{align}
Therefore from \eqref{estP1} we find
\begin{align}
\nonumber
\| \dot p_2 \|_{\Theta,l} \leq C \big( T^{\alpha_0-\Theta} |\log T|^{l-1} +  \| a(\cdot) - a(T) \|_{\Theta,l-1} ) .
\end{align}

In equation \eqref{eqP1aModified00} the constant $c$ depends on $\kappa$ and we claim that it is possible to choose $\kappa$ satisfying \eqref{condKappa} such that $c=0$. Evaluating \eqref{eqP1aModified00} at $t=T$ we find
\begin{align*}
\int_{-T} ^T \frac {\dot \p_\kappa(s) + \dot p_2(s)}{T-s}  \, ds
= a(T) + c .
\end{align*}

We consider then the equation $c=0$ with $\kappa$ as an un known, that is, we look for $\kappa$ satisfying
\begin{align}
\label{equ-kappa}
\int_{-T} ^T \frac {\dot \p_\kappa(s) + \dot p_2(s)}{T-s}  \, ds
= a(T) .
\end{align}

Using \eqref{p0kappa},  \eqref{lip-p1} and \eqref{lip-p2-kappa-linfty} we see that
\begin{align*}
\int_{-T} ^T \frac {\dot \p_\kappa(s) + \dot p_2(s)}{T-s}  \, ds =
\kappa + \tilde f(\kappa)
\end{align*}
where $\tilde f$ satisfies
\begin{align}
\nonumber
|\tilde f(\kappa_1) - \tilde f(\kappa_2)|\leq \frac{C}{|\log T|}|\kappa_1-\kappa_2|
\end{align}
for $\kappa_1$, $\kappa_2$ satisfying \eqref{condKappa}.
It follows that there exists a unique $\kappa$ so that
\eqref{equ-kappa} holds. Moreover
\begin{align}
\nonumber
\kappa= a(T) \Bigl(1+O(\frac{1}{|\log T|}) \Bigr)
\end{align}
as $T\to 0$.

Now let us prove the estimate \eqref{ineqL1b-1}.
For this we note that what we left out in \eqref{eqP1aModified00} is $R_\alpha[\dot p_2]$. In other words, the remainder $\Rem [a]$ is just $R_\alpha[\dot p_2]$.
By Lemma~\ref{lemma05} we have
\begin{align*}
[  \dot\p_2 ]_{\gamma,m	,l}
& \leq
C \frac{T^\Theta}{|\log T|}
\left(
T^{-\gamma-m}  +  \log |\log T|  \right)
\| a(\cdot) - a(T) \|_{\Theta,l-1}
\\
& \quad
+ C \frac{T^\Theta}{|\log T|}
\left(
T^{-\gamma-m}  +  \log |\log T|  \right)
\| \mathcal R_1(\kappa) \|_{\Theta,l-1}
\\
& \quad
+ C  [ a(\cdot ) -a(T)   ]_{\gamma,m,l-1}
+ C  [  \mathcal R_1(\kappa)   ]_{\gamma,m,l-1} .
\end{align*}

Using \eqref{est-der-R1} we see that for $s \leq t$ in $ [0,T]$  such that $t-s \leq \frac{1}{10}(T-t)$ we have
\begin{align*}
\frac{|  \mathcal R_1(t)- \mathcal R_1(s)|}{(t-s)^\gamma}
& \leq \lambda_*(t)^{\alpha_0-\gamma}
\end{align*}
and since $m \leq \Theta-\gamma$, $\Theta<\alpha_0$ by hypothesis we get
\begin{align*}
[ \mathcal R_1(\kappa) ]_{\gamma,m,l-1}
\leq  C \lambda_*(0)^{\sigma}
\end{align*}
for some $\sigma>0$.
From  this and \eqref{weighted-norm-R1} we obtain
\begin{align*}
[  \dot\p_2 ]_{\gamma,m	,l}
& \lesssim
T^\sigma
+ C \frac{T^\Theta}{|\log T|}
\left( T^{-\gamma-m}  +  \log |\log T|  \right)
\| a(\cdot) - a(T) \|_{\Theta,l-1}
+ [ a]_{\gamma,m,l-1} ,
\end{align*}
for some $\sigma>0$.
Then
\begin{align*}
| R_{\alpha }[\dot \p_2]  |
& \leq
\int_{t-(T-t)^{1+\alpha }}^{t-\lambda_*(t)^2}
\frac{ | \dot\p_2(t) - \dot\p_2(s) | }{t-s}\, ds
\\
& \leq
C
\Bigl(
T^\sigma
+ C \frac{T^\Theta}{|\log T|}
\left( T^{-\gamma-m}  +  \log |\log T|  \right)
\| a(\cdot) - a(T) \|_{\Theta,l-1}
+ [ a]_{\gamma,m,l-1}
\Bigr)
\\
& \qquad
\cdot
\frac{(T-t)^{m+(1+\alpha ) \gamma}}{  |\log(T-t)|^{l}} .
\end{align*}
\end{proof}

\subsection{Proof of Lemma~\ref{lemma1}}
\label{appProofLemma01}

To do this we look for  $p_\kappa$ of the form
\[
p_\kappa = p_{0,\kappa} + p_1,
\]
where $p_{0,\kappa}$ is defined in \eqref{p0kappa},
and we would like
\begin{align}
\label{goal1}
&
\mathcal I [  p_{0,\kappa}]
+
\mathcal I[p_1]
+\tilde{\mathcal B}[p_{0,\kappa} + p_1 ](t)
- c(\kappa)
=O((T-t)^{\alpha_0 }) \quad \text{for } t\in [0,T] .
\end{align}
The idea is to replace in \eqref{goal1} the operator $ \mathcal I[  p_1] $  by $S_{\alpha _0}[ \dot p_1]$ defined in \eqref{defSalpha}  and try to solve the corresponding equation.
We claim that if $\alpha _0>0$ is small, then we can find $p_1$ such that
\begin{align}
\label{eqP1}
\mathcal I[  p_{0,\kappa}]
+
S_{\alpha _0}[ \dot p_1]
+\tilde{\mathcal B}[ p_{0,\kappa} +  p_1](t)
- c(\kappa)
=0
\quad
\text{in } [0,T],
\end{align}
for some $c(\kappa)$.
This means that instead of \eqref{goal1} we have obtained
\[
\mathcal B_0[p_{0,,\kappa}+ p_1]
- c(\kappa)  = R_{\alpha _0}[\dot p_1] \quad \text{in } [0,T].
\]

The second step is to prove that there is $\kappa$ such that $c(\kappa) = A$.
The final step is to show that
\[
|R_{\alpha _0}[\dot p_1]| \leq C ( T-t)^{\alpha _0} ,
\]
and this implies \eqref{goal1}.

\subsection*{Construction of a solution to \texorpdfstring{\eqref{eqP1}}{eqP1}}

To obtain a function $\p$ satisfying \eqref{eqP1} we formulate a fixed point problem as follows.

We decompose
\begin{align*}
S_{\alpha _0}[g] =  \tilde L_0[g] + \tilde L_1[g]
\end{align*}
where
\begin{align*}
\tilde L_0[g](t) &= (1-{\alpha _0})|\log(T-t)| g (t)  + \int_{-T}^t \frac{g(s)}{T-s}\,ds
\end{align*}
and $\tilde L_1$ contains all other terms, that is,
\begin{align*}
\tilde L_1[g](t) &=\int_{t-(T-t)}^{t-(T-t)^{1+{\alpha _0}}} \frac{g(s)}{t-s}\,ds
- \int_{t-(T-t)}^t \frac{g}{T-s}\, ds
\\
& \quad
+\int_{-T}^{t-(T-t)} g(s)\left( \frac{1}{t-s} - \frac{1}{T-s}\right)\, ds\\
& \quad
+( 4\log(|\log(T-t)|)  - 2 \log(|\log(T)|) ) g(t) .
\end{align*}

Given a continuous function $f$ in $[-T,T]$ with a certain modulus of continuity at $T$, we would like to find $g$ such that
\[
S_{\alpha _0}[ \dot g]  = f   \quad \text{in } [-T,T].
\]
We will not quite obtain this, but we will  solve a modified version of this equation.
Let $\eta$ be a smooth cut-off function such that
\begin{align}
\label{cutOff1}
\eta(s) =1\quad \text{for } s\geq 0, \quad
\eta(s) =0\quad \text{for } s\leq -\frac{1}{4}.
\end{align}
We will be able to find a function $g$ such that
\begin{align}
\label{eq3modified}
\tilde L_0[\dot g]  + \eta(\frac{t}{T}) \tilde L_1[\dot g] =   f + c \quad \text{in } [-T,T].
\end{align}

We use the norm $\| \ \|_{*,k}$ defined in \eqref{norm g} for the solution $g $ of
the above equation.
For the right hand  side of \eqref{eq3modified} we take the space $C([-T,T];\C)$ with $f(T)=0$ and the norm
\begin{align}
\label{norm f}
\| f \|_{**,k} = \sup_{t\in [-T,T]} |\log(T-t)|^k | f(t)  | .
\end{align}

\elim{Note that in \eqref{eq3modified} the expression
 $\eta(\frac{t}{T})\tilde L_1[\dot g](t)$ is well defined for $g$ of class $C^1$ in $[-T,T)$. Indeed, because of the cut-off function,  $\tilde L_1[\dot g](t)$ needs to be computed only for $t\geq -\frac{T}{4}$, and for  $t\geq -\frac{T}{4}$ the integrals appearing in $L_1[\dot g]$ are well defined, since they start at either at $-T$ or  $t-(T-t) = 2t-T \geq -\frac{T}{2}$.}

Note that in \eqref{eq3modified} the expression
 $\eta(\frac{t}{T})\tilde L_1[\dot g](t)$ is well defined for $g$ of class $C^1$ in $[-T,T)$. Indeed, because of the cut-off function,  $\tilde L_1[\dot g](t)$ needs to be computed only for $t\geq -\frac{T}{4}$, and for  $t\geq -\frac{T}{4}$ the integrals appearing in $L_1[\dot g]$ are well defined, since they start at either at $-T$ or $t-\frac{1}{2}(T-t) = \frac{3}{2}	t-\frac{1}{2}T \geq -T$.

\medskip

The next lemma gives the solvability of \eqref{eq3modified}  in the weighted spaces introduced above.
Let
\[
\Upsilon=\frac{2-{\alpha _0}}{1-{\alpha _0}}
\]
\begin{lemma}
\label{lemma L inverse}
Let $C_2>1$ be fixed, $\kappa$ satisfying \eqref{condKappa}, and assume that $k>\Upsilon-1$.
Then, there is $\bar\alpha _0>0$,  so that for $0<\alpha _0\leq\bar\alpha _0$, and  $T>0$ small,  there is a linear operator $T_1$ such that $g = T_1[f]$ satisfies \eqref{eq3modified} for some constant $c$ and
\begin{align}
\label{est L inverse}
\|g\|_{*,k+1} + |c| \leq \frac{C}{k+1-\Upsilon} \|f\|_{**,k}.
\end{align}
The constant $C$ is independent of $T$, ${\alpha _0}$.
\end{lemma}

Let
\begin{align}
\label{defE}
E(t) & := \mathcal I[p_{0,\kappa}](t) , \\
\nonumber
\tilde E(t) &= E(t) - E(T) ,
\end{align}
where $ \mathcal I$  is given by \eqref{defI},
and consider the fixed point problem
\begin{align}
\label{fixedP1}
p_1
= \mathcal A[p_1]
\end{align}
where
\begin{align}
\label{defA}
\mathcal A [p_1] =
T_1
\bigl[
-\eta \tilde E
-  \tilde{\mathcal B}[ p_{0,\kappa} + p_1 ]
\bigr] ,
\end{align}
where  $\eta$ is the cut-off function defined in \eqref{cutOff1}.

Note  that if $p_1$ is a solution of \eqref{fixedP1} then $p_1$ satisfies
\begin{align*}
\tilde L_0[\dot p_1]  + \eta(\frac{t}{T}) \tilde L_1[\dot p_1] =
\eta \tilde E -   \tilde{\mathcal B}[ p_{0,\kappa} + p_1 ](t)
 + c
\end{align*}
in $[-T,T]$ for some constant $c$.
This implies that $p_1$ satisfies
\begin{align*}
S_{\alpha_0}[\dot p_1] +  \tilde{\mathcal B}[ p_{0,\kappa} + p_1 ] -E =   c
\end{align*}
in $[0,T]$ for some possibly different constant $c$.
This is precisely the equation \eqref{eqP1}.


\begin{prop}
\label{prop1b}
Let $k>0$, $k<2$ close to 2 and ${\alpha _0}>0$ small. Then for $T>0$ small
there is a function $ p_1$ satisfying  \eqref{fixedP1} and moreover
\begin{align}
\nonumber
\| p_1 \|_{*,k+1} \leq M
\end{align}
where
\begin{align}
\label{def-M}
M = C_0  |\log(T)|^{k-1} \log(|\log(T)|)^2 ,
\end{align}
with $C_0$ a fixed large constant.

Moreover, if we denote by $p_1(\kappa)$ the solution just constructed, we have, for $\kappa_1,\kappa_2$ satisfying \eqref{condKappa}
\begin{align}
\label{p1-lipschitz}
\| p_1(\kappa_1) - p_1(\kappa_2) \|_{*,k+1}
\leq
C |\log T|^{k-1} \log(|\log T |)^2
\, |\kappa_1 - \kappa_2 |.
\end{align}
\end{prop}

\medskip

The rest of the subsection is devoted to the proof of Proposition~\ref{prop1b}.

We start with the construction of the linear operator $T_1$ in Lemma~\ref{lemma L inverse}.
We want to find an inverse for $\tilde L_0$, namely given $f$ find $g$ such that
$
\tilde L_0[ \dot g]  = f  $.
To do this, we differentiate this  equation and  we get
\begin{align}
\label{ode1}
\ddot g(t) + \frac{2-{\alpha _0}}{1-{\alpha _0}}\frac{ \dot g(t)}{(T-t)|\log(T-t)|}  = \frac{1}{1-{\alpha _0}}\frac{\dot f(t)}{|\log(T-t)|} .
\end{align}

Then we can write a particular solution for $\dot g$ to \eqref{ode1} as
\begin{align}
\label{dot-g}
\dot g(t) = \frac{f(t)}{(1-{\alpha _0}) |\log(T-t)|}
+\frac{\Upsilon-1}{1-{\alpha _0}} |\log(T-t)|^{-\Upsilon}
\int_t^T \frac{|\log(T-s)|^{\Upsilon-2}}{T-s}f(s)\,ds,
\end{align}
where $ \Upsilon = \frac{2-{\alpha _0}}{1-{\alpha _0}}$ and where we have assumed that
$ \frac{|\log(T-s)|^{\Upsilon-2}}{T-s}f(s) $
is integrable near $T$ (for example $f(s) = O(|\log(T-s)|^{-k})$ with $k>\Upsilon-1$ suffices).

Define the operator
\begin{align}
\label{a0-1}
T_0[f] &=g ,
\end{align}
where $g$ is such that $\dot g$ is given by \eqref{dot-g} and $g(T)=0$. Note that  $g = T_0[f]$ solves \eqref{ode1} and therefore
$$
\tilde L_0[\dot g] = f + c ,
$$
for some constant $c$.

\begin{lemma}
\label{lemma L0 inverse}
Assume $ k>\Upsilon-1$.
Then for $f\in C([-T,T];\C)$ with $f(T)=0$
$$
\|T_0[f]\|_{*,k+1}\leq \frac{C}{k+1-\Upsilon} \|f\|_{**,k}.
$$
The constant is independent of $\Upsilon$ (if $\Upsilon$ is bounded), $k$, $T$.
\end{lemma}
\begin{proof}
This is direct from  \eqref{dot-g}.
\end{proof}

\begin{proof}[Proof of Lemma~\ref{lemma L inverse}]

We construct $g $ as a solution of the fixed point problem
\begin{align*}
g = T_0\left[f- \eta\bigl(\frac{t}{T}\bigr) \tilde L_1[g] \right] .
\end{align*}
where $T_0$ is the operator constructed in \eqref{a0-1}
and $\eta$ is the cut-off function \eqref{cutOff1}.

By Lemma~\ref{lemma L0 inverse}
$$
\| T_0[\tilde L_1[g]] \|_{*,k+1} \leq \frac{C }{k+1-\Upsilon} \|\tilde L_1[g]\|_{**,k}.
$$
A computations shows that
\begin{align*}
\| T_0[\tilde L_1[g]] \|_{*,k+1}
&\leq  \frac{C }{k+1-\Upsilon}   \left( {\alpha _0} + \frac{1}{|\log T|}  +  \frac{\log|\log T|}{|\log T|} \right) \|g\|_{*,k+1} .
\end{align*}
we get a contraction if ${\alpha _0}>0$ is fixed small and then $T>0$ is sufficiently small.	
\end{proof}

Next we need an estimate for the error $E$ defined in \eqref{defE}.

\begin{lemma}
\label{lemma-est-E}
Let $p_{0,\kappa}$ be given by \eqref{p0kappa} and assume $\kappa \in \C$ satisfies \eqref{condKappa}.
Then
\begin{align}
\label{est-E}
|E(t) -E(T)|\leq C
\frac{ |\log T|  \log |\log(T-t)| }{|\log(T-t)|^2} ,
\quad
-\frac{T}{4} \leq t \leq T.
\end{align}
\end{lemma}
\begin{proof}
By definition we have
\[
E(t)
=
\int_{-T}^{t-\lambda_*(t)^2} \frac{\dot p_{0,\kappa}(s)}{t-s}\, ds .
\]
Let $t\in [-\frac{T}{4},T]$ and  let us write
\begin{align*}
E(t)
&=\int_{-T}^{t} \frac{\dot p_{0,\kappa}(s)}{T-s}\, ds  - \int_{t-(T-t)/5}^t  \frac{\dot p_{0,\kappa}(s)}{T-s}\, ds
+ \int_{-T}^{t-(T-t)/5} \dot p_{0,\kappa}(s)\left( \frac{1}{t-s}- \frac{1}{T-s}\right) \, ds \\
& \quad +  \int_{t-(T-t)/5}^{t- \lambda_*(t)^2} \frac{\dot p_{0,\kappa}(s)}{t-s}\, ds .
\end{align*}
We estimate
\begin{align*}
\left| \int_{t-(T-t)/5}^{t} \frac{\dot p_{0,\kappa}(s)}{T-s}\, ds \right|
&\leq \frac{C \kappa |\log T |}{|\log(T-t)|^2} ,
\end{align*}
and
\begin{align*}
\left| \int_{-T}^{t-(T-t)/5} \dot p_{0,\kappa}(s)\left( \frac{1}{t-s}- \frac{1}{T-s}\right) \, ds \right|
\leq \frac{C\kappa |\log(T)|}{|\log(T-t)|^2}.
\end{align*}
With the fourth  term in $E$ we proceed as follows
\begin{align*}
\int_{t-(T-t)/5}^{t-\lambda_*(t)^2} \frac{\dot  p_{0,\kappa}(s)}{t-s}\, ds
&= \dot  p_{0,\kappa}(t) (  \log(T-t) - 2\log(\lambda_*)) - \int_{t-(T-t)/5}^{t- \lambda_*(t)^2} \frac{\dot  p_{0,\kappa}(t)-\dot  p_{0,\kappa}(s)}{t-s}\, ds  .
\end{align*}
But
\begin{align*}
\left|\int_{t-(T-t)/5}^{t-\lambda_*(t)^2} \frac{\dot p_{0,\kappa}(t)-\dot  p_{0,\kappa}(s)}{t-s}\, ds \right|
& \leq \frac{C\kappa |\log(T)|}{|\log(T-t)|^3} ,
\end{align*}
and therefore
\begin{align*}
E =  \int_{-T}^{t} \frac{\dot  p_{0,\kappa}(s)}{T-s}\, ds  + \dot  p_{0,\kappa}(t)  ( \log(T-t)  -2\log(\lambda_*))+O(\frac{\kappa|\log(T)| }{|\log(T-t)|^2}).
\end{align*}

We note that
\[
\dot p_{0,\kappa}(t)| \log(T-t)| +  \int_{0}^{t} \frac{\dot  p_{0,\kappa}(s)}{T-s}\, ds  = c
\]
for some constant $c$. Indeed, by \eqref{derP0Kappa}
\begin{align*}
\frac{d}{dt}
\left(
 \dot  p_{0,\kappa}(t)| \log(T-t)| +  \int_{0}^{t} \frac{\dot  p_{0,\kappa}(s)}{T-s}\, ds
\right)
&= \frac{\frac{d}{dt}(\dot p_{0,\kappa}(t) |\log(T-t)|^2)}{|\log(T-t)|}
\\
&= 0 .
\end{align*}
This shows that
\begin{align*}
E(t) = E(T)
+O(\frac{ |\log T | [ \log(|\log T |) + \log(|\log(T-t)|)] }{|\log(T-t)|^2})  ,
\end{align*}
which implies the estimate \eqref{est-E}.
\end{proof}

\begin{proof}[Proof of Proposition~\ref{prop1b}]
Let $T_1$ be the operator constructed in Lemma~\ref{lemma L inverse} for $T>0$, ${\alpha _0}>0$ small and $\mathcal A$ defined in \eqref{defA}.

We will apply inequality  \eqref{est L inverse} with $k<2$ close to 2.
The constant in this inequality remains bounded  as ${\alpha _0}\to0^+$, because $\Upsilon=\frac{2-{\alpha _0}}{1-{\alpha _0}}\to 2$ as  ${\alpha _0}\to0^+$.

For the poof we  use the norm \eqref{norm f} with $k<2$, $k$ close to 2 so $k+1<3$ is close to 3.
We work with $p_1$ in the space $X=C([-T,T];\C)\cap C^1([-T,T);\C)$ with the norm $\|\cdot\|_{*,k+1}$ defined in \eqref{norm g}.
By Lemma~\ref{lemma L inverse}
\begin{align}
\label{abc1}
\|\mathcal A[p_1] \|_{*,k+1}
 \leq C
\Big(
\|  \eta \tilde  E  \|_{**,k}
+ \|
\tilde{\mathcal B}[ p_{0,\kappa} + \p	_1 ](t)  - \tilde{\mathcal B}[  p_{0,\kappa} + \p	_1 ](T)  \|_{**,k}
\big) ,
\end{align}
and  by Lemma~\ref{lemma-est-E}
\begin{align}
\label{estE1}
\|  \eta \tilde  E  \|_{**,k}
\leq C_E |\log T|^{k-1} \log(|\log T |) ,
\end{align}
for some $C_E>0$.
We take in $X$ the closed ball $\overline B_M(0)$ of center 0 and radius $M$ given by \eqref{def-M} with $C_0>0$ suitably large.
The proof of Proposition~\ref{prop1b} consists in showing that $\mathcal A : \overline B_M(0) \to \overline B_M(0)$ is a contraction.
The estimates required for this are the following:
for $\| p_1\|_{*,k+1}\leq M$
we have
\begin{align}
\label{estM1-1}
\| \tilde{ \mathcal B}[  p_{0,\kappa} + \p	_1 ]
\|_{**,k}
\leq   C |\log(T)|^{k-1},
\end{align}
and for $\| p_i \|_{*,k+1}\leq M$, $ i=1,2$ we have
\begin{align}
\label{estM1-2}
\|
\tilde{\mathcal B}[  p_{0,\kappa}+ p_1]
-
\tilde{\mathcal B}[  p_{0,\kappa}+ p_2]
\|_{**,k}
\leq \frac{C}{|\log T |}
\|p_1-p_2\|_{*,k+1} .
\end{align}
These inequalities are proved in a straightforward way. We omit the details

\medskip

Form these estimates we see that $\mathcal A$ is a contraction in the ball $\overline{B}_M$.
Indeed, from \eqref{abc1}, \eqref{estE1} and \eqref{estM1-1} we have
\begin{align*}
\|\mathcal A[p_1] \|_{*,k+1}
&\leq C \cdot C_E |\log T|^{k-1} \log(|\log T |) +
  C |\log(T)|^{k-1} \\
& \leq C_0  |\log T|^{k-1} \log(|\log T |)^2
\end{align*}
by fixing $C_0$ large. Therefore  $\mathcal A : \overline B_M(0) \to \overline B_M(0)$.

Next, for  $\| p_i \|_{*,k+1}\leq M$, $ i=1,2$, by Lemma~\ref{lemma L inverse} and \eqref{estM1-2} we get
\begin{align*}
\|\mathcal A[p_1] - \mathcal A[p_2] \|_{*,k+1}
&\leq C
\|
\tilde{\mathcal B}[  p_{0,\kappa}+ p_1]
-
\tilde{\mathcal B}[  p_{0,\kappa}+ p_2]
\|_{**,k}
\leq \frac{C}{|\log T |}
\|p_1-p_2\|_{*,k+1} .
\end{align*}

The proof of \eqref{p1-lipschitz} will be given in Corollary~\ref{coroLambda1} below.
%
%
\end{proof}

We also have the following estimates
\begin{lemma}
Let $p_1$ be the solution constructed in Proposition~\ref{prop1b}. Then
\begin{align}
\nonumber
|\ddot p_1(t)| &\leq  C \frac{  |\log T| }{|\log(T-t)|^{3} (T-t)}
\\
\nonumber
\Bigl|\frac{d^3}{dt^3}\p_1(t) \Bigr| & \leq  C \frac{  |\log T | }{|\log(T-t)|^{3} (T-t)^2}  .
\end{align}
\end{lemma}
The proof is done by formally differentiating the equation and using suitable estimates on the operators involved. We omit the details.

Proposition~\ref{prop1b} defines a function
that to $\kappa $ satisfying \eqref{condKappa} associates  $p_1(\kappa)$, which is the unique fixed point of $\mathcal A$ in the ball $\{ \| p_1 \|_{*,k+1} \leq M \}$, $M=C_0  |\log(T)|^{k-1} \log(|\log(T)|)^2 $.

The next result gives several Lipschitz estimates of this map.
\begin{corollary}
Let $k\in (0,2)$.
\label{coroLambda1}
For $\kappa_1, \kappa_2$ satisfying \eqref{condKappa} we have
\begin{align}
\nonumber
\| p_1(\kappa_1) - p_1(\kappa_2) \|_{*,k+1}
\leq
C |\log T|^{k-1} \log(|\log T |)^2
\, |\kappa_1 - \kappa_2 |.
\end{align}
\end{corollary}


We will also need a Lipschitz estimate for $\ddot p_1$ in the norm $\| \ \|_{-1,3}$ and $\frac{d^3}{d t^3} p_1$ in the norm $\| \ \|_{-2,3}$.

\begin{lemma}
For $\kappa_1, \kappa_2$ satisfying \eqref{condKappa} we have
\begin{align}
\nonumber
\| \ddot p_1(\kappa_1) - \ddot p_1(\kappa_2) \|_{-1,3}
\leq
C |\log T|
\, |\kappa_1 - \kappa_2 |
\\
\nonumber
\Bigl\| \frac{d^3}{d t^3} p_1(\kappa_1) - \frac{d^3}{d t^3}  p_1(\kappa_2) \Bigr\|_{-2,3}
\leq
C |\log T|
\, |\kappa_1 - \kappa_2 |
\end{align}

\end{lemma}

Next we use the previous results on $p_1$ to obtain an estimate of $R_{\alpha _0}[\dot p_1]$.

\begin{lemma}
Let $p_1$ be the solution constructed in Proposition~\ref{prop1b}.
Then
\begin{align}
\nonumber
|R_{\alpha _0}[\dot p_1](t)|
& \leq C
\frac{|\log T| }{|\log(T-t)|^{3}} (T-t)^{\alpha _0} ,
\end{align}
and for $\kappa_1$, $\kappa_2$ satisfying \eqref{condKappa} we have
\begin{align}
\nonumber
|R_{\alpha_0}[\dot p_1(\kappa_1)]
-
R_{\alpha_0}[\dot p_1(\kappa_2)]|
\leq C\frac{|\log T| }{|\log(T-t)|^{3}} (T-t)^{\alpha _0}
|\kappa_1-\kappa_2|.
\end{align}
\end{lemma}

\begin{lemma}
Let $p_1$ be the solution constructed in Proposition~\ref{prop1b}.
Then
\begin{align}
\nonumber
\Bigl| \frac{d}{dt}	R_{\alpha _0}[\dot p_1](t) \Bigr|
& \leq C
\frac{|\log T| }{|\log(T-t)|^{3}} (T-t)^{\alpha _0-1}  ,
\\
\nonumber
\Bigl|
\frac{d}{dt}	R_{\alpha _0}[\dot p_1(\kappa_1)](t)
-
\frac{d}{dt}	R_{\alpha _0}[\dot p_1(\kappa_2)](t)
\Bigr|
& \leq C
\frac{|\log T| }{|\log(T-t)|^{3}} (T-t)^{\alpha _0-1}
| \kappa_1 - \kappa_2|.
\end{align}
\end{lemma}

\section{Final adjustment of the parameters $p$ and $\xi$}
\label{sectSolPar}

In this section we prove that the last equations of the gluing system \eqref{eq-psi}--\eqref{1.4} can be solved, by adjusting the parameter functions $p = \lambda e^{i\omega}$ and $\xi$, as stated in Proposition~\ref{propSolPar}, thus concluding the proof of Theorem \ref{teo1}.

\begin{proof}[Proof of Proposition~\ref{propSolPar}]
Let  $ \Psi (p,\xi,\Phi,Z_0^*)$ be the solution to equation~\equ{eq-psi}  constructed in Proposition \ref{propi1}.
Let $\Phi(p,\xi,Z_0^*)$ denote the solution of \eqref{ptofijo} constructed in Proposition~\ref{innnerSystem}.
In \eqref{1.3}-\eqref{1.4} we replace $\Psi^*$ by $\Psi^*(p,\xi,\Phi(p,\xi,Z_0^*),Z_0^*)$.
Then to find a solution of the full system \eqref{eq-psi}--\eqref{1.4} it is sufficient to find $p$, $\xi$ such that
\begin{align}
\label{1.3b}
c_{0j}[h(p,\xi, \Psi^*(p,\xi,\Phi(p,\xi,Z_0^*),Z_0^*))](t) - c_{0j}^*[p,\xi,\Psi^*(p,\xi,\Phi(p,\xi,Z_0^*),Z_0^*)] (t) &= 0
\\
\label{1.4b}
c_{1j}[h(p,\xi, \Psi^*(p,\xi,\Phi(p,\xi,Z_0^*),Z_0^*))](t)  &=  0
\end{align}
for all $t\in (0,T)$, $j=1,2$.

We recall from Section~\ref{sectInnerOuter} that \eqref{1.3b} is equivalent
to
\begin{align}
\label{1.3bb}
\mathcal B_0[p] = a_0^{(0)}[p,\xi,\Psi^*]  + \mathcal R_0 \left[
 a_0^{(0)}[p,\xi,\Psi^*]
\right], \quad t \in [0,T]
\end{align}
where $\Psi^*=\Psi^*(p,\xi,\Phi(p,\xi,Z_0^*),Z_0^*)$.
We recall that $\mathcal B_0$ is the integral operator defined in \eqref{defB0-new} which has the approximate form
\begin{align*}
\mathcal B_0[p ]
=   \int_{-T} ^{t-\la^2}    \frac{\dot p(s)}{t-s}ds\, + O\big( \|\dot p\|_\infty \big).
\end{align*}
In  Proposition~\ref{propIntegralOp} we constructed an approximate inverse $\mathcal P$  of the operator $\mathcal B_0$, so that given $a$ satisfying \eqref{hypA00},  $ p := \mathcal P  \left[ a \right] $,
satisfies the equation
\[
\mathcal B_0[ p ]   = a +\Rem[ a] , \quad \text{in }[0,T],
\]
for a small remainder $\Rem[ a]$.
The proof of that proposition gives the decomposition
\begin{align*}
\mathcal P[a] = p_{0,\kappa} + \mathcal P_1[a],
\end{align*}
where $p_{0,\kappa}$ is defined in \eqref{p0kappa},  $\kappa = \kappa[a] \in \C$  and the function $ p_1 = \mathcal P_1[a]$ has the estimate
\[
\| p_1 \|_{*,3-\sigma} \leq C |\log T|^{1-\sigma} \log^2(|\log T|) ,
\]
where $\| \ \|_{^*,3-\sigma}$ is defined in \eqref{norm g} and $\sigma \in (0,1)$.
This leads us to define the space $
X_1 := \C \times \tilde X_1
$
where
\[
\tilde X_1 := \{  p_1 \in C([-T,T;\C]) \cap C^1([-T,T;\C]) \ | \ p_1(T) = 0 , \ \|p_1\|_{*,3-\sigma}<\infty \} .
\]
Let us rewrite equation \eqref{1.4b} as follows.
By \eqref{defCij}, \eqref{1.4b} is equivalent to
\begin{align*}
\int_{\R^2}
h[p,\xi, \Psi^*]
\cdot Z_{1j}(y)\, dy =0 , \quad t\in (0,T), \ j=1,2,
\end{align*}
and recalling \eqref{HH2}, this is equivalent to
\begin{align*}
\lambda \int_{B_{2R}} Q_{-\omega} \tilde L_U[\Psi^*] \cdot Z_{1j}
+ \lambda  \int_{B_{2R}} \KK_1[p,\xi] \cdot Z_{1j}=0 ,
\end{align*}
which yields the following equation
\begin{align}
\label{1.4bb}
\dot \xi_j = \frac{1}{4\pi}(1+(2R)^{-2}) \int_{B_{2R}} Q_{-\omega} \tilde L_U[\Psi^*] \cdot Z_{1j},
\quad j=1,2.
\end{align}
We reformulate \eqref{1.3bb}-\eqref{1.4bb} as the fixed point problem
\begin{align}
\label{fixedPar}
[p,\xi] = \mathcal A[p,\xi]  \inn \mathcal B
\end{align}
where the space $\mathcal B$ will be introduced below and
the operator  $\mathcal A = [\mathcal A_1 , \mathcal A_2]$ is defined by
\\
\begin{align*}
\mathcal A_1[p,\xi]
&= \mathcal P  \left[ a_0^{(0)}[p,\xi,\Psi^*(p,\xi,\Phi(p,\xi,Z_0^*),Z_0^*)]\right]
\\
\mathcal A_2 [p,\xi]
&=
q -  \int_t^T b[p,\xi] (s) \,ds
\end{align*}
with
\begin{align*}
b_{1j}[p,\xi](t)&=
 \frac{1}{4\pi}(1+(2R)^{-2}) \int_{B_{2R}} Q_{-\omega} \tilde L_U[
 \Psi^*(p,\xi,\Phi(p,\xi,Z_0^*),Z_0^*)
] \cdot Z_{1j}  .
\end{align*}
To define $\mathcal B$ consider the closed ball
\begin{align*}
\mathcal B_1 &= \overline B_{l_1}(\kappa_0 ) \times \overline B_{l_2}(0) \subset X_1,
\end{align*}
where  $\kappa_0 =  \div z_0^{*0}(q) + i \curl z_0^{*0}(q)  $
with $z_0^{*0}$ so that
\begin{align*}
Z_0^{*0}(x) =  \left [ \begin{matrix}z_0^{*0}(x) \\  z_{03}^{*0} (x) \end{matrix}   \right ] , \quad z_0^{*0}(x) = z^{*0}_{01}(x)  + i z^{*0}_{02}(x)  ,
\end{align*}
and $Z_0^* = Z_0^{*0} + Z_0^{*1}$ is the initial condition as described in \eqref{condZ0}. Here
the numbers $l_1$, $l_2$ are given by
\begin{align*}
l_1 = T^\sigma, \quad l_2 = C_0 |\log T|^{1-\sigma} \log^2( |\log T| ) ,
\end{align*}
with $\sigma>0$ small and  and $C_0>0$ is a fixed large constant.
We consider  $\xi$ in the space
\begin{align*}
X_2 &= \{  \xi \in C^{1}([0,T];\R^2) \ : \ \dot \xi(T) = 0\}
\end{align*}
endowed with the norm
\[
\| \xi \|_{X_2} = \|\xi\|_{L^\infty(0,T)} + \sup_{t\in (0,T)} \lambda_*(t)^{-\sigma} |\dot \xi(t)|
\]
where $\sigma \in (0,1)$ is fixed.
In $X_2$ we consider the closed ball
$
\mathcal B_2 := \overline B_{1}( \xi^* ) ,
$
where  $\xi^* \equiv q \in \Omega$.
We consider the Banach space
$
X := X_1 \times X_2
$
and its  closed ball
$
\mathcal B := \mathcal B_1 \times \mathcal B_2 .
$
We formulate the fixed point problem \eqref{fixedPar} in $\mathcal B$.
We claim that $\mathcal A ( \mathcal B) \subset  \mathcal B$ and that  $\mathcal A$ is a contraction mapping on $\mathcal B$ for
the norm $\| \ \|_X$.
This is consequence of the various bounds and Lipschitz estimates derived in \S \ref{secLambda} for the operator $\mathcal P$ and in \S \ref{sectInnerOuter} for the operators $\Psi^*$ and $\Phi$.

\end{proof}

\section{Stability of blow-up}
\label{Stability}

In this section we discuss the stability of the blow-up phenomenon  predicted in Theorem \ref{teo1} and prove
Theorem~\ref{stability}. We consider the class of initial conditions that lead to blow-up at a given point as described in \S \ref{asss}.
The solution has the form
$$
u(x,t)  =  U_{\la(t), \omega(t) , \xi(t)}   +   \vp +  a( |\vp|^2 ) U_{\la(t), \omega(t) , \xi(t)}
$$
where
$ a(s) = \sqrt{1-s} -1$
 and
$$
\vp(x,t)=  \Pi_{U_{\la(t), \omega(t) , \xi(t)}^\perp } \left [ \ttt Z^*(x,t)  +  \Phi(\la, \omega, \xi)(x,t) + \psi(x,t)  + \eta \phi(x,t)  \right ] ,
$$
where the point  $\xi(T)  \in \Omega$ is prescribed.
Changing slightly the proof we can achieve that the value $\xi(0)=q$ be prescribed. Let us denote
$
\varepsilon = \lambda(0).
$
A simple application of implicit function theorem to the system of equations determining $(\la,\omega, \xi)$  leads to the fact that
the blow-up time  $T$ and the final point $\xi(T)$ can be regarded as functions of arbitrary small values $\ve>0$ and points $q\in \Omega$.

\medskip
The functions $(\la, \omega, \xi)$ as well as $\psi$ and $\phi$   have Lipschitz dependence in $ p:= (\ve,q)$ and $Z^*$ in suitable topologies. We relabel
$$\omega (p) := \omega(0), \quad  U_p :=  U_{\ve, \omega(p), q}, \quad \ttt \Phi (p)(x) =   \Phi(\la, \omega, \xi)(x,0) + \psi(x,0)  $$
so that the initial condition of the solution above becomes
$$
u_0(p) =    U_{p}   +   \Pi_{ U_{p}^\perp }[ Z^* + \ttt \Phi (p) ]  + a ( |\Pi_{U_{p}^\perp }[ Z^* + \ttt \Phi (p) ]|^2 )  U_p.
$$
A generic initial condition close to
$$
  U_{p_0}   +   \Pi_{  U_{p_0}^\perp }[ Z^*_0 + \ttt \Phi (p_0) ]  + a( |\Pi_{ U_{p_0}^\perp }[ Z^*_0 + \ttt \Phi (p_0) ]   |^2)  U_{p_0}
$$
with values in $S^2$ can be written in the form
$$
v(x; \vp_1 ) :=   U_{p_0}  +   \Pi_{  U_{p_0}^\perp }[ Z^*_0 + \ttt \Phi (p_0) + \vp_1]  + a  (|\Pi_{ U_{p_0}^\perp }[ Z^*_0 + \ttt \Phi (p_0) + \vp_1]   |^2)  U_{p_0}
$$
where $\vp_1$ is a small function, otherwise arbitrary. We shall show that if $\vp_1$  is sufficiently small in $C^2$-topology and it lies on a certain codimension-1 manifold, then problem \equ{har map flow}
with initial condition $u_0 (x)= v(x ; \vp_1 ) $ has  blow-up as predicted. Thus what we need is that for
suitable $$\zeta = (\ve, q, Z^*) =  \zeta_0 + \zeta_1,\quad  \zeta_1= (\ve_1,q_1, Z_1^*) $$
we have   that
\be\label{eqb}
v(\cdot ; \vp_1 ) = u_0(p).
\ee
It is convenient to measure the size of $\zeta_1$ with respect to the norm (see \equ{normZ0}),
$$
\|p_1\|:=  |q_1|  +  |\ve_1|  +  \| Z_1^*\|_*.
$$
We expand $u_0(p)$ around $p=p_0$ and get
\begin{align*}
u_0(\zeta)=  &   U_{\zeta_0} +  \vp (\zeta) + a(|\vp(\zeta)|^2 ) )\,  U_{\zeta_0},
\end{align*}
where
\begin{align*}
\vp(\zeta) = &  \Pi_{  U_{\zeta_0}^\perp }[ Z^* + \ttt \Phi (\zeta)  +  ( U_{\zeta} - U_{\zeta_0}) (1 -\gamma(\zeta)  +   a(p)    ) ],\\ \gamma(\zeta) = & U_p\cdot ( Z^* + \ttt \Phi (\zeta)   ) \\
a(p)= & a(|\Pi_{  U_{\zeta}^\perp }[ Z^* + \ttt \Phi (\zeta) ] |^2   ).
\end{align*}
Therefore, equation \equ{eqb} becomes
$$
\Pi_{  U_{\zeta_0}^\perp }[ Z^*_0 + \ttt \Phi (\zeta_0) + \vp_1] =   \Pi_{  U_{\zeta_0}^\perp }[
Z^* + \ttt \Phi (\zeta)  +  (U_{\zeta} -  U_{p_0}) (1 -\gamma  +   a    )  ]
$$
or, equivalently
$$
\Pi_{ U_{\zeta_0}^\perp }[Z^*_1 + \ttt \Phi (\zeta)- \ttt \Phi (\zeta_0) +
( U_{\zeta} - U_{\zeta_0}) (1 -\gamma (\zeta) +   a (\zeta)   ) -\vp_1 ] \ = \ 0.
$$
We will get a solution to  this equation if we find a constant $c_0$ such that
$$
Z^*_1 + \ttt \Phi (\zeta_0+ \zeta_1)- \ttt \Phi (\zeta_0) +
(U_{\zeta} -  U_{\zeta_0}) (1 -\gamma(\zeta)  +   a (\zeta)   ) = \vp_1 + c_0 U_{\zeta_0}
$$
Let us consider the functions $Z_{lj}(y)$ defined in \eqref{ZZ}, $l=0,1$, $j=1,2$, with $y= \frac{x-q}{\ve}$.
We introduce the following intermediate problem: we want to find a function $Z_1^*$ and five constants
$c_0$, $c_{lj}$ such that
\be
Z^*_1 + \ttt \Phi (\zeta_0+ p_1)- \ttt \Phi (\zeta_0) +
(U_{\zeta} -  U_{\zeta_0}) (1 -\gamma(\zeta)  +   a (\zeta)   ) = \vp_1 + c_0 U_{\zeta_0} +  c_{lj} Z_{lj}
\label{eqc}\ee
 and the following  five real constraints  hold for the function $Z_1^*(x)$:
\be
\div z^*_1(\zeta_0) = 0, \quad   \curl  Z^*_1(q_0)=0 , \quad Z^*_1(q_0) = 0  .
\label{eqd}\ee
Summation convention is used in \equ{eqc}.

\medskip
To make the argument more transparent, we consider a simplified linearized version of \eqref{eqc}-\eqref{eqd}, in which lower order terms are neglected, and only the constants associated to mode 0 (associated to dilations and rotations) are considered.
Thus we consider the model equation for $Z_1^*$,
\be\label{eqZ}\left\{
\begin{aligned}
Z_1^* + \Phi_0[Z_1^*] = \varphi_1 + \sum_{j=1}^2 c_{0j} Z_{0j},\\
\div z_1^*(q_0,0) = 0, \quad \curl z_1^*(q_0,0) = 0.
\end{aligned}\right.
\ee
where
\[
\Phi_0[Z_1^*] (r) =
\left( \begin{matrix}
\phi_0[Z_1^*](r,t) \\ 0
\end{matrix}\right)
\]
with
\begin{align}
\label{phi0}
\phi_0[Z_1^*] (r) = r e^{i\theta}
\int_{-T}^{0}
\dot p(s)
k(r^2+\varepsilon^2 , -s)
\,ds , \quad k( \zeta ,t)= 2\frac{1-e^{-\frac{ \zeta }{4 t}}}{\zeta} ,
\end{align}
where
$p(t) = \lambda(t) e^{i \omega(t)}$,
$r = |x-q_0|$, $\varepsilon = \lambda(0)$,
and $p=p[Z_1^*]$ is such that the following equation is satisfied
\begin{align}
\label{eqP}
\left\{
\begin{aligned}
\dot p(t) |\log(T-t)| + \int_{-T}^t \frac{\dot p(s)}{T-s}\,ds
& =
\div \tilde z_1(q,t) + i \curl \tilde z_1(q,t) ,
\quad t\in [0,T].
\\
p(T) &=0 ,
\end{aligned}
\right.
\end{align}
where
\begin{align}
\label{heatZ1-new}
\left\{
\begin{aligned}
& \partial_t \tilde Z_1 (x,t) = \Delta \tilde Z_1(x,t) \quad \text{in } \Omega\times (0,T)
\\
&
\tilde Z_1(x,0) = Z_1^*(x) \quad x\in\Omega
\\
& \tilde Z_1 (x,t) = 0 \quad (x,t)\in \partial \Omega\times (0,T) ,
\end{aligned}
\right.
\end{align}
and we use the notation
\[
\tilde Z_1 = \left( \begin{matrix}
\tilde z_1 \\  \tilde z_{1,3}
\end{matrix} \right) ,
\quad
Z_1^* = \left( \begin{matrix}
z_1^* \\  z_{1,3}^*
\end{matrix} \right) .
\]


%

The main result here is the solvability of \eqref{eqZ}.
\begin{prop}
\label{propZ1}
Assume $\| \varphi_1 \|_*$ is finite. Then for $T>0$ sufficiently small equation \eqref{eqZ} has a unique solution $Z_1^*$, $c_{01}$, $c_{0,2}$ and moreover
\[
\| Z_1^* \|_* + |c_{01}| + |c_{02}| \leq C \| \varphi_1 \|_*  .
\]
\end{prop}
We can obtain a similar result if all constraints and constants are considered, with essentially the same proof as that below. On the other hand, to derive the corresponding result to the full problem \equ{eqc}-\equ{eqd}, we need to use the linearized
version and contraction mapping principle. For that  we need to use the  precise  Lipschitz estimates of the solution of the inner-outer gluing system on the parameters involved as done in \S  \ref{sectInnerOuter} and \S \ref{secLambda}. 
The $C^1$ character of the manifold predicted in Theorem~\ref{stability} follows from the fixed point characterization and the implicit function theorem.

\medskip
We devote the rest of this section to the proof of the proposition, whose
main step is the following estimate.
\begin{lemma}
\label{lemmaD1}
Assume that
\begin{align}
\nonumber
\div z_1^*(q_0) = 0, \quad  \curl z_1^*(q_0) = 0.
\end{align}
Then
\[
\| \Phi_0[Z_1^*] \|_* \leq \frac{C}{|\log T|} \|Z_1^* \|_*.
\]
\end{lemma}


%
%
%
%

To prove this we need a corollary  of Lemma~\ref{lemmaGradEst-b} adapted to the norm $\| \ \|_*$ defined in \eqref{normZ0} is the following.
\begin{lemma}
\label{lemmaGradEst}
Suppose $Z_1^*\in C^2(\overline \Omega)$ satisfies
\begin{align*}
|\nabla_x Z_1^* (x) | & \leq |\log \varepsilon| , \quad x \in \Omega
\\
|D^2_x Z_1^*(x)| &  \leq \frac{|\log \varepsilon|^{ \frac{1}{2}}}{|x-q_0|+\varepsilon }\quad x \in \Omega.
\end{align*}
Then the solution $\tilde Z_1$ of \eqref{heatZ1-new} satisfies
\begin{align}
\label{estGrad}
| \nabla_x \tilde Z_1(x,t) |
\leq   |\log \varepsilon| , \quad t\geq 0,
\end{align}
and
\[
|\nabla_x \tilde Z_1(x,t) - \nabla_x \tilde Z_1(x,T)| \leq C
\begin{cases}
|\log \varepsilon| & \text{if }0\leq t \leq \varepsilon^2 \\
|\log \varepsilon|^{ \frac{1}{2} }
\frac{T-t}{T} (1+\log(\frac{T}{t})\, )& \text{if }  \varepsilon^2 \leq t \leq T.
\end{cases}
\]
\end{lemma}
\begin{proof}
As in Lemma~\ref{lemmaGradEst-b}
we consider the function given by Duhamel's formula in $\R^2$ and then  decompose the solution as a sum of the one in $\R^2$ and a smooth one in $\Omega$ with zero initial condition.

From \eqref{formulaGrad-b} and $| \nabla_x Z_1^*(x) |\leq |\log \varepsilon| $ we get \eqref{estGrad}.

For $0\leq t \leq \varepsilon^2$ we get
\[
|\nabla_x \tilde Z_1(x,t) - \nabla_x \tilde Z_1(x,T)| \leq C
|\log \varepsilon|
\]
from  \eqref{estGrad}.
For $ \varepsilon^2 \leq t \leq T$ from Lemma~\ref{lemmaGradEst-b} we obtain
\begin{align*}
|\nabla_x \tilde Z_1(x,t)  - \nabla_x \tilde Z_1(x,T)  |
& \leq C  |\log\varepsilon|^{1/2}\frac{\sqrt T - \sqrt t}{\sqrt T }
\left( 1+ \log\Bigl(\frac{T}{t} \Bigr) \right).
\end{align*}

\end{proof}

\begin{proof}[Proof of Lemma~\ref{lemmaD1}]
Let $f(t) = \div \tilde z_1(q,t) + i \curl \tilde z_1(q,t) $.
Differentiating \eqref{eqP}  we find
\[
\frac{d}{dt}\left( \dot p(t) |\log(T-t)|^2\right) = |\log(T-t)| \dot f(t).
\]
This can be integrated explicitly and we get
\[
\dot p(t) = -\frac{1}{|\log(T-t)|^2}\int_t^T |\log(T-s)| \dot f(s) \,ds + \frac{c |\log T|}{|\log(T-t)|^2}
\]
for some constant $c$ to be determined.
Integrating by parts we find that
\[
\dot p(t) = \frac{f(t)-f(T)}{|\log(T-t)|}
+\frac{1}{|\log(T-t)|^2} \int_t^T \frac{f(s)-f(T)}{T-s}\,ds + \frac{c |\log T|}{|\log(T-t)|^2}.
\]
This function is defined for $t\in [0,T]$ and we need to extend it to $[-T,T]$ to make sense of \eqref{eqP}. A possible extension is $\dot p(t) = \dot p(0)$ for $t \in [-T,0]$ but this makes this lemma too simple and not useful to adapt to the real situation. For this reason we make the analysis with the following extension.
Define
\begin{align}
\label{defP1}
\dot p_1(t) = \frac{f(t)-f(T)}{|\log(T-t)|}
+\frac{1}{|\log(T-t)|^2} \int_t^T \frac{f(s)-f(T)}{T-s}\,ds
\end{align}
so that
\[
\dot p(t) = \dot p_1(t) + \frac{c |\log T|}{|\log(T-t)|^2} \quad \text{for } t\in [0,T]
\]
Then define
\begin{align}
\label{extension}
\dot p(t) = \dot p_1(0) +  \frac{c |\log T|}{|\log(T-t)|^2}, \quad t\in [-T,0].
\end{align}

%
%

\medskip

We want to estimate
\[
\phi_0[Z_1^*] (r) = r e^{i\theta}
\int_{-T}^{0}
\dot p(s)
k(r^2+\varepsilon^2 , -s)
\,ds , \quad k( \zeta ,t)= 2\frac{1-e^{-\frac{\zeta}{4 t}}}{\zeta} ,
\]
which, thanks to \eqref{extension} depends only on $\dot p_1(0)$ and $c$. Therefore we need to estimate these quantities.
We claim that
\begin{align}
\label{eq1d}
\dot p_1(0)
&= \frac{f(0)-f(T) + O(\frac{\log(|\log T|)}{|\log T|^{1/2}  }) \|Z_1^*\|_*}{|\log T|}
(1+O(\frac{1}{|\log T|}))
\\
\label{eq2d}
c &= f(T)(1+O(\frac{1}{|\log T|}) ) + O( \frac{1}{|\log T|}) f(0)+ O(\frac{\log(|\log T|)}{|\log T|^{1/2} }) \|Z_1^*\|_* .
\end{align}

To obtain these estimates we note that evaluating equation \eqref{eqP} at $t=0$ we get
\begin{align}
\label{int0}
\dot p_1(0)(  |\log T|  + \log 2)
+ c (1 + O(\frac{1}{|\log T|}))
= f(0)
\end{align}
and evaluating equation \eqref{eqP} at $t=T$ we get
\begin{align}
\label{intTT}
\int_{-T}^T \frac{\dot p_1(s)}{T-s}\,ds
+c (1+O(\frac{1}{|\log T|})) = f(T).
\end{align}

Thus we need to estimate $\int_{-T}^T \frac{\dot p_1(s)}{T-s}\,ds$
where $p_1$ is given  \eqref{defP1}.
We have
\begin{align*}
\int_{-T}^T \frac{\dot p_1(s)}{T-s}\,ds
&= \int_{-T}^0  \frac{\dot p_1(s)}{T-s}\,ds
+ \int_0^T  \frac{\dot p_1(s)}{T-s}\,ds
\\
&= \dot p_1(0)  \log 2 + \int_0^T  \frac{\dot p_1(s)}{T-s}\,ds .
\end{align*}

To estimate $ \int_{0}^T \frac{\dot p_1(s)}{T-s}\,ds$  we write
\[
\dot p_1 = \dot p_{1a} + \dot p_{1b}
\]
with
\begin{align*}
\dot p_{1a}(t)  & =
\frac{f(t)-f(T)}{|\log(T-t)|}
\\
\dot p_{1b}(t)  &= \frac{1}{|\log(T-t)|^2} \int_t^T \frac{f(s)-f(T)}{T-s}\,ds.
\end{align*}
We compute
\begin{align*}
\int_{0}^T \frac{\dot p_{1a}(s)}{T-s} \,ds
=  \int_{0}^{\frac{T}{|\log T|}}...
 +\int_{\frac{T}{|\log T|}}^T... .
\end{align*}
By Lemma~\ref{lemmaGradEst} we have that
\begin{align}
\label{boundsTildeF}
| f(t) - f(T) |
\leq C  \| Z_1^* \|_*
\left\{
\begin{aligned}
&  \log( |\log T|) |\log T|^{1/2}
\frac{T-t}{T} , &&  \frac{T}{|\log T|} \leq t \leq T
\\
&  |\log T|, &&  0 \leq  t \leq \frac{T}{|\log T|}  .
\end{aligned}
\right.
\end{align}
which in particular implies
\begin{align*}
|\dot p_{1a}(t) |
\leq  C \frac{ \| Z_1^* \|_* }{|\log (T-t)|}
\left\{
\begin{aligned}
&  \log( |\log T|) |\log T|^{1/2}
\frac{T-t}{T} , &&  \frac{T}{|\log T|} \leq t \leq T
\\
&  |\log T|, &&  0 \leq  t \leq \frac{T}{|\log T|}  .
\end{aligned}
\right.
\end{align*}
Therefore
\[
\int_{0}^{\frac{T}{|\log T|}}
\frac{|\dot p_{1a}(s)|}{T-s} \,ds
\leq \frac{C}{|\log T|} \| Z_1^* \|_*,
\]
and
\begin{align*}
&
\int_{\frac{T}{|\log T|}}^T
\frac{ | \dot p_{1a}(s)  | }{T-s} \,ds
\leq C  \frac{\log( |\log T|)   |\log T|^{1/2}  }{|\log T| }\| Z_1^* \|_* .
\end{align*}
It follows that
\begin{align}
\label{intTT1}
\int_{0}^T
\frac{ | \dot p_{1a}(s) | }{T-s} \,ds
\leq  C \frac{\log( |\log T| )     }{|\log T|^{1/2}}  \| Z_1^* \|_* .
\end{align}

By \eqref{boundsTildeF}, we find that
\[
|\dot p_{1b}(t)|\leq
C  \| Z_1^* \|_*
\left\{
\begin{aligned}
& \frac{\log( |\log T|)   |\log T|^{1/2}   }{ |\log(T-t)|^2} \frac{T-t}{T}, && \frac{T}{|\log T|}\leq t \leq T
\\
&  \frac{\log( |\log T|)}{|\log T|^{3/2}}, && 0 \leq t \leq  \frac{T}{|\log T|}.
\end{aligned}
\right.
\]
This implies that
\begin{align}
\label{intTT2}
\int_{0}^T \frac{| \dot p_{1b}(s)|}{T-s}\,ds  \leq \frac{\log( |\log T|)}{|\log T|^{3/2}} \| Z_1^* \|_*.
\end{align}

From \eqref{intTT1} and \eqref{intTT2} we find that
\[
\left|
 \int_0^T  \frac{\dot p_1(s)}{T-s}\,ds
\right|
\leq \frac{\log(|\log T|)}{|\log T|^{1/2}} \|Z_1^*\|_*.
\]
Therefore \eqref{intTT} gives
\begin{align}
\label{int000}
\dot p_1(0) \log 2 + O(\frac{\log(|\log T|)}{|\log T|^{1/2}}) \|Z_1^*\|_*
+ c (1 + O(\frac{1}{|\log T|}))
= f(T).
\end{align}
Equations \eqref{int0} and \eqref{int000} form a system
\begin{align*}
\left[
\begin{matrix}
 |\log T|  + \log 2 & 1 + O(\frac{1}{|\log T|}) \\
 \log 2 & 1 + O(\frac{1}{|\log T|})
\end{matrix}
\right]
\left[
\begin{matrix}
\dot p_1(0) \\ c
\end{matrix}
\right]
=
\left[
\begin{matrix}
f(0) \\ f(T) + O(\frac{\log(|\log T|)}{|\log T|^{1/2} }) \|Z_1^*\|_*
\end{matrix}
\right]
\end{align*}
for  $\dot p_1(0)$ and $c$, and solving we get
\eqref{eq1d}, \eqref{eq2d}.

\medskip

We use \eqref{eq1d}, \eqref{eq2d} to estimate $\phi_0$ given by \eqref{phi0}:
and obtain
\[
\| \Phi_0[Z_1^*] \|_* \leq \frac{C}{|\log T|} \|Z_1^* \|_* .
\]

\end{proof}

\begin{proof}[Proof of Proposition~\ref{propZ1}]
We look for a solution of \eqref{eqZ} in the space of functions
\[
\mathcal Z = \{ Z_1^* \in C^2(\overline\Omega) : \|Z_1^*\|_* < \infty, \
\div z_1^*(q_0) = 0, \  \curl z_1^*(q_0) = 0 \} .
\]
To determinte $c_{0j}$ we apply divergence and curl \eqref{eqZ} at $q_0$ to obtain
\begin{align*}
c_{01} = \varepsilon \left( \div \phi_0[Z_1^*] (q_0,0) -   \div \varphi_1(q_0) \right), \quad
c_{02} = \varepsilon \left( \curl \phi_0[Z_1^*] (q_0,0) -   \curl \varphi_1(q_0) \right)  .
\end{align*}
With this equation \eqref{eqZ} becomes the fixed point problem
\begin{align}
\label{fixedZ1}
Z_1^* = \mathcal F[Z_1^*] + \varphi_1
+  \div \varphi_1(q_0) \varepsilon  Z_{01}
+  \curl \varphi_1(q_0)    \varepsilon Z_{02}.
\end{align}
where
\begin{align*}
\mathcal F[Z_1^*]
&=
- \Phi_0[Z_1^*]
-  \div \phi_0[Z_1^*] (q_0,0)  \varepsilon Z_{01}
- \curl \phi_0[Z_1^*] (q_0,0)  \varepsilon Z_{02}
\end{align*}

By Lemma~\ref{lemmaD1} we get
\begin{align*}
\left|\
\div \phi_0[Z_1^*] (q_0,0) + i \curl \phi_0[Z_1^*] (q_0,0)
\right|
\leq C |\log \varepsilon|  \| \Phi_0[Z_1^*] \|_*
\leq C \|Z_1^*\|_*.
\end{align*}
But
\[
\| \varepsilon Z_{0j } \|_* \leq \frac{C}{|\log T|^{1/2}} .
\]
This and Lemma~\ref{lemmaD1} shows that
\[
\| \mathcal F[Z_1^*] \|_*  \leq \frac{C}{|\log T|^{1/2}} \|Z_1^*\|_*.
\]
By the contraction mapping principle, equation \eqref{fixedZ1} has a unique fixed point in $\mathcal Z$.
\end{proof}


%
%

\section{Reverse bubbling}\label{reverse}

The proofs of Theorems \ref{teo2} and \ref{teo3}  follow very similar lines to those of Theorem \ref{teo1}, with a ``backwards'' construction.  In Theorem \ref{teo2}
we consider the exact ansatz as in \equ{upa}  for $u(x,t)$ in $(0, 2T)$, extended for $T<t<2T$ in the form
\begin{align*}
u(x,t) =  \bar U +  \Pi_{\bar U^\perp} [ \Phi^0 + \Psi^* + \eta Q_\omega \phi ] + a(\Pi_{\bar U^\perp} [ \Phi^0 + \Psi^* + \eta Q_\omega \phi ] ) \bar U
\end{align*}
where $\la(t)$ is defined in the interval $(-T,2T) $ and satisfies $\la(T)=0$, while
\begin{align*}
\bar U (x,t) = Q_{\omega(t)} (t) \bar W (\frac{x-\xi(t)}{\lambda(t)})
\end{align*}
and  $\bar W(y)$ is the reverse bubble as in \equ{barW}.
A main point is that the linear theory
for the inner problem, corresponding to $\phi(y,t)$  has to be performed for $t>T$ for ``ancient solutions'', which exactly mirror the forward theory of Section \ref{sectLinearTheory}. More precisely, we need to consider a problem of the form
 We consider the linear equation
\begin{align*}
\la^2 \pp_t  \phi   &=     L_{ W  }  [\phi]       + h(y,t)\inn  \ttt\DD_{2R}
\\
 \phi(\cdot,0 ) &=   0 \inn B_{2R( 0)}
\nonumber
\\
\nonumber
\phi \cdot  W  &=  0 \inn   \DD_{2R}
\end{align*}
where
\[
\ttt\DD_{2R} =  \{ (y,t) \ /\  t \in (T, 2T) ,\  y \in  B_{2R(t)}(0) \} .
\]
We assume that $h(y,t)$ is defined for all $(y,t)\in \R^2 \times (0,T)$ and   satisfies
\[
 h \cdot  W  =   0  , \quad
|h(y,t)| \leq C \frac{\ttt \lambda_*(t)^\nu}{(1+|y|)^a} ,
\]
where we extend the definition of $\la_*(t)$ for $t>T$ as  $\la_*(t) = \frac {|\log T| |t-T|)}{\log^2 |t-T|}$. Inverses for the linear problem with right bounds (which vanish as $t\downarrow T$) are found as before.

\medskip
In the full construction a major ingredient is the adjustment of the parameter $\la(t)$ for times $t>T$.
The main term in the error (the one due to the effect of dilations) has now the extended form
\begin{align*}
\begin{cases}
\frac{\dot\lambda}{\lambda} \rho w_\rho
\sim -2\frac{\dot\lambda}{r} & \text{if } 0\leq t < T \\
\frac{\dot\lambda}{\lambda} \rho \tilde w_\rho
\sim 2\frac{\dot\lambda}{r}
& \text{if } T < t \leq 2 T .
\end{cases}
\end{align*}
Therefore we extend $\Phi^0$ by considering the function  $\Phi^0[\omega, b \lambda,\xi]$, where
$$
b(s) = \begin{cases}
1 & s\leq T\\
-1 & s>T.
\end{cases}
$$
$\Psi^*(x,t)$ has $Z^*(x,t)$ as its main term.
Testing the error as before by the generator of dilations, we get the approximate equation
\be\label{jj}
\int_0^{t-\lambda(t)^2} b(s)\frac{\dot\lambda(s)}{t-s}\,ds =
-|\, [ \div \Psi^*+  i\curl \Psi^*](q,t )|.
\ee
We would like to find a solution such that $\la(T)=0$,
$
\dot\lambda(t)<0 \text{ if } t<T,\
\dot\lambda(t)>0 \text{ if } t>T .
$
The solution for $t <T$ is the one of the forward bubbling of Theorem \ref{teo1}, which we recall is given at main order by
\begin{align*}
\lambda(t) = \kappa_* \frac{|\log T|(T-t)}{|\log(T-t)|^2}, \quad
t<T,
\end{align*}
For $t>T$ the approximate equation reads
\begin{align*}
\int_{-T}^{t-\lambda^2} b(s)\frac{\dot\lambda(s)}{t-s}\,ds
&=
-|[ \div \psi^*+  i\curl \psi^*](q,T )|
+ \int_0^T  \dot\lambda(s) \left( \frac{1}{t-s}-\frac{1}{T-s} \right)\,ds
- \int_T^{t-\lambda^2} \frac{\dot\lambda(s)}{t-s}\,ds . \end{align*}
Equation \equ{jj} then approximately reads
\begin{align*}
\int_T^{t-\lambda^2} \frac{\dot\lambda(s)}{t-s}\,ds
=  \int_0^T  \dot\lambda(s) \left( \frac{1}{t-s}-\frac{1}{T-s} \right)\,ds \quad \text{for } t\geq T,
\end{align*}
The integral in the left hand side is approximately
$
\dot\lambda(t) |\log(t-T)|
$,
while
\begin{align*}
\int_0^T  \dot\lambda(s) \left( \frac{1}{t-s}-\frac{1}{T-s} \right)\,ds
&=
(t-T) \kappa_*|\log T|
\int_0^T   \frac{1}{(t-s)(T-s)|\log(T-s)|^2 }\,ds .
\end{align*}
Arguing as before we get
\begin{align*}
\int_0^T   \frac{1}{(t-s)(T-s)|\log(T-s)|^2 }\,ds
=\frac{1}{t-T}\frac{1}{|\log(t-T)|}
+ O(\frac{1}{|\log(t-T)|^2}) .
\end{align*}
Hence, for $t>T$ we get the apprximate equation
$$
\dot\lambda(t)|\log(t-T)| = \kappa_* \frac{1}{|\log(t-T)|}
$$
which gives
$$
\lambda(t) = \kappa_* \frac{(t-T)|\log T|}{|\log(t-T)|^2} ,\quad \text{for }t>T
$$
as desired.
This computation can be made fully rigorous with the same arguments already employed in the forward bubbling construction, leading to the proof of Theorem \ref{teo2}.

\medskip
For the proof of Theorem \ref{teo3} we proceed in exactly the same way however, now with an ansatz that does not include a bubble for $t<T$. In that case the approximate equation for $\lambda$ takes the form
\begin{align*}
\int_{T}^{t-\lambda^2} \frac{\dot \la(s)}{t-s}\,ds
&=
- | [\div \psi^*+  i\curl \psi^*](q,T )|
\end{align*}
for $t>T$.
From this equation we get
\begin{align*}
\lambda(t) = \kappa_* \frac{T-t}{|\log (t-T)|} ,
\end{align*}
as desired.

\bigskip
It is interesting to notice that this continuation, even at the level of the parameter $\la(t)$, does not seem to exhibit analyticity near $t=T$, even one-sided, in terms of $(T-t)$ or the natural parameter $s= \frac1{ \log (T-t)}$. It is not hard to check for instance that, even though formal improvements of approximation in powers of $s$ are possible for $\la(t)$, they so not lead to a power series with positive convergence radius.

\bigskip


%

\appendix

\section{The heat equation with right hand side}
\label{lips1}

We are going to measure the solution to \eqref{heat-eq0}
in the norm $\| \ \|_{\sharp,\Theta,\gamma}$, (c.f. \eqref{normPsi})
with $\Theta$ and $\beta$ (recall that $R= \lambda_*^{-\beta}$) satisfying:
\begin{align}
\label{assumpPar1}
\beta \in \Bigl( 0,\frac{1}{2} \Bigr)  ,
\quad
\Theta \in
(0,\beta)
\end{align}
Our main result in this section is the following, where we use the norm $ \| \ \|_{\sharp, \Theta,\gamma} $ defined in \eqref{normPsi}.
\begin{prop}
\label{prop3}
Assume  \eqref{assumpPar1}.
For $T$, $\varepsilon>0$ small there is a linear operator that maps a function $f:\Omega \times (0,T) \to \R^3$ with  $\|f\|_{**}<\infty$ into $\psi$, $c_1,c_2,c_3$ so that \eqref{heat-eq0} is satisfied.
Moreover the following estimate holds
\begin{align}
\label{estPsi}
\| \psi\|_{\sharp, \Theta ,\gamma}
+ \frac{\lambda_*(0)^{-\Theta}  ( \lambda_*(0) R(0) )^{-1} }{ |\log T| }( |c_1| +|c_2| +|c_3| )  \leq C \|f\|_{**} ,
\end{align}
where $\gamma\in(0,\frac{1}{2})$.
\end{prop}

\begin{remark}
The condition $\beta \in (0,\frac{1}{2})$ is a basic assumption to have the singularity appear inside the self-similar region.
%
%
The condition $\Theta>0$ is needed for Lemma~\ref{lemma-heat1}.
The assumption $\Theta <  \beta $ is so that the estimates provided by Lemma~\ref{lemmaHeat6} are stronger than the ones of Lemma~\ref{lemma-heat1}.
\end{remark}

To prove Proposition~\ref{prop3} we consider
\begin{align}
\label{heatEqOmega}
\left\{
\begin{aligned}
\psi_t &= \Delta \psi + f \quad \text{in }\Omega \times (0,T)
\\
\psi(x,0) &= 0 , \quad x \in \Omega
\\
\psi(x,t) &= 0 ,\quad x \in \partial\Omega, t \in  (0,T) ,
\end{aligned}
\right.
\end{align}
and let $q$ be a point in $\Omega$.

We always assume that $R$ is given by \eqref{RRR}.

\begin{lemma}
\label{lemma-heat1}
Assume $\beta\in (0,\frac{1}{2})$ and $\Theta>0 $.
Let $\psi$ solve \eqref{heatEqOmega} with $f$ such that
\[
|f(x,t)|\leq \lambda_*(t)^{\Theta} (\lambda_* (t) R(t))^{-1}
\chi_{ \{  |x-q| \leq 3 \lambda_*(t) R(t) \}} .
\]
Then
\begin{align}
\label{heat1-size}
|\psi(x,t)| & \leq  C \lambda_*(0)^\Theta \lambda_*(0) R(0) |\log T|,
\\
\label{heat1-mod-cont}
|\psi(x,t)-\psi(x,T)| & \leq  \lambda_*(t)^\Theta  \lambda_*(t) R(t)|\log(T-t)|,
\\
\label{heat1-size-grad}
| \nabla \psi(x,t)|  &  \leq C \lambda_*(0)^\Theta,
\\
\label{heat1-mod-contA-grad}
| \nabla \psi(x,t) - \psi(x,T)|  &  \leq C \lambda_*(t)^\Theta,
\end{align}
and for any $\gamma \in (0,\frac{1}{2})$,
\begin{align}
\label{heat1GradHolderT}
\frac{|\nabla \psi(x,t)-\nabla\psi(x,t')|}{ |t-t'|^{\gamma}}
\leq
C
\frac{ \lambda_*(t)^\Theta }{  ( \lambda_*(t) R(t) )^{2 \gamma} }
\end{align}
for any $x,$ and $0\leq t'\leq t\leq T$ such that $t-t'\leq\frac{1}{10}(T-t)$,
and
\begin{align}
\label{heat1GradHolderX}
\frac{|\nabla \psi(x,t)-\nabla\psi(x',t')|}{ |x-x'|^{2\gamma}}
\leq
C
\frac{ \lambda_*(t)^\Theta }{  ( \lambda_*(t) R(t) )^{2\gamma} }
\end{align}
for any $|x - x'|\leq 2 \lambda_*(t) R(t)$ and $0\leq t\leq T$.
\end{lemma}
The proof is in section~\ref{appendixHeatEqEst1}.

\begin{lemma}
\label{lemmaHeat6}
Assume $\beta \in (0,\frac{1}{2})$  and $ m \in (\frac{1}{2},1)$.
Let $\psi$ solve \eqref{heatEqOmega} with $f$ such that
\[
|f(x,t)|\leq \frac{\lambda_*(t)^m}{|z-q|^2}  \chi_{ \{  |x-q| \geq  \lambda_*(t) R(t) \}}  .
\]
Then
\begin{align}
\nonumber
|\psi(x,t)|
& \leq C T^m |\log T|^{2-m} ,
\\
\nonumber
|\psi(x,t) - \psi(x,T)|
& \leq
C
|\log T|^m (T-t)^m |\log(T-t)|^{2-2m} ,
\\
\nonumber
|\nabla\psi(x,t)|
& \leq
C \frac{ T^{m-1}  |\log T|^{2-m} }{ R( T )}  ,
\\
\nonumber
|\nabla \psi(x,t)-\nabla\psi(x,T)|
&
\leq
C \frac{\lambda_*(t)^{m-1} |\log(T-t)|}{ R(t) }
\end{align}
and for any $\gamma \in (0,\frac{1}{2})$:
\begin{align*}
\frac{|\nabla \psi(x,t)-\nabla\psi(x',t')|}{(|x-x'|^2 + |t-t'|)^{\gamma}}
\leq C
\frac{1}{(\lambda_*(t) R(t) )^{2\gamma} }
\frac{\lambda_*(t)^{m-1} |\log(T-t)|}{ R(t) }
\end{align*}
for any  $|x - x'|\leq 2 \lambda_*(t) R(t)$ and $0\leq t'\leq t\leq T$ such that $t-t'\leq\frac{1}{10}(T-t)$.
\end{lemma}
The proof is in section~\ref{appendixHeatEqEst1}.


\begin{lemma}	
\label{lemma-heat4}
Let $\psi$ solve \eqref{heatEqOmega} with $f$ such that
\[
|f(x,t)|\leq 1 ,
\]
Then
\begin{align}
\nonumber
|\psi(x,t) |
&\leq  C t ,
\\
\nonumber
|\psi(x,t)-\psi(x,T)|
& \leq C (T-t) |\log(T-t)|,
\\
\nonumber
|\nabla \psi (x,t) | & \leq T^{1/2}
\end{align}
\begin{align}
\nonumber
|\nabla \psi(x,t) - \nabla \psi(x,T) |
\leq C  (T-t)^{1/2}
\end{align}
\begin{align}
\nonumber
|\nabla \psi(x,t_2) - \nabla \psi(x,t_1) |
\leq C  |t_2-t_1|^{1/2} .
\end{align}
\begin{align*}
|\nabla \psi(x_1,t) - \nabla \psi(x_2,t) |
\leq C |x_1-x_2| |\log(|x_1-x_2|) .
\end{align*}
\end{lemma}	
The proof is in section~\ref{appendixHeatEqEst1}.


\begin{proof}[Proof of Proposition~\ref{prop3}]
Let $\psi_0[f]$ denote the solution of \eqref{heatEqOmega} where $f$ satisfies  $\|f\|_{**}<\infty$.

We claim that  $\|\psi_0[f]\|_*\leq C \|f\|_{**}$.
Indeed, given  $f$ with $\| f\|_{**}<\infty$ we decompose $f = \sum_{i=1}^3 f_i$ with $|f_i|\leq C\| f \|_{**} \varrho_i$. By linearity it is sufficient to prove that when $f$ is each of the $\varrho_i$, the corresponding $\psi$ has finite $\| \ \|_{**}$ norm.

The case $f = \varrho_1$ is direct from Lemma~\ref{lemma-heat1}.
Using the hypothesis $\Theta<\beta$ we can find $\sigma_0$ small so that the case
$f = \varrho_2$ follows from Lemma~\ref{lemmaHeat6}.
The case $f = \varrho_3$ follows from Lemma~\ref{lemma-heat4}.

Finally,  let us show that in problem \eqref{heat-eq0} we can choose $c_i$ so that  that $\psi(q,T)=0$.
To do this we let $\psi_i$ the solution
\begin{align*}
\left\{
\begin{aligned}
\partial_t \psi_i  & = \Delta_x \psi_i  \inn \Omega \times (0,T) \\
\psi_i  & = 0 \onn \pp \Omega \times (0,T) \\
\ \psi_i(x,0)  & =   \mathbf{e_i}   \eta_1    \inn  \Omega
\end{aligned}
\right.
\end{align*}
Let
\[
\psi  = \psi_0 + \sum_{i=1}^3 c_i \psi_i .
\]
Then for $T>0$ small there is unique choice of $c_i$ such that $\psi(q,T)=0$. Moreover $|c_i|\leq C \lambda_*(0)^\nu R(0)^{2-a} |\log T| \|f\|_{**}$ and hence $\psi$ satisfies \eqref{estPsi}.
\end{proof}

\subsection{Proof of Lemmas~\ref{lemma-heat1}, \ref{lemmaHeat6}, and \ref{lemma-heat4}}
\label{appendixHeatEqEst1}
The proof of the estimates is done by analyzing the solution $\psi$ of
\begin{align}
\label{heatEqR2}
\left\{
\begin{aligned}
\partial_t \psi_0 & = \Delta \psi_0  + f \quad \text{in }\R^2\times (0,T) ,
\\
\psi_0(x,0) &=0 \quad x\in \R^2 ,
\end{aligned}
\right.
\end{align}
defined by Duhamel's formula
\[
\psi_0(x,t ) = \int_0^t \int_{\R^2} \frac{e^{-\frac{|x-y|^2}{4(t-s)}}}{4\pi (t-s)} f(y,s) \, d y ds
\]
assuming
\begin{align}
\nonumber
|f(x,t) | \leq  \chi_{\{ |y| \leq 2 \lambda_*(s) R(s)  \} } \lambda_*(t)^{\nu-2} R(t)^{-a} .
\end{align}
The solution to \eqref{heatEqOmega} is then given by $ \psi = \psi_0 + \psi_1$ where $\psi_1$ solves the homogeneous heat equation in $\Omega\times (0,T)$ with boundary condition given by $-\psi_0$. In the sequel we prove that the estimates \eqref{heat1-size}--\eqref{heat1-mod-cont} are valid for $\psi_0$. Then the conclusion for $\psi_1$ follows from standard parabolic estimates.
In what follows we denote by $\psi$ the solution to \eqref{heatEqR2} given by Duhamel's formula.

\begin{proof}[Proof of  \eqref{heat1-size}]
We have, using the heat kernel,
\begin{align*}
\psi(x,t)
&= C
\int_0^t
\frac{\lambda_*(s)^{\nu-2} R(s)^{-a}}{t-s}
\int_{|y|\leq 2 \lambda_*(s) R(s)}
e^{-\frac{|x-y|^2}{4(t-s)}}
\,dyds
\\
&=
C
\int_0^t
\lambda_*(s)^{\nu-2} R(s)^{-a}
\int_{|z|\leq 2 \lambda_*(s) R(s) (t-s)^{-1/2}}
e^{-|\tilde x-z|^2}
\,dzds
\end{align*}
where $\tilde x = x (t-s)^{-1/2}$.
First we estimate
\begin{align}
\nonumber
&
\int_{0}^{t-(T-t)}
\lambda_*(s)^{\nu-2} R(s)^{-a}
\int_{|z|\leq 2 \lambda_*(s) R(s) (t-s)^{-1/2}}
e^{-|\tilde x-z|^2}
\,dzds
\\
& \leq
C
\int_0^{t-(T-t)}
\frac{\lambda_*(s)^{\nu} R(s)^{2-a}}{t-s}
ds
\label{e003}
\leq
C \lambda_*(0)^{\nu} R(0)^{2-a} .
\end{align}

Consider the integrals $\int_{t-(T-t)}^{t-\lambda_*(t)^2}$ and  $\int_{t-\lambda_*(t)^2}^t$.
We have
\begin{align}
\nonumber
&
\int_{t-(T-t)}^{t-\lambda_*(t)^2}
\lambda_*(s)^{\nu-2} R(s)^{-a}
\int_{|z|\leq 2 \lambda_*(s) R(s)(t-s)^{-1/2}}
e^{-|\tilde x - z|^2/4 }
\,dzds
\\
\label{e001}
& \leq
C
\lambda_*(t)^{\nu}  R(t)^{2-a} |\log(T-t)| .
\end{align}
For the second part we have
\begin{align}
\nonumber
&
\int_{t-\lambda_*(t)^2}^t
\lambda_*(s)^{\nu-2} R(s)^{-a}
\int_{|z|\leq 2 \lambda_*(s) R(s)(t-s)^{-1/2}}
e^{-|\tilde x - z|^2/4 }
\,dzds
\\
\label{e002}
&
\leq C
\int_{t-\lambda_*(t)^2}^t
\lambda_*(s)^{\nu-2} R(s)^{-a}
\,ds
\leq
C \lambda_*(t)^{\nu}  R(t)^{-a} .
\end{align}

From  \eqref{e003}, \eqref{e001}, \eqref{e002},we deduce
\[
|\psi(x,t)|\leq C \lambda_*(0)^{\nu} R(0) |\log T|.
\]
which is the desired estimate. Estimates \eqref{heat1-size}, \eqref{heat1-mod-cont}, \eqref{heat1-size-grad}, \eqref{heat1-mod-contA-grad}, \eqref{heat1GradHolderT}, \eqref{heat1GradHolderX} follow in similar manner.

\end{proof}

The proofs of Lemmas \ref{lemmaHeat6} and \ref{lemma-heat4} follow similar lines to those above, and we omit them.

\section{The heat equation with initial condition}

In this section we consider the heat equation
\begin{align}
\label{heatZ1}
\left\{
\begin{aligned}
& \partial_t \tilde Z_1 (x,t) = \Delta \tilde Z_1(x,t) \quad \text{in } \Omega\times (0,T)
\\
&
\tilde Z_1(x,0) = Z_1^*(x) \quad x\in\Omega
\\
& \tilde Z_1 (x,t) = 0 \quad (x,t)\in \partial \Omega\times (0,T) ,
\end{aligned}
\right.
\end{align}
and derive estimates assuming, roughly speaking, that $Z_1^*$ behaves like $(r+\varepsilon) |\log(r+\varepsilon)|$.

\begin{lemma}
\label{lemmaGradEst-b}
Suppose $Z_1^*\in C^2(\overline \Omega)$ satisfies
\begin{align*}
|D^2_x Z_1^*(x)| &  \leq \frac{1}{|x-q_0|+\varepsilon }\quad x \in \Omega.
\end{align*}
Then the solution $\tilde Z_1$ of \eqref{heatZ1} satisfies
\begin{align}
\nonumber
|\nabla_x \tilde Z_1(x,t) - \nabla_x \tilde Z_1(x,T)| \leq
C \frac{T-t}{T} \Bigl(1+\log(\frac{T}{t})\, \Bigr) \quad
\text{if }  \varepsilon^2 \leq t \leq T.
\end{align}
\end{lemma}
\begin{proof}
We do the computation when $\Omega$ is $\R^2$ and we deal with the solution given by Duhamel's formula.
The general case follows by the decomposing the solution as a sum of the one in $\R^2$ and a smooth one in $\Omega$.
Then
\begin{align}
\label{formulaGrad-b}
\nabla_x \tilde Z_1(x,t)
&= \frac{1}{4\pi t }\int_{\R^2} e^{-\frac{|y|^2}{4 t}} \nabla_x Z_1^*(x-y)\,dy .
\end{align}
Assume $\varepsilon^2\leq t \leq T $.
Then, using \eqref{formulaGrad-b}, we have
\begin{align*}
& |\nabla_x \tilde Z_1(0,t)  - \nabla_x \tilde Z_1(0,T)  | \\
&= \frac{1}{4\pi}
\left|
\int_{\R^2}
e^{-\frac{|y|^2}{4}}
\int_0^1
\nabla_x \tilde Z_1^*\bigl(-s \sqrt T y  + (1-s) \sqrt t y \bigr)
\,ds dy
\right|
\\
&\leq C
\int_{\R^2} e^{-\frac{|y|^2}{4 }}
\int_0^1
|D^2 Z_1^*\bigl(-s \sqrt T y  + (1-s) \sqrt t y \bigr)   |
( \sqrt T - \sqrt t)|y|
\, ds
dy
\\
&\leq C  \int_{\R^2} e^{-\frac{|y|^2}{4 }}
\int_0^1
\frac{( \sqrt T - \sqrt t)|y| }{s (\sqrt T- \sqrt t) |y|+ \sqrt t   |y| + \varepsilon }
\, ds  dy \\
& \leq
C \int_{\R^2} e^{-\frac{|y|^2}{4 }}
\log\Bigl( \frac{\sqrt T |y|+\varepsilon}{\sqrt t |y|  + \varepsilon} \Bigr) \,dy
\\
&=C
\int_0^\infty e^{-\frac{\rho^2}{4}}\log\Bigl( \frac{\sqrt T \rho+\varepsilon}{\sqrt t \rho  + \varepsilon} \Bigr)\rho \,d\rho\\
&=
C (\sqrt T - \sqrt t) \varepsilon \int_0^\infty e^{-\frac{\rho^2}{4}}
\frac{1}{(\sqrt T  \rho + \varepsilon)(\sqrt t \rho + \varepsilon)}\,d\rho .
\end{align*}
Using this representation, after some computation the desired result follows.

\end{proof}

Similar computations leads us to the following estimates.
\begin{lemma}
Suppose $Z_1^*\in C^2(\overline \Omega)$ satisfies
\begin{align*}
|D^2_x Z_1^*(x)| &  \leq \frac{1}{|x-q_0|+\varepsilon }\quad x \in \Omega.
\end{align*}
Then the solution $\tilde Z_1$ of \eqref{heatZ1} satisfies
\begin{align}
\nonumber
|D^2_x \tilde Z_1(x,t)  \leq
\frac{C}{\varepsilon + \sqrt t}
\end{align}
\end{lemma}

\begin{lemma}
Suppose $Z_1^*\in C^2(\overline \Omega)$ satisfies
\begin{align*}
|D^2_x Z_1^*(x)| &  \leq \frac{1}{|x-q_0|+\varepsilon }\quad x \in \Omega.
\end{align*}
Then the solution $\tilde Z_1$ of \eqref{heatZ1} satisfies for $0\leq t_0 \leq t_1$:
\begin{align*}
|\nabla_x \tilde Z_1(x,t_1)  - \nabla_x \tilde Z_1(x,t_0)  |
& \leq
C
\begin{cases}
\frac{\sqrt{t_1}-\sqrt{t_0}}{\sqrt{t_1}} \log(2\frac{t_1}{t_0})
& \text{if } t_0 \geq  \varepsilon^2
\\
\frac{\sqrt{t_1}-\sqrt{t_0}}{\sqrt{t_1}}  \log(2\frac{t_1}{\varepsilon^2})
&  \text{if } t_0 \leq  \varepsilon^2, \ t_1 \geq  \varepsilon^2
\\
\frac{\sqrt{t_1}-\sqrt{t_0}}{\varepsilon}
& \text{if } t_1 \leq  \varepsilon^2
\end{cases}
\end{align*}
\end{lemma}

Let us recall the norm $\| \ \|_*$ defined in \eqref{normZ0}.
As a corollary of the previous estimates we have.

\begin{lemma}
Suppose $Z_0^*\in C^2(\overline \Omega)$.
Then the solution $\tilde Z^*$ of \eqref{heatZ1} satisfies
\begin{align*}
| \nabla_x \tilde Z^*(x,t) |
\leq   |\log \varepsilon|  \, \| Z_0^*\|_*
, \quad t \geq 0,
\end{align*}
\[
|Z^*(x,t)-Z^*(x,T)|\leq C |\log T| \frac{T-t}{\sqrt T} \|Z_0^*\|_*,
\]
\[
|\nabla_x \tilde Z_1(x,t) - \nabla_x \tilde Z_1(x,T)| \leq C
\| Z_0^*\|_*
\begin{cases}
|\log \varepsilon| & \text{if }0\leq t \leq \varepsilon^2 \\
|\log \varepsilon|^{1/2}
\frac{T-t}{T} (1+\log(\frac{T}{t})\, )& \text{if }  \varepsilon^2 \leq t \leq T.
\end{cases}
\]
\end{lemma}

\section{Derivatives for the exterior problem}
\label{derExterior}

\begin{corollary}
\label{coroLipOuter1}
Let  $ \Psi (p,\xi,\Phi,Z_0^*)$ be the solution to equation \eqref{eq-psi}  constructed in Proposition \ref{propi1}.
Let  $p_l,\xi$ satisfy \eqref{cotin}, \eqref{cotin1} and $p_l= \la e^{i\omega_l}$, $\|\Phi_l\|_E \le 1$,   and $\| Z_{0l}^* \|_*<\infty$, $l=1,2$.
Then
\begin{align*}
& \|  \Psi (p_1, \xi, \Phi_1 ,Z_{01}^* )
- \Psi (p_2, \xi, \Phi_2 ,Z_{02}^* )\|_{\sharp, \Theta,\gamma}
\\
& \quad \le  C T^\sigma
(  \|\Phi_1-\Phi_2\|_E
+  \|\lambda_* (\dot \omega_1 - \dot \omega_2)\|_\infty
+ \| Z_{01}^* - Z_{02}^*\|_*
) .
\end{align*}
\end{corollary}

Corollary~\ref{coroLipOuter1} gives a partial Lipschitz property of the exterior solution $ \Psi (p,\xi,\phi)$ of \eqref{eq-psi} with respect to $p$, namely it only considers variations of $p  = \lambda e^{i \omega}$ with respect to $\omega$.
We will need Lipschitz estimates for variations of $p  = \lambda e^{i \omega}$ in  $\lambda$ and also variations with respect to $\xi$.
These estimates are obtained for  $ \Psi (p,\xi,\phi)$ when considered as a function of the inner variable $(y,t)\in \DD_{2R}$.

For this let us introduce some notation.
Suppose that $\psi(x,t)$ is defined in $\Omega\times (0,T)$.
We let
\[
\tilde \psi  (y,t) =  \psi  (\xi(t) + \lambda(t) y ,t) , \quad
(y,t) \in \DD_{2R} .
\]
The following expression is $\| \psi \|_{\sharp,\Theta,\gamma}$ expressed in terms of $\tilde \psi$ (and restricted to $\DD_{2R}$):
\begin{align*}
\|  \tilde \psi \|_{\sharp\sharp, \Theta,\gamma}
&:=
\lambda_*(0)^{-\Theta}
\frac{1}{|\log T|  \lambda_*(0) R(0) }\|\tilde\psi\|_{L^\infty(\DD_{2R} )}
+ \lambda_*(0)^{-\Theta-1} \|\nabla_y \psi\|_{L^\infty(\DD_{2R} )}
\\
\nonumber
&\quad
+
\sup_{ \DD_{2R} }   \lambda_*(t)^{-\Theta-1} R(t)^{-1}
\frac{1}{|\log(T-t)|} |\tilde \psi(y,t)-\tilde \psi(y,T)|
\\
& \quad + \sup_{(y,t)\in \DD_{2R}}
 \, \lambda_*(t)^{-\Theta-1}
|\nabla_y  \tilde \psi(y,t)-\nabla_y  \tilde \psi(y,T) |
\\
\nonumber
& \quad
+ \sup_{(y,t) , (y',t)\in \DD_{2R}}
\lambda_*(t)^{-\Theta-1}
R(t)^{2\gamma}
\frac{|\nabla_y  \tilde \psi(y,t) -\nabla_y \tilde \psi(y',t) |}{  |y-y'|^{2\gamma   }}
\\
\nonumber
& \quad
+ \sup_{}
\lambda_*(t)^{-\Theta-1}
(\lambda_*(t) R(t))^{2\gamma}  \frac {|\nabla_y  \tilde \psi(y,t) -\nabla_y \tilde \psi(x',t') |}{  |t-t'|^{\gamma   }} ,
\end{align*}
where the last supremum is taken in the region
\[
(y,t) , (y,t') ,  \in \DD_{2R} , \quad |t-t'|\leq \frac{1}{10}(T-t) .
\]

\begin{corollary}
\label{coroLipOuter2}
Let
$ \Psi (p,\xi,\phi)$ be the solution to equation \eqref{eq-psi}  in Proposition \ref{propi1}.
Let  $p_l = \la_l e^{i\omega}$, $\xi_l$ satisfy  \eqref{cotin}, \eqref{cotin1} and $\|\phi\|_{*,a,\nu}\le 1$. Then for $\tilde \Theta \in (0,\Theta)$ we have
\begin{align*}
& \| \tilde \Psi(p_1,\xi_1,\phi) - \tilde \Psi(p_2,\xi_2,\phi) \|_{\sharp\sharp,\tilde \Theta,\gamma}
\\
& \quad \leq C
\Bigl[ \,
\Bigl\| \frac{\lambda_1-\lambda_2}{\lambda_*} \Bigr\|_{L^\infty}
+
\| \dot\lambda_1 - \dot\lambda_2\|_{L^\infty}
+
\Bigl\| \frac{\xi_1-\xi_2}{\lambda_* R} \Bigr\|_{L^\infty}
+
\Bigl\| \frac{ \dot\xi_1 - \dot\xi_2}{R} \Bigr\|_{L^\infty}
\Bigr] .
\end{align*}
\end{corollary}

Let $f(y,t)$ be a function satisfying
\[
|f(y,t)|\leq \lambda_*(t)^{\nu} R(t)^{-a} \chi_{B_R(t)}  ,
\]
and let  $   \psi[ \lambda,\xi ] $ be the solution of
\[
\left\{
\begin{aligned}
\psi_t &= \Delta_x \psi + \frac{1}{\lambda(t)^2}	f(\frac{x-\xi}{\lambda},t)  \quad \text{in } \R^2\times (0,T)
\\
\psi(x,0)&=0 \quad x\in \R^2 ,
\end{aligned}
\right.
\]
given by Duhamel's formula.

Let
\[
\tilde \psi [ \lambda,\xi ] (y,t) = \psi[\lambda,\xi](\xi(t) + \lambda(t) y ,t).
\]

We consider the directional derivative with respect to $\lambda$ of $\tilde \psi$ in the direction of $\lambda_1$, defined by
\[
D_\lambda \tilde \psi[\lambda,\xi][\lambda_1]
= \lim_{s\to 0} \frac{1}{s}\left(
\tilde \psi[\lambda+s\lambda_1,\xi]-\tilde \psi[\lambda,\xi]
\right)
\]
and also the directional derivative with respect to $\xi$ of $\tilde \psi$ in the direction of $\xi_1$, defined by
\[
D_\xi  \tilde \psi[\lambda,\xi][\xi_1]
= \lim_{s\to 0} \frac{1}{s}\left(
\tilde \psi[\lambda,\xi+ s\xi_1]-\tilde \psi[\lambda,\xi]
\right) .
\]

\subsection{Derivative with respect to \texorpdfstring{$\lambda$}{}}

The proofs of the estimates below are based on Duhamel's formula for the solution:
\[
\psi(x,t) = \int_0^t \int_{\R^2} \frac{\exp(-\frac{|x-x'|^2}{4(t-s)})}{t-s}
\frac{1}{\lambda(s)^2}
f(\frac{x'-\xi(s)}{\lambda(s)},s)\,dx' d s .
\]
We change variables writing $x= \xi(t) + \lambda(t) y$ and $x'= \xi(s) + \lambda(s) y'$. Then
\[
\tilde \psi(y,t)
=
\int_0^t
\int_{\R^2} \frac{\exp(-\frac{|\xi(t) - \xi(s) +  \lambda(t) y  - \lambda(s) y'|^2}{4(t-s)})}{t-s}
f(y',s)\,dy' d s  ,
\]
and we obtain the following formula for the directional derivative:
$$\begin{aligned}
&D_\lambda \tilde \psi[\lambda_1]
(y,t)
= \\ &
-
\frac{1}{2}
\int_0^t
\int_{\R^2}
\frac{\exp ( -\frac{|\xi(t) - \xi(s) +  \lambda(t) y  - \lambda(s) y'|^2}{4(t-s)} ) }{(t-s)^2}
\cdot
( \xi(t) - \xi(s) +  \lambda(t) y  - \lambda(s) y')\cdot (\lambda_1(t) y - \lambda_1(s)y' )
f(y',s)\,dy' d s  .
\end{aligned}
$$

Lengthy but direct computations show the validity of the following estimates.
\begin{lemma}
We have
\begin{align*}
|D_\lambda \tilde \psi[\lambda,\xi](\lambda_1)(y,t)  |
\leq C
 \Bigl(
\Big\| \frac{\lambda_1}{\lambda} \Big\|_{L^\infty} + \| \dot\lambda_1\|_{L^\infty}
\Bigr) \lambda_*(0)^{\nu} R(0)^{2-a},
\end{align*}
and
\begin{align*}
 |D_\lambda \tilde \psi[\lambda,\xi](\lambda_1)(y,t)
-
D_\lambda \tilde \psi[\lambda,\xi](\lambda_1)(y,T) |
\leq C
 \Bigl(
\Big\| \frac{\lambda_1}{\lambda} \Big\|_{L^\infty} + \| \dot\lambda_1\|_{L^\infty}
\Bigr) \lambda_*(t)^{\nu} R(t)^{2-a},
\end{align*}
for $ |y|\leq R(t) $, $ t\in (0,T) $. On the other hand,
for any $\sigma>0$, $\gamma\in (0,\frac 12 )$ there is a $C$   such that
\begin{align*}
|\nabla_x D_\lambda \tilde \psi[\lambda,\xi](\lambda_1)(y,t)
-
\nabla_x D_\lambda \tilde \psi[\lambda,\xi](\lambda_1)(y,T) |
 \leq C
 \Bigl(
\Big\| \frac{\lambda_1}{\lambda} \Big\|_{L^\infty} + \| \dot\lambda_1\|_{L^\infty}
\Bigr) \lambda_*(t)^{\nu-1-\sigma} R(t)^{1-a},
\end{align*}
\begin{align}
\nonumber
 |\nabla_x D_\lambda \tilde \psi[\lambda,\xi](\lambda_1)(y_1,t)
-
\nabla_x D_\lambda \tilde \psi[\lambda,\xi](\lambda_1)(y_2,t) |
 \leq C
\Bigl( \frac{|y_1-y_2|}{R(t)} \Bigr)^\gamma
 \Bigl(
\Big\| \frac{\lambda_1}{\lambda} \Big\|_{L^\infty} + \| \dot\lambda_1\|_{L^\infty}
\Bigr) \lambda_*(t)^{\nu-1-\sigma} R(t)^{1-a},
\end{align}
\begin{align}
\nonumber
 |\nabla_x D_\lambda \tilde \psi[\lambda,\xi](\lambda_1)(y,t_2)
-
\nabla_x D_\lambda \tilde \psi[\lambda,\xi](\lambda_1)(y,t_1) |
\nonumber
\leq C
\frac{(t_2-t_1)^{\gamma}}{(\lambda_*(t_2) R(t_2) )^{2\gamma}}
\left(
\Bigl\|
\frac{\lambda_1}{\lambda}
\Bigr\|_{L^\infty}
+ \| \dot \lambda_1\|_{L^\infty}
\right)
\lambda_*(t_2)^{\nu-1-\sigma} R(t_2)^{1-a}
\end{align}
for $t_1$, $t_2$ in $[0,T]$ with $0 \leq t_2 - t_1 \leq \frac{1}{10}(T-t_2)$
and  $ |y|\leq R(t_2) $.
\end{lemma}

We can derive a similar expression  for the {derivative with respect to \texorpdfstring{$\xi$}{}}
and obtain the following estimates.

\begin{lemma}
Assume that$
|\dot \xi(t)|  \leq C$, $
|\dot \lambda(t)|  \leq C $, $
C_1  \lambda_*(t) \leq \lambda(t) \leq C_2 \lambda_*(t) ,$  $\text{in } (0,T)$
and let $
R(t) = \lambda_*(t)^{-\beta} $, $\beta < \frac{1}{2}$, $C, C_1,C_2>0$.
Then  there is $C$ such that
\begin{align}
\nonumber
 |\nabla_x D_\xi \tilde \psi[\lambda,\xi](\xi_1) (y,t)
-
\nabla_x D_\xi \tilde \psi[\lambda,\xi](\xi_1) (y,T) |
\leq C
\Bigl(
\Bigl\| \frac{\xi_1(\cdot)-\xi_1(T)}{\lambda R} \Big\|_{L^\infty}
+ \Bigl\| \frac{\dot\xi_1}{R} \Big\|_{L^\infty}
\Bigr)
 \lambda_*(t)^{\nu-1} R(t)^{1-a},
\end{align}
for $ |y|\leq R(t) $, $ t\in (0,T) $.
\end{lemma}

\bigskip\noindent
{\bf Acknowledgements:}
The  research  of J.~Wei is partially supported by NSERC of Canada. J.~D\'avila has been supported  by grants  Fondecyt  1130360 and
 PAI AFB-170001, Chile. M. del Pino
 M.~del Pino has been supported by a UK Royal Society Research Professorship and Grant  PAI AFB-170001, Chile.


\end{document}